\documentclass[12pt,a4paper]{amsart}
\makeatletter
\renewcommand\normalsize{%
    \@setfontsize\normalsize{11.7}{14pt plus .3pt minus .3pt}%
    \abovedisplayskip 10\p@ \@plus4\p@ \@minus4\p@
    \abovedisplayshortskip 6\p@ \@plus2\p@
    \belowdisplayshortskip 6\p@ \@plus2\p@
    \belowdisplayskip \abovedisplayskip}
\renewcommand\small{%
    \@setfontsize\small{9.5}{12\p@ plus .2\p@ minus .2\p@}%
    \abovedisplayskip 8.5\p@ \@plus4\p@ \@minus1\p@
    \belowdisplayskip \abovedisplayskip
    \abovedisplayshortskip \abovedisplayskip
    \belowdisplayshortskip \abovedisplayskip}
\renewcommand\footnotesize{%
    \@setfontsize\footnotesize{8.5}{9.25\p@ plus .1pt minus .1pt}
    \abovedisplayskip 6\p@ \@plus4\p@ \@minus1\p@
    \belowdisplayskip \abovedisplayskip
    \abovedisplayshortskip \abovedisplayskip
    \belowdisplayshortskip \abovedisplayskip}
\ifdefined\pdfpagewidth
\else
\fi
\calclayout
\makeatother

\usepackage{url}
\usepackage{amsthm}
\usepackage{amsmath}
\usepackage{amssymb}
\usepackage{mathtools}
\usepackage{amsfonts}
\usepackage{enumitem}
\usepackage{mathscinet}
\usepackage{lipsum}
\usepackage[english]{babel}
\usepackage{bbold}
\usepackage{soul}
\usepackage[utf8]{inputenc}
\usepackage[autostyle]{csquotes}
\usepackage{hyperref}

\newtheorem{theoremm}{Theorem}

\hypersetup{
    colorlinks=true,
    linkcolor=red,
    citecolor=blue,
    }

\usepackage[english]{babel}

\usepackage{color}
\usepackage{comment}

\newtheorem{thm}{Theorem}[section]
\newtheorem{definition}[thm]{Definition}

\newtheorem{theorem}[thm]{Theorem}
\newtheorem{cor}[thm]{Corollary}
\newtheorem{claim}[thm]{Claim}
\newtheorem{lemma}[thm]{Lemma}
\newtheorem{prop}[thm]{Proposition}

\newtheorem*{mainthm*}{Main Theorem}
\newtheorem*{theorem*}{Theorem}

\makeatletter
\def\moverlay{\mathpalette\mov@rlay}
\def\mov@rlay#1#2{\leavevmode\vtop{%
   \baselineskip\z@skip \lineskiplimit-\maxdimen
   \ialign{\hfil$\m@th#1##$\hfil\cr#2\crcr}}}
\newcommand{\charfusion}[3][\mathord]{
    #1{\ifx#1\mathop\vphantom{#2}\fi
        \mathpalette\mov@rlay{#2\cr#3}
      }
    \ifx#1\mathop\expandafter\displaylimits\fi}
\makeatother

\newcommand{\nocontentsline}[3]{}
\newcommand{\tocless}[2]{\bgroup\let\addcontentsline=\nocontentsline#1{#2}\egroup}

\usepackage{wasysym}

\def\Jac{\ensuremath{\mathrm{Jac}}}
\def\Vol{\ensuremath{\mathrm{Vol}}}

\begin{document}

\title{Generalized $u$-Gibbs measures for $C^\infty$ diffeomorphisms}
\author{S. Ben Ovadia, D. Burguet}
\date{}

\begin{abstract}
We show that for every $C^\infty$ diffeomorphism of a closed Riemannian manifold, if there exists a positive volume set of points which admit some expansion with a positive Lyapunov exponent (in a weak sense) then there exists an invariant probability measure with a disintegration by absolutely continuous conditionals on smoothly embedded disks subordinated to unstable leaves. As an  application, we prove a strong version of the Viana conjecture in any dimension. 

Our methods include developing a quantitative approach to high-dimensional Yomdin theory which allows to control the geometry of disks, and introducing a notion of ``measured disks" in order to provide a disintegration by absolutely continuous conditionals. In particular, we provide also a new proof for the case of surfaces (a previous result by the second author) proving directly the absolute continuity of conditionals rather than mere entropy estimates.
\end{abstract}

\newcommand{\Addresses}{{
  \bigskip
  \footnotesize

  S.~Ben Ovadia, \textsc{The Einstein Institute of Mathematics, The Hebrew University of Jerusalem, 91904 Jerusalem, Israel}.\\ \textit{E-mail address}: \texttt{Snir.BenOvadia@mail.huji.ac.il}
  
D.~Burguet, \textsc{LAMFA, CNRS - 
Universit\'{e} Picardie Jules Vernes, 80000 Amiens, France}.\\
\textit{E-mail address}: \texttt{David.Burguet@u-picardie.fr}
}}

\maketitle

\tableofcontents

\section{Introduction, main results, and key steps of proof}

\subsection{Introduction}\label{intro}

A dynamical system is a pair $(X,T)$ composed of a space $X$ endowed with a transformation $T:X\to X$. One of the key problems in Dynamical Systems is to study the statistics of a chaotic system at a steady state. This heuristic is interpreted as studying the collection of invariant probability measures which govern chaotic orbits (``chaos" can be interpreted as positive entropy, or exponential sensitivity to initial conditions in Smooth Dynamical Systems- that is positive Lyapunov exponents- or both). An invariant probability measure is a probability $p\in \mathbb{P}(X)$ s.t.  $p(E)=p(T^{-1}[E])$ for all $E\in \mathcal{B}(X)$, the Borel sigma algebra.

The first problem with which one is faced is then, which measure should we single out? Uniquely ergodic systems, which admit a unique invariant measure, often do not answer all criteria for chaos. When a system is not uniquely ergodic, there are natural candidates for reference measures in which we are interested.

In Smooth Dynamical Systems- that is differentiable maps of Riemannian manifolds- the most natural reference measure is the Riemannian volume of the manifold. This is due to several reasons: The first reason is the underlying physical assumption on the mathematical model, which asserts that when we carry out an experiment or a simulation of the system with a random initial condition, this initial condition is chosen randomly w.r.t. the Riemannian volume of the manifold (also sometimes called the Liouville measure).

The second reason which makes the Riemannian volume a natural candidate for a reference measure is the obvious relationship with the geometry of the space. 

The third reason is the fundamental Liouville theorem, which asserts that whenever the system admits no dissipation/heat loss/friction (i.e a closed Hamiltonian system), the Riemannian volume is an invariant measure. This fact is very useful, as invariant measures admit many statistical properties, which are studied in the field of Ergodic Theory.

However, it is natural to wish to extend the scope of the systems which we study beyond the non-dissipative systems. What happens in a system which has friction? Or for example the earth's atmosphere which constantly gets heat insertion from the sun? In those cases we do not expect the Riemannian volume to be an invariant measure, while we still wish to find natural reference measures for the system which are invariant (as we are interested in the unique statistical properties that invariant measures govern).

Uniformly Hyperbolic systems are differentiable maps which admit a continuous decomposition of the tangent bundle into two invariant sub-bundles, such that on one sub-bundle the differential uniformly contracts tangent vectors, and on the other the differential uniformly expands tangent vectors. For Uniformly Hyperbolic systems, the groundbreaking works of Sinai, Ruelle, and Bowen give an answer to the question in the paragraph above (see \cite{Si1,Ruelle76,B4,BowenRuelle}). They study a class of measures which are called {\em SRB measures} (named after them), which are invariant measures with compatibility with the Riemannian structure, even in cases where the Riemannian volume is not preserved. 

From now on, we only consider $C^{1+\epsilon}$, $\epsilon>0$, smooth diffeomorphisms. In this setting, SRB measures are invariant probability measures, which when disintegrated on unstable leaves\footnote{Unstable leaves are embedded disks which are tangent to the asymptotically expanding direction of the tangent space of points with positive Lyapunov exponents (i.e exponential sensitivity to initial conditions). Their existence is given in \cite{Pesin, RuelleFoliations}. } (in the sense of the Rokhlin disintegration theorem), admit conditional measures which are absolutely continuous w.r.t. the induced Riemannian volume on the unstable leaves. This produces the first instance of compatibility with the Riemannian volume.

The two more reasons which make SRB measures an object of importance are:

\begin{enumerate}\item {\em Physicality:} Every ergodic and hyperbolic SRB measure $\nu$ is physical \cite{EntropyFormulaSRB}: its basin $$\mathcal B(\nu):=\left\{x\in \mathbf M: \ \forall \phi: \mathbf{M}\to \mathbb R \text{ cont.}, \frac{1}{n}\sum\limits_{k=0}^{n-1}\phi(f^k(x))\xrightarrow[n\rightarrow\infty]{}\int \phi\, d\nu\right\}$$ has positive  Riemannian volume. This was first shown for Uniformly Hyperbolic diffeomorphisms by Ruelle \cite{Ruelle76}. Physicality does not imply the SRB property. For the gap between the two properties, see \cite{Tsuji}.
 \item {\em Entropic variational principle:} SRB measures satisfy the entropy formula: $h_\mu(f)=\int \sum_{i:\chi_i>0}\chi_i(x)\,d\mu(x)$
where $\chi_i(x)$ is the $i$-th Lyapunov exponent of $x$ (with multiplicity),\cite{Ruelle76,Pesin77,EntropyFormulaSRB, StrelcynEntropyFormulaSRB}; and the l.h.s. is strictly smaller than the r.h.s. for all other measures \cite{Ruelle76,MargulisRuelleIneq,LedrappierYoungI}.
\end{enumerate}
In particular, note that the physicality property admits an additional notion of compatibility with the Riemannian volume. For more properties of SRB measures, see \cite{YoungSRBsurvey}. 

\medskip
Unfortunatelly, SRB measures do not always exist, even for ``nice" systems (see \cite{YoungCounterExample} for example). This fact pushed forwards the field of Smooth Dynamical Systems for almost three decades, in trying to understand which systems admit SRB measures, or a corresponding object of interest.

For unimodal interval  maps with a negative Schwarzian derivative, Keller   showed that there is a positive Lebesgue measure of points with a positive Lyapunov exponent if and only if there exists an absolutely continuous invariant measure \cite{keller}.
 
In the non-uniformly hyperbolic setting, in her celebrated result \cite{young}, Young showed the existence of an SRB measure for non-uniformly hyperbolic maps with Young towers, subject to an assumption of integrability of the return-time to the base of the tower w.r.t. the Riemannian volume. In \cite{CDP}, the authors study the existence of an SRB measure through the positive volume of ``effectively hyperbolic" points. Another important approach and body of works is the study of existence via parameter families. For example, such as in the setting of quadratic families and H\' enon maps (see \cite{jak, BC,BY,berger}).

Indeed, most approaches to constructing SRB measures rely on a hyperbolic structure (which holds for surfaces for every measure with positive entropy, by the Ruelle inequality \cite{MargulisRuelleIneq}). Very few results address the construction of SRB measures in high dimension, in a general setup which allows $0$ Lyapunov exponents. In \cite{ExpConvSRB}, the authors construct SRB measures for a general diffeomorphism whose volume converges exponentially fast to a limit under the dynamics.

The {\em partially hyperbolic} setting is where the tangent space splits continuously everywhere into three invariant sub-bundles. An {\em unstable bundle $E^\mathrm{u}$} which expands uniformly, a {\em stable bundle $E^\mathrm{s}$} which contracts uniformly, and a {\em central bundle $E^\mathrm{c}$}, which is uniformly dominated by, and dominates, the unstable and stable bundles respectively. One often bunches together the unstable and central bundles or the stable and central bundles (denoted by $E^\mathrm{cu}$ or $E^\mathrm{cs}$ respectively), which uniformly dominate $E^s$ or dominated by $E^u$ respectively. In this setting, where $E^s$ is allowed to be trivial but not $E^{cs}$, in their celebrated work \cite{Pesin_Sinai_1982}, Pesin and Sinai construct {\em $u$-Gibbs measures}. That is, measures with an absolutely continuous disintegration on strong unstable leaves. For such partially hyperbolic systems which admit only negative Lyapunov exponents in the central bundle volume-a.e., these $u$-Gibbs measures are in fact SRB \cite{mostlycontracting} measures.

More in the partially hyperbolic setting, allowing $E^u$ to be trivial but not $E^{cu}$, in \cite{ABV,ADLP} the authors showed that volume-a.e. point which admits only positive Lyapunov exponents in the central bundle (for a certain notion of Lyapunov exponent, see the discussion below), lies in the basin of an ergodic hyperbolic SRB measure.

Then M. Viana  posed the following  conjecture in his famous  ICM talk. 

\medskip
\noindent\textbf{Conjecture} (Viana \cite{VianaConj})\textbf{.} {\em  If a smooth map has only non-zero Lyapunov exponents at Lebesgue almost every point then it admits some SRB measure.}
 
\medskip 

The importance of Viana's conjecture is two-fold. First, it is stated in a very elementary and natural way, and second the condition it suggests is important in terms of physical testability, as one would have to check random initial conditions and observe whether expansion happens or not.

Given a point, its {\em Lyapunov exponents} estimate exponential expansion / contraction rates associated to the action of the differential. These exponents are commonly used when $x$ is typical for an invariant measure or satisfy some {\em regularity}. 
By using Markov partitions, the first author \cite{BenOvadiaLeafCond}  proved a version  of Viana conjecture for $C^{1+\epsilon}$ diffeomorphisms of manifolds in any dimension   by assuming a positive volume for a collection of hyperbolic regular points (see also \cite{PesinVaughnLuzzatto} for a related result for surfaces, using Young towers and \cite{Burguetphysical} for another notion of regularity for $C^\infty$ systems).

In the present work we do not assume any kind of regularity
, and so Lyapunov exponents may be defined in different ways. For example, we may consider the {\em top-upper exponent} 
\begin{equation}\label{WeakExps}
\overline{\chi}(x):=\limsup_{n\to \infty}\frac{1}{n}\log \|d_x f^n\|,    
\end{equation}
or  the {\em top-lower exponent}  $\underline{\chi}(x):=\liminf_{n\to \infty}\frac{1}{n}\log \|d_x f^n\|$. In any case, any such notion should coincide with  the common definition of Lyapunov exponents for typical points with respect to invariant measures.  

\medskip 
Recently some progress has been made regarding the Viana Conjecture for $C^\infty$ diffeomorphisms of surfaces  (\cite{BCSVianaConj,BurguetViana}). In \cite{BurguetViana} the second author proved more precisely that for a $C^r$ diffeomorphism, $r>1$, Lebesgue a.e. $x$ with $\overline{\chi}(x)>\frac{\log \|d_\cdot f\|}{r}$ lies in the bassin of an ergodic hyperbolic  SRB measure (see also  \cite{delplanque}  for $C^r$ interval maps).   In these settings, the lower bound on $\overline{\chi}$ is sharp \cite{Burguetphysical, Burpre}. Morally, $C^r$ smoothness allows one to bound the distortion in the dynamics of disks, similar to the role of domination in partially hyperbolic systems. 

\medskip 
Before moving to the description of our main results (which indeed apply to manifolds in any dimension), we wish to describe the limits of the general Viana Conjecture in high dimension. Consider the famous ``Bowen's eye" example (see \cite{BowenEyeEx}), which admits a single hyperbolic fixed point $p$. 
It can always be made to be $C^\infty$ smooth. Consider the product dynamics of this map, times a hyperbolic linear toral automorphism $A$ of the 2-torus, where we assume w.l.o.g. that the exponents at the fixed point on Bowen's eye dominate 
the derivative bounds of the 
toral map. In this product dynamics which is $C^\infty$ smooth, every point admits a positive Lyapunov exponent, and there are no measures with $0$ exponents. However, the system admits no SRB measures. This can be seen by the fact that the Bowen's eye example admits only $4$ ergodic invariant measures, all Dirac delta measures- three are on fixed points which admit only positive Lyapunov exponents, and one on the hyperbolic fixed point. By making the exponents at the hyperbolic fixed point in the Bowen's eye dominate all exponents in the toral automorphism, we can guarantee that there is even no invariant measure with absolutely continuous conditionals on strong unstable leaves (i.e no $u$-Gibbs).
\medskip 

Therefore we can see that in dimension larger than two, expansion alone is not enough to conclude the existence of an SRB measure, nor a $u$-Gibbs measure. But still in the above example there is a measure with some absolutely continuous property, which is the product $\delta_p\times \mu_A^{SRB}$ of the Dirac measure $\delta_p$ at $p$ with the SRB measure $\mu_A^{SRB}$ of $A$ (just the Lebesgue measure on $\mathbb T^2$ in this case). In the present paper we prove a strong version of the Viana Conjecture for $C^\infty$ diffeomorphisms, by showing that volume almost every  point with non-zero Lyapunov exponents lies in the basin of a hyperbolic SRB measure (see \textsection \ref{SVCinMain}). This will be a consequence  of our main result which implies the existence of a measure with absolutely continuous properties assuming  only expansion at Lebesgue typical points. It addresses in particular the above example. More precisely, if the $k$-th exponent is positive on a set of positive volume, then there is an invariant  measure which disintegrates into smooth measures on $k$-disks which are contained in local unstable manifolds (see \textsection \ref{GuGibbsSection1}). 

In \textsection \ref{mainResultsSection} we describe our %
results and in \textsection \ref{keyStepsOfProof} we give a detailed description of the steps of our proof, which follows a geometric approach. However, allow us to give a brief overview of these steps first. 

\begin{enumerate}
    \item First we wish study the geometry of a $k$-disk which is pushed forwards by the $C^\infty$-dynamics, using Yomdin theory. This poses several new challenges, compared to previous works which study the geometry of curves, e.g \cite{BCSVianaConj,BurguetViana}.
    \item The refined Yomdin theory is then used to find points with a positive density of times with {\em bounded geometry}, which is adapted to the notion of expansion on disks.
    \item Finally, we construct an invariant measure using a notion of {\em Measured Disks} (morally similar to standard pairs \cite{pairs}). These disks do not come equipped with a compact space of densities, and so proving that the limiting measure is absolutely continuous requires a new approach of comparing measures via atoms of the {\em Yomdin partition}.
\end{enumerate}

\subsection{Main results}\label{mainResultsSection}
Let $\mathbf{M}$ be a closed Riemannian manifold of dimension $d\geq 2$, and let $f\in \mathrm{Diff}^\infty(\mathbf{M})$.  We  denote the Riemannian volume of $\mathbf M$ by $\Vol$. Below we extend the notion of a Lyapunov exponent in the weak sense from \eqref{WeakExps}.

\subsubsection{Generalized $u$-Gibbs measures}\label{GuGibbsSection1}
Let $\wedge^kT\mathbf M$ be the $k$-th exterior power bundle of the tangent space $T\mathbf M$ endowed with the Riemannian structure inherited from $\mathbf M$. We denote by $\wedge^kdf$ the map induced by $df$ on $T\mathbf M$. 

\begin{definition}[$k$-th exponent]\label{kthExp} For a point $x\in \mathbf{M}$, its {\em $k$-th exponent}, where $k\in\{1,\ldots,d\}$, is defined by $df$ on $\wedge^kT\mathbf M$.

    \noindent$\lambda_k(x):=\lim_{p\to\infty}\lambda_{k,p}(x)=\sup_{p\in\mathbb{N}}\lambda_{k,p}(x)$ and $$\lambda_{k,p}(x):=\limsup_{n\to\infty}\frac{1}{n}\left(\log \|\wedge^kd_xf^n\|-\frac{1}{p}\sum_{j=0}^{n-1}\log^+\max_{1\leq s\leq k-1}\|\wedge^{s}
	d_{f^j x} f^p\|\right).$$
    
\end{definition}

\medskip
\noindent\textbf{Remark:} \text{ }
\begin{enumerate}
\item $\lambda_k$ is defined for every point $x\in \mathbf M$, is invariant (i.e. $\lambda_k\circ f=\lambda_k$) and for every invariant ergodic probability measure $\mu$, $\int \lambda_k d\mu$ coincides with the $k$-th Lyapunov exponent of $\mu$ whenever it is non-negative. For a general formula, see the remark after Definition \ref{defOfLmabdaPK}.
\item    $\forall x\in \mathbf M, \ \lambda_1(x)=\overline{\chi}(x)=\limsup_{n\to\infty}\frac{1}{n}\log\|d_xf^n\|$ with the convention $\max \varnothing=0$.  
\item $\forall 1\leq k\leq d, \ \forall x\in \mathbf M, \  \lambda_k(x)\leq \limsup_n\frac{1}{n}\log \sigma_k(d_xf^n)\leq \lambda_1(x)$ with $\sigma_k(d_xf^n)$ being the $k$-th singular value of $d_xf^n$.   
    \item To see that the definition is proper and that the limit on $p$ exists, see Lemma \ref{limOnP}.
    \item $\lambda_d(x)\leq 0$ for $\mathrm{Vol}$-a.e. $x$. See for example the remark following Lemma \ref{limOnP}.
\end{enumerate}

For a smooth embedded disk $D$ we let $\Vol_D$ be the Riemannian probability volume on $D$. 

\begin{definition}[Generalized $u$-Gibbs measure]\label{GuGibbs}
Let $f$ be a $C^{1+\epsilon}$ diffeomorphism of $\mathbf{M}$, $\epsilon>0$. 
    Let $k\in\{1,\ldots,d-1\}$
. An $f$-invariant Borel probability measure $\widehat{\mu}$ is called a {\em generalized $u$-Gibbs measure  over 
    $k$-disks} (G-$u$-Gibbs for short), if  it can be written as
    $$\widehat{\mu}=\int \mu_\varpi d\widehat{\mathbf{p}}(\varpi),$$
where $\widehat{\mathbf{p}}$ is a probability on the space of $C^{1+\epsilon}$ embedded $k$-disks, and 
\begin{enumerate}
\item for $\widehat{\mathbf{p}}$-a.e. $\varpi$, $$\mu_\varpi\ll \mathrm{Vol}_{\mathrm{Im}(\varpi)},$$
    \item for $\widehat{\mathbf{p}}$-a.e. $\varpi$, for $\mu_\varpi$-a.e. $x$, $$\mathrm{Im}(\varpi)\subseteq V^u(x)$$
    where, $V^u(x):=\{y\in\mathbf{M}: \limsup\frac{1}{m}\log d(f^{-m}(x),f^{-m}(y))<0\}$ is the Pesin unstable manifold of $x$.
\end{enumerate}
\end{definition}

\medskip
\noindent\textbf{Remark:}\text{ }
\begin{enumerate}    
    \item Notice that almost every ergodic component of a G-$u$-Gibbs measure is also a G-$u$-Gibbs measure, as different ergodic components are carried by disjoint collections of unstable leaves (see Lemma \ref{ergGuGibbs}).
    \item In Corollary \ref{everyGuGibbsIsSmooth}, we in fact prove that for every G-$u$-Gibbs measure, the conditionals $\mu_\varpi$ are equivalent to $\mathrm{Vol}_{\mathrm{Im}(\varpi)}$, with an explicit density given the by the product of Jacobians.
    \item Note that every G-$u$-Gibbs measure given by a disintegration by $(k+1)$-disks, is in particular a G-$u$-Gibbs with a disintegration by $k$-disks.
    \item If $\widehat{\mu}$ is a G-$u$-Gibbs for $f^p$, then $\frac{1}{p}\sum_{i=0}^{p-1}\widehat{\mu}\circ f^{-i}$ is a G-$u$-Gibbs for $f$.
\end{enumerate}

\medskip
Our main theorem is as follows:

\begin{theoremm}[Main Theorem]\label{mainThm2026}
Let $f$ be a $C^\infty$ diffeomorphism of a closed Riemannian manifold. Then if $\Vol([\lambda_k>0])>0$, then there exists a generalized $u$-Gibbs measure over 
$k$-disks, $\widehat{\mu}$.
\end{theoremm}

\medskip
\noindent\textbf{Remark:}\text{ } 
\begin{enumerate}
\item In fact, we construct for every $r>0$ a G-$u$-Gibbs measure where the disks are of regularity $C^{r}$. 
    \item Later (see Proposition \ref{EisConstProp}), if $\Vol([\lambda_k>a>0])>0$, we construct a G-$u$-Gibbs measure which satisfies also, 
    $$\text{ for }\widehat{\mathbf{p}}\text{-a.e. }\varpi\text{, for }\mu_\varpi\text{-a.e. }x\text{, }\lim_{p\to \infty}\chi_{A_{k,p}}(x, T_x\varpi)>a,$$ where $A_{k,p}$ is a linear cocycle over $(\mathbf{M},f,\wedge^k T\mathbf{M})$ which induces $\lambda_{k,p}$ (see Definition \ref{linCocyc}). 
    \item Our main theorem is optimal in the sense that it applies to the example which is illustrated in the second-to-last paragraph of \textsection \ref{intro}. While that example does not admit an SRB measure, nor a $u$-Gibbs measure, it admits a G-$u$-Gibbs measure.
    \item The condition in the main theorem is sufficient, but not necessary. To see this, consider the skew-product on $\mathbb{S}^1\times \mathbb{T}^2$ given by $F(t,x)=(g(t),f_t(x))$, where $g:\mathbb{S}^1\to\mathbb{S}^1$ is a North-South dynamics, fixing $0$ and $1$, where $0$ is an indifferent repelling point ($0$ Lyapunov exponent) and $1$ is an attracting fixed point (negative exponent). Let $f_0$ be an Anosov map, and $f_t$ be a homotopy to $f_1$, which is a DA map with an attracting fixed point $p$ and a repeller (see for example \cite{SmaleDA}). Thus almost every point w.r.t. to the volume converges to $\delta_1\times\delta_p$, and hence has no positive exponents. However, the system admits the SRB measure $\delta_0\times \mu^{SRB}_{f_0}$.
    \item While the condition of the main the theorem is not necessary, the proof quickly reduces to finding a disk with a positive disk-volume of points which see some expansion tangent to the disk. Such a ``leaf condition" is clearly necessary, and through our construction also sufficient.
    \item $u$-Gibbs measures (and in particular G-$u$-Gibbs measures) do not have to be physical, as can be observed by $f=A\times \mathrm{Id}: \mathbb{T}^2\times \mathbb{T}^2\to \mathbb{T}^2\times \mathbb{T}^2$, where $A$ is a linear Anosov map. The diffeomorphism $f$ admits SRB measures, and satisfies the condition of our main theorem, but admits no physical measures.
    \item Our proof uses tools from Yomdin theory. Yomdin theory is sufficiently robust, as it applies to the dynamics of a disk under a sequence of maps, rather than a single map (with uniform regularity bounds). This implies potential applications to the setting of random dynamics.
\end{enumerate}

\medskip
\noindent\textbf{Conjecture} (G-$u$-Gibbs rigidity)\textbf{.} {\em Given an ergodic generalized $u$-Gibbs measure $\widehat{\mu}=\int \mu_\varpi d\widehat{\mathbf{p}}(\varpi)$, either $\widehat{\mu}$ is a $u$-Gibbs measure, or for some $k$, for $\widehat{\mathbf{p}}$-a.e. $\varpi$, for $\mu_\varpi$-a.e. $x$, $\mathrm{Im}(\varpi)$ is tangent to a sum of Oseledec directions of dimension $k$ at $x$, and the conditional measures on the stronger unstable leaves are atomic. }

\medskip
The idea is to study the distribution of the tangent spaces to $\varpi$ within the projective tangent space of a larger unstable leaf. The invariance of the measure should imply that either the distribution of directions is atomic (hence allowing an invariant integration, or the larger unstable leaf is absolutely continuous ``in the direction of $\varpi$" where these directions are well spread (hence implying that the conditionals on the large unstable lead should be absolutely continuous as well).\footnote{In \cite{BEFRH} the authors study a related notion of ``generalized $u$-Gibbs measures": if the unstable conditionals are invariant under the action of an adapted group (given by the normal forms), then either the conditionals are invariant under the action of a larger group, or the  QNI condition \cite{asaf}  is violated.}

The example which is illustrated in the second-to-last paragraph of \textsection \ref{intro} falls within the case where the conditionals on the stronger unstable leaves are atomic.

\medskip
\noindent\textbf{Remark:}
Given the  G-$u$-Gibbs measure $\widehat \mu=\int \mu_\varpi d\widehat{\mathbf{p}}(\varpi)$ which is constructed in Theorem  \ref{mainThm2026}, we describe the following measure:  $$\mu^\star:=\int \mu_\varpi^\star d\widehat{\mathbf{p}}(\varpi) \text{  with }\mu_\varpi^\star=\int \delta_{(x,T_x\varpi)}d\mu_\varpi(x)$$ is an $F_k$-invariant lift of $\widehat \mu$, where 
   $F_k:C_k(\mathbf M)\to C_k(\mathbf M)$ is the map induced by $f$ on the the $k$-th contact bundle $C_k(\mathbf M)$.

In the case where the Lyapunov spectrum of $\widehat\mu$ is simple we get that any ergodic component of $\mu^\star$ gives full weight to  a direct sum of Oseledec unstable subspaces. Thus, in this case,  
for $\widehat{\mathbf{p}}$-a.e. $\varpi$, for $\mu_\varpi$-a.e. $x$, $\mathrm{Im}(\varpi)$ is tangent to a sum of Oseledec directions of dimension $k$ at $x$.

\medskip
In \cite{BurguetViana}, the author conjectures the following:

\medskip
\noindent\textbf{Conjecture} (\cite{BurguetViana})\textbf{.} {\em Let $f : \mathbf M\to \mathbf M$ be a $C^\infty$ diffeomorphism of a closed manifold. If $\Vol([\Sigma^k(x)>\Sigma^{k-1}(x)\geq 0])>0$, then there exists an ergodic measure with at least $k$ positive Lyapunov exponents, such that its entropy is larger than or equal to the sum of its $k$ smallest positive Lyapunov exponents.}

\medskip
In the conjecture above, $ \Sigma^k(x):=\limsup \frac{1}{n}\log\|\wedge^k d_xf^n\|$.

\medskip
We prove that almost every ergodic component of the measure $\widehat{\mu}$ from the main theorem satisfies the following (recall, almost every ergodic component of a G-$u$-Gibbs over $k$-disks is a G-$u$-Gibbs over $k$-disks):

\begin{theoremm}\label{entropy}
Let $\widehat{\mu}$ be an ergodic generalized $u$-Gibbs measure  over $k$-disks. Then 
$$h_{\widehat{\mu}}(f)\geq \sum_{k\text{-smallest positive }\chi_i(\widehat{\mu})}\chi_i(\widehat{\mu}).$$
\end{theoremm}

\medskip
\noindent\textbf{Remark:}\text{ }
\begin{enumerate}
\item Note that this statement is optimal, as can be seen by the example which is illustrated in the second-to-last paragraph of \textsection \ref{intro}, where the inequality is an equality.
    \item In particular, this implies that $\widehat{\mu}$ has positive entropy (by the affinity of entropy), and that $\widehat{\mu}$ admits also at least $1$ negative Lyapunov exponent a.e. (by the Ruelle inequality for $f^{-1}$).
\end{enumerate}

\subsubsection{The Strong Viana Conjecture}\label{SVCinMain} Finally, as an application to the main theorem of \textsection \ref{GuGibbsSection1}, we prove a  strong version of Viana conjecture in any dimension.

\begin{definition}[Negative exponents]\label{newNegExp0}
For $k\in \{1,\ldots,d-1\}$ we set $$\varkappa^-_{k+1}(x):=\lim_{\Delta\to 0}\limsup_{q\to\infty}\liminf_{n\to\infty} \varkappa_{k+1,q,n,\Delta}^-(x),$$
where
	$$\varkappa_{k+1,q,n,\Delta}^-(x):=\min_{\overset{\mathcal{L}\subseteq \{1,\ldots,n\}}{\Delta n\leq \#\mathcal{L}\leq n}}\frac{1}{\#\mathcal{L}}\sum_{\ell\in\mathcal{L}}\Phi_{k+1}^{(q)}\circ f^{\ell }(x),$$
	and $$\Phi_{k+1}^{(q)}(x):=\frac{1}{q}\log^+\frac{\|\wedge^k d_xf^q\|}{\|\wedge^{k+1} d_xf^q\|}.$$

\end{definition}

\medskip
\noindent\textbf{Remark:}\text{ }
\begin{enumerate}
    \item For an $f$-invariant probability $\nu$, for $\nu$-a.e. $x$,\\ $\varkappa^-_{k+1}(x)=\max\{0,-\chi_{k+1}(x)\}$, where $\chi_{k+1}(x)$ is the $(k+1)$-th Lyapunov exponent (with multiplicity) at $x$. For proof, see Lemma \ref{negExpCoincide}.
   \item  $\varkappa^-_{k+1}(x)\leq \max\{-\beta_{k+1}(x),0\}$, where
   \begin{align*}
       \beta_{k+1}(x):=&\sup_{\nu\in p\omega(x)}\mathrm{ess \, sup}_\nu \ \chi_{k+1},
   \end{align*} where $p\omega(x)$ denotes the set of empirical measures of $x$ and $\mathrm{ess \, sup}_\nu \ \chi_{k+1}$ is the essential supremum of $y\mapsto \chi_{k+1}(y)$ w.r.t. $\nu$. For proof, see Lemma \ref{limsupToLiminf}.    
\end{enumerate}

\begin{definition}[Points with non-zero Lyapunov exponents]
	For $\chi\geq 0$ and $k\in \{1,\ldots,d-1\}$, we let
    \begin{align*}
		\mathrm{Hyp}_\chi^k:=\Big\{x\in\mathbf{M}:\lambda_{k}(x),\varkappa^-_{k+1}(x)>\chi\Big\}.
	\end{align*}
\end{definition}

\begin{theoremm}[Strong Viana Conjecture]\label{strongv} Let $f\in \mathrm{Diff}^{\infty}(\mathbf M)$ and let $\chi>0$. Then 
    $\mathrm{Vol}$-a.e. $x\in\mathrm{Hyp}_\chi^k$  lies in the basin of attraction of an ergodic $\chi$-hyperbolic SRB measure with exactly $k$ positive Lyapunov exponents.
\end{theoremm}

\medskip 
\noindent\textbf{Remark:}\text{ }
\begin{enumerate}\item  We do not know if the number of SRB measures whose basins cover $\mathrm{Hyp}_\chi^k$  is finite or not (see also the last item of the remark following Theorem \ref{codim}). 
\item In fact, $\Vol\Big(\mathrm{Hyp}_\chi^k\Big)>0$ if and only if there exists a $\chi$-hyperbolic SRB measure with exactly $k$ positive Lyapunov exponents.
\item Moreover, our proof shows that if for some $\chi>0$, $\Vol\Big([\lambda_k>\chi \geq \alpha\geq \beta_{k+1}]\Big)>0$ implies that there exists a $u$-Gibbs measure (i.e. absolutely continuous conditionals along strong unstable manifolds of dimension $k$), whose $(k+1)$-th Lyapunov exponent (with multiplicity) is less or equal to $\alpha$ (note, $\alpha$ may be non-negative).

\end{enumerate}

\begin{cor}\label{viana}
Let $f$ be a $C^\infty$ system with a partially hyperbolic attractor $\Lambda$ of the form 
$T_\Lambda M=E^u\oplus E^c\oplus E^s$ with $\mathrm{dim}(E^c)=2$. Moreover, assume that
$$\forall x\in \Lambda, \ \limsup_{n\to \infty}\frac{1}{n}\log \mathrm{Jac}(d_xf^n|_{E^c(x)})\leq 0.$$
Then Lebesgue a.e. $x$ in the topological basin of $\Lambda$ with $\limsup_n\frac{1}{n}\log \|df^n|_{E^c(x)}\|>\chi>0$ lies in the basin of a $\chi$-hyperbolic ergodic SRB measure.
\end{cor}
In particular, Corollary \ref{viana} applies to the examples by Viana from his celebrated paper \cite{Viana97}.

\subsubsection{Hyperbolic SRB measure for co-dimension $1$}
In the case where $k=d-1$ (i.e the co-dimension $1$ case), we do not need any condition on the exponent $\varkappa_{d}^-$:

\begin{theoremm}\label{codim}
Let $f$ be a $C^\infty$ diffeomorphism of a closed manifold of dimension $d$. Then $\Vol$-a.e. $x$ with $\lambda_{d-1}(x)>0$ lies in the basin of a  hyperbolic SRB measure which admits exactly $d-1$ positive Lyapunov exponents almost everywhere.
\end{theoremm}
\newpage
\medskip
\noindent\textbf{Remark:}\text{ }
\begin{enumerate}
\item Moreover, this is an ``if and only if", as the existence of a hyperbolic SRB measure with $d-1$ positive exponents implies $\Vol([\lambda_{d-1}>0])>0$, as a classical application of Pesin's absolute continuity theorem for stable leaves.
\item In particular for $d=2$ we recover the main result of \cite{BurguetViana} for $C^\infty$ surface diffeomorphisms.
\item In dimension $2$, the number of SRB measures with entropy larger than $a>0$ is finite. This follows from the finiteness of homoclinic classes with entropy larger than $a>0$ which is proved by using a  Sard  argument (firstly appeared in \cite{hertzunique}). We hope  that  
by a similar approach one could show that   $\Vol$-a.e. $x$ with $\lambda_{d-1,p}(x)>0$ lies in the basin of finitely  many SRB measures (see Definition \ref{defOfLmabdaPK} for the definition  of the exponent $\lambda_{d-1,p}$). 
\end{enumerate}

\subsubsection{$C^r$ statements with $1<r<+\infty$}\label{crstatessection}

The  results stated in the above \textsection \ref{SVCinMain} and \textsection \ref{GuGibbsSection1} for  $C^\infty$ diffeomorphisms follow immediately from the following  general $C^r$ version of our Main Theorem. Fix $1<r<+\infty$.
We let $$\forall n\in \mathbb N, \ M_{f^n}:=\max\{\|d_\cdot f^n\|,\|d_\cdot f^{-n}\|\}$$
then 
$$R(f)=\lim_{n\to \infty}\frac{1}{n}\log M_{f^n}.$$

\begin{theoremm}[$C^r$ version]\label{crcr}
Let $f$ be a $C^r$ diffeomorphism of a closed manifold. Then if $\Vol([\lambda_k>\frac{3k^2}{r-1}R(f)])>0$, then there exists a generalized $u$-Gibbs measure over $C^{r-1,1}$ $k$-disks, $\widehat{\mu}$.
\end{theoremm}

\medskip
\noindent\textbf{Remark:}\text{ }
\begin{enumerate}
\item We write $C^{r-1,1}$ (see definition in \textsection \ref{bddGeoSection}) for the case where $r\in \mathbb{N}$, and for simplicity we keep this assumption throughout the paper. However, for $r\notin \mathbb{N}$, the proof easily adapts by using  $C^{\lfloor r\rfloor,r-\lfloor r\rfloor}$. 
\item By the aforementioned examples in dimension two, one can not expect to replace the lower bound  $\frac{3k^2}{r-1}R(f)$ by $0$. However we do not claim  that the lower bound in the above statement is sharp.
\item The $C^r$ smoothness property together with the assumption\\ $\Vol([\lambda_k>\frac{3k^2}{r-1}R(f)])>0$ is used in \textsection \ref{posDensOfGeoTimes} to construct a disk with a positive density of {\em geometric times} on a subset of positive disk-volume, using Yomdin theory. Given such a disk, the construction of the G-$u$-Gibbs measure and its properties is done in \textsection \ref{constMu}, \textsection \ref{smoothCondSect}, and \textsection \ref{propOfMuSect}. All arguments in those sections assume only $f\in \mathrm{Diff}^{1+}(\mathbf{M})$.
\item In \textsection \ref{SVCSect}, by using the construction of the SRB measure from the previous sections we show in addition that $\Vol$-a.e. $x\in [\lambda_k>\frac{3k^2}{r-1}R(f)]\cap [\varkappa^-_{k+1}>\chi ]$ lies in the basin of a $\min\{\chi, \frac{3k^2}{r-1}R(f)\}$-hyperbolic SRB measure. The  proofs in this section assumes only $C^{1+}$ regularity as well. In addition, $\Vol$-a.e. $x\in [\lambda_{d-1}>\frac{3k^2}{r-1}R(f)]$ lies in the basin of a hyperbolic SRB measure.
\end{enumerate}

\subsection{Key steps of proof}\label{keyStepsOfProof}

The proof is composed of seven key steps, which we describe below.

\begin{enumerate}
    \item \textbf{A disk with positive disk-volume for expanding points:} The first step is reducing the dynamics to dynamics of an ``almost expanding" map. This is achieved by  an embedded $k$-dimensional disk $D$, such that the induced Riemannian volume $\Vol_D$ gives a positive measure  to points whose expansion of $\lambda_k(\cdot)$ is achieved on $T_\cdot D$. In particular it requires using an iterative power of the dynamics of $f$, to observe the expansion more directly.
    \item \textbf{Refined high-dimensional Yomdin theory:} The second key step is developing a quantitative approach to Yomdin theory which applies to high-dimensional disks. We expand here a bit about this step, as it is crucial.

    Yomdin theory allows to partition and reparameterize semi-algebraic sets into components whose image under the dynamics of a smooth map remain with a bounded derivative. The idea is that if the image of a disk has to ``bend" a lot, it forces the number of components to be larger. The strength of Yomdin theory is controlling the number of such components under a $C^\infty$ map, by approximating the map with its Taylor expansion, and proving that the number is exponentially small. The usefulness of Yomdin theory is demonstrated for example in \cite{Yomdin87, NewhouseVolGrowth,BCSexp,Burguetphysical}.

    In \cite{BurguetViana} the second author develops an alternative approach, which allows to better control the geometry of curves. The geometry of a curve is simple, as it entails of only its derivative and length. The geometry of high-dimensional disks is much more complicated, as they could be very narrow in some directions, but large in other direction, or even wind above themselves like a staircase. Moreover, the volume of the disk can increase under the dynamics, while some directions demonstrate contraction.

    In this paper we develop an applicable high-dimensional quantitative approach to Yomdin theory, which allows us to control the geometry of the image of a disk under a smooth map. 
    Moreover, we are able to do so in a manner where the reparameterized disks (called {\em Yomdin charts}) form a partition. We call this partition the {\em Yomdin partition}, and it serves an important role later when we construct and compare measures by the atoms of this partition.
\item\textbf{Tree dynamics of the disk, and positive frequency of ``good" times:}  The third key step is to construct the tree of the dynamics of the disk $D$. This refers to the tree structure which is achieved by applying the Yomdin subdivision iteratively to the images of $D$ under the dynamics. In particular, this tree description is crucial in order to study the points on the disk $D$ whose corresponding position on the tree at time $k$ has the {\em bounded geometry} property, for a uniformly positive portion of $k\leq n$. This is the only part in the paper where we require the regularity of the diffeomorphims to be possibly greater than $C^{1+}$. The notion of bounded geometry is how we keep enough control of the geometry of the Yomdin subdivision, without being too restrictive so the tree structure becomes insufficient. Here we makes use of the quantitative Yomdin theory for high-dimensional disks.

\item\textbf{A limiting measure on the space of Measured Disks:} In the fourth step, we use the positive density of times with bounded geometry for a subset of positive Lebesgue measure on $D$ in order to construct a measure on the space of {\em Measured Disks}- that is the space of embedded $k$-dimensional disks in $\mathbf M$, endowed with a probability measure on them. The positive density of times with bounded geometry allows us to restrict to a pre-compact subset of the space of Measured Disks, in which we can take limits.

The main idea here is to construct a sequence of measures which are not invariant, but correspond to the restriction of the desired invariant measure onto increasing subsets (similarly to Pesin blocks). The subsets are parameterized by the bounded geometry of the disks.

\item\textbf{Absolutely continuous conditionals:}
In the fifth step, we prove that the invariant measure which we construct in the fourth step is given by a disintegration into absolutely continuous measures on disks which are subordinated to the unstable foliation. In fact, we prove that these conditional measures are equivalent to the induced Riemannian volume on the disks.

We decompose each conditional measure in the limiting process into atoms of the Yomdin partition, and bound the total mass of all atoms on which the conditional measure does not compare with the respective disk volume. The idea is to make sure that the conditional measure and the respective disk volume compare over finer and finer atoms, as we take the limit, guaranteeing that the limiting measures will coincide over all continuous functions. 

Such analysis, and a choice for an increasing degree of atoms which still cover a big portion of the measure, relies on the sub-exponential decay of the measure of sets which we push forward.

\item\textbf{Dynamical properties of the invariant measure:} In the sixth step, we prove the variety of dynamical properties which the invariant measure holds. That is, we prove the estimates of the exponents, the entropy; And we prove the fact that the limiting measures on the space of {\em Measured Disks} concentrate on disks which are contained in unstable leaves. This gives an indirect construction of disks which are contained in unstable leaves, and an indirect construction of their densities. By ``indirect", we mean without the constructive approach of a graph transform or a Perron-Hadamard method.

\item\textbf{Controlling the transverse Lyapunov exponents:} In the seventh and final step, we estimate the Lyapunov exponents of the G-$u$-Gibbs measure which we construct, in the direction transverse to the disks on which it disintegrates with absolutely continuous conditionals.

In general Lyapunov exponents may not be continuous, and using the fact that our G-$u$-Gibbs measure is constructed by pushing forwards sets, we are able to prove that the negative exponents for the empirical measures translate to the limit. In fact, our methods allow to control the Lyapunov spectrum of the limiting measure, through the spectrum of the emprical measures. Controlling the spectrum of a measure which is constructed this way is meaningful even in the Anosov setting.
\end{enumerate}

\section{Setup and definitions}\label{setupAndDefs}

\subsection{Lyapunov exponents}

\subsubsection{Standard Lyapunov exponent}

For a dynamical system $T:X\circlearrowleft$  and a normed vector bundle $\pi:V\to X$, a linear cocyle  $A:V\circlearrowleft$ is  a map satisfying $\pi\circ G=T$ such that the restriction of $A$ to the vector space $V_x:=\pi^{-1}(x)$ to $V_{Tx}$ is a linear isomorphism. 
Then one defines the Lyapunov exponent  of $A$ at $v=(x,v_x)\in V$ as follows
$$\chi_A(v)=\limsup_n\frac{1}{n}\log \|A^n(v)\|.$$
We may define the Lyapunov exponent $\chi_A$ on the projective bundle $\mathbb P V$ by letting $\chi_A([v])=\chi_A(v)$ for any $0\neq v\in V$  with its associated class $[v]\in \mathbb PV$. 
Consider a $C^1$ diffeomorphism $f$  on a compact smooth Riemannian manifold $\mathbf M$ of dimension $d$. For any $x\in \mathbf M$ we  let $(\chi_k(x))_{1\leq k\leq d}$ be the usual Lyapunov exponent associated to the derivative cocycle of $f$. It is well-known that  $$\chi_1(x)=\overline{\chi}(x)=\limsup_n\frac{1}{n}\log \|d_xf^n\|.$$
When $\nu$ is a $f$-invariant probability measure on $\mathbf M$ we let $\chi_k(\nu)=\int \chi_k(x)\, d\mu(x)$ for any $1\leq k\leq d$.

\subsubsection{New positive exponents}
We introduce now the new positive exponents $\lambda_k$, $1\leq k\leq d$. 

\begin{definition}\label{defOfLmabdaPK} For a point $x\in \mathbf{M}$ and  $k\in\{1,\ldots,d\}$, we let for all $p\in \mathbb N\setminus \{0\}$,\ $\lambda_{k,p}(x):=\limsup_{n\to\infty}\lambda_{k,p,n}(x)$, where 
$$\lambda_{k,p,n}(x):=\frac{1}{n}\left(\log \|\wedge^kd_xf^n\|-\frac{1}{p}\sum_{\ell=0}^{n-1}\log^+\max_{1\leq j\leq k-1}\|\wedge^{j}
	d_{f^\ell x} f^p\|\right).$$
\end{definition}

\medskip
\noindent\textbf{Remark:}\text{ }
\begin{enumerate}
\item For $k=1$ and for all $p$, we have \\ $\lambda_{1,p}(x)=\lambda_1(x)=\limsup_n\frac{1}{n}\log \|d_xf^n\|=\overline\chi(x)$.
    \item We always have $\overline\chi(x)\geq \lambda_{k,p}(x)$.
    \item  When $x$ is a typical point for an ergodic measure $\nu$, then 
$$\lambda_{k,p}(x)=\sum_{1\leq l\leq k }\chi_l(\nu)-\frac{1}{p} \int \log^+\max_{j\leq k-1}\|\wedge^{j}
	d_{x} f^p\| d\nu(x).$$
\end{enumerate}

\begin{lemma}\label{limOnP}
For any $x\in \mathbf M$ and $1\leq k\leq d$ the sequence  $\left(\lambda_{k,p}(x)\right)_p$ is converging to $\lambda_k(x):=\sup_p\lambda_{k,p}(x)$.
\end{lemma}
\begin{proof} Fix $x\in \mathbf M$ and $1\leq k\leq d$. 
For all $n\in \mathbb N\setminus \{0\}$, we set $$\alpha_p(n,x):=\sum_{\ell=0}^{n-1}\log^+\max_{j\leq k-1}\|\wedge^{j}
	d_{f^\ell x} f^p\|.$$
We also let $C_f:=\log^+\max_{1\leq j\leq k-1}\sup_{y\in \mathbf M}\|\wedge^{j}
	d_{y} f\|$. Observe that 
   \begin{align}\label{limompun}\forall p,n\in \mathbb N\setminus\{0\} \, \forall t\in \mathbb N ,\   \alpha_p(n,f^tx)&\leq \alpha_p(n,x)+tp C_f.
   \end{align}

Fix $\epsilon>0$.  Take $q_0$ such that $\lambda_{k,q_0}(x)>\sup_q\lambda_{k,q}(x)-\epsilon/2$ and let $p_0=\lceil\frac{2}{\epsilon}q_0C_f\rceil$. Let us show that  $\lambda_{k,p}(x)\geq \lambda_{k,q_0}(x)-\epsilon/2> \sup_q\lambda_{k,q}(x)-\epsilon$ for any $p\geq p_0$. Observe that for $p=tq_0+r$, $0\leq r<q_0$, we have 
\begin{align*}
\alpha_p(n,x)\leq& \sum_{\ell=0}^{n-1}\Big(\log^+\max_{1\leq j\leq k-1}\|\wedge^{j}
	d_{f^{tq_0+\ell}x} f^{r}\|\\
    &+ \sum_{0\leq t'<t}\log^+\max_{1\leq j\leq k-1}\|\wedge^{j}
	d_{f^{t'q_0+\ell}x} f^{q_0}\|\Big)\\
    \leq& n rC_f+\sum_{0\leq t'<t}\alpha_{q_0}(n,f^{t'q_0}x)\\
    \leq& n q_0C_f+t\alpha_{q_0}(n,x)+tq_0C_f \text{  by using (\ref{limompun})}\\
    \leq& n p\frac{\epsilon}{2}+t\alpha_{q_0}(n,x)+tq_0C_f \text{ according to the choice of $p_0(\leq p)$}.
\end{align*}

Therefore,
\begin{align}\label{forSection9}
    \lambda_{k,p}(x) & =\limsup_n\frac{1}{n}\left(\log \|\wedge^kd_xf^n\|-\frac{1}{p}\alpha_p(n,x)\right)\nonumber\\
    &\geq \limsup_n\frac{1}{n}\left(\log \|\wedge^kd_xf^n\|-\frac{t}{p}\alpha_{q_0}(n,x)\right)-\frac{\epsilon}{2}\nonumber\\
    &\geq \limsup_n\frac{1}{n}\left(\log \|\wedge^kd_xf^n\|-\frac{1}{q_0}\alpha_{q_0}(n,x)\right)-\frac{\epsilon}{2}\nonumber\\
    &\geq \lambda_{k,q_0}(x)-\frac{\epsilon}{2}.
\end{align}

\end{proof}

\medskip
\noindent\textbf{Remark:}\text{ }
\begin{enumerate}
    \item Lemma \ref{limOnP} together with the remark following Definition \ref{defOfLmabdaPK} implies that for a typical point $x$ for an ergodic measure $\nu$, $\lambda_k(x)=\chi_k(\nu)+\sum\limits_{1\leq l\leq k-1}\min\{\chi_l(\nu),0\}$. In particular, when $\chi_k(x)=\chi_k(\nu)>0$, one gets $\lambda_k(x)=\chi_k(x)$.
    \item for a $C^1$ partially hyperbolic splitting of the form $E^\mathrm{u}\oplus E^\mathrm{cs}$ the exponent $\lambda_{k+1}(x)$ for $k=\mathrm{dim}(E^\mathrm{u})$ is  equal to the top Lyapunov  exponent of $df|_{E^\mathrm{cs}}$. Indeed, for an adapted metric, we have $\max\{\max_{l\leq k}\|\wedge d_xf^m\|,1\}=\|\wedge^kd_xf^m|_{E^\mathrm{u}}\|$ and $\| \wedge^{k+1}d_xf^m\|=\|\wedge^kd_xf^m|_{E^\mathrm{u}}\|\times \|d_xf^m|_{E^\mathrm{cs}}\|$ for all $m\in \mathbb N$, so that we get for any $p\in \mathbb N$:  
    \begin{align*}
    &\lambda_{k+1,p}(x)=\\
    &=\limsup_n\frac{1}{n}\left(\log \|\wedge^{k+1}d_xf^n\|-\frac{1}{p}\sum_{\ell=0}^{n-1}\log^+\max_{1\leq j\leq k}\|\wedge^{j}
	d_{f^\ell x} f^p\|\right)\\
    &=\limsup_n\frac{1}{n}\left(\log \|\wedge^{k+1}d_xf^n\|-\frac{1}{p}\sum_{\ell=0}^{n-1}\log \|\wedge^{k}
	d_{f^\ell x} f^p|_{E^\mathrm{u}}\|\right)\\
    & =\limsup_n\frac{1}{n}\left(\log \|\wedge^{k+1}d_xf^n\|-\log \|\wedge^{k}
	d_{x} f^n|_{E^\mathrm{u}}\|\right)\\
    &=\limsup_n\frac{1}{n}\log \|d_xf^n|_{E^\mathrm{cs}}\|. 
    \end{align*}
\item Observe that $\lambda_d(x)\leq \limsup \frac{1}{n}\log \Jac_x(f^n)$. In particular $\Vol([\lambda_d>0])=0$.  
 If not, then there would exist a $\chi>0$ and sets $E_{n_j}$ s.t.   $\Vol(E_{n_j})\geq \frac{1}{n_j^2}$ with $E_{n_j}:=\{x:\Jac_x(f^{n_j})\geq e^{n_j\chi}\}$, $n_j\uparrow\infty$ (by the Borel-Cantelli lemma). Then $1\geq \Vol(f^{n_j}[E_{n_j}])\geq \frac{e^{\chi n_j}}{n_j^2}$, which is a contradiction.
\end{enumerate}

\subsection{A disk and a positive disk-volume subset}\label{theDiskAndSubset}

\begin{definition}\label{linCocyc} Fix $p\in\mathbb N$ and let $k\leq \dim \mathbf M$, we consider the linear cocycle $A_{k,p}$ on $\wedge^k T\mathbf M$ over $f$ defined as 

$$\forall x\in \mathbf{M}, \ v\in \wedge^k T_x\mathbf M, \ A_{k,p}(x,v):=\frac{\wedge^k d_xf (v)}{\max\{1, \max_{j\leq k-1}\|\wedge^{j}
	d_xf^p\|^\frac{1}{p}\}}.$$
\end{definition}

\medskip
\noindent\textbf{Remark:}\text{ }
\begin{enumerate}
    \item The top Lyapunov exponent of $A_{k,p}$ is equal to $\lambda_{k,p}$. 
    \item 
When $p$ divides $q$, we have $\chi_{A_{k,q}}(v)\geq \chi_{A_{k,p}}(v)$ for all $v\in\wedge^k T_x\mathbf M $.
\end{enumerate}

\medskip

\begin{definition}[$k$-Volume]\text{ }
\begin{enumerate}
    \item For a smooth embedded $k$-disk $D$, we denote by $\Vol_{D}$ the restriction of the Riemannian volume form to the tangent space of $D$ normalized to induce a probability measure.  
        \item For a subset $E$ of a smooth embedded $k$-disk $D$ we denote by $\Vol_k(E)$ the measure of $E$ w.r.t. the
$k$-volume of $D$ induced by the Riemannian structure of $\mathbf{M}$.\footnote{ The notation does not mention the dependence on the disk $D$, as two $C^{1}$ embedding  disks, intersect transversely on at most a set of zero measure with respect to the volumes on these disks.}
\end{enumerate}
\end{definition}

Let $\iota$ be the Pl\"ucker embedding of the Grassmanian $\mathrm{Grass}( T\mathbf M)$ to the projective space of the exterior bundle $\mathbb{P}\wedge^*T\mathbf M $.

\begin{lemma}\label{theDiskD}
	Write $E:=[\lambda_k>a]$ with $a>0$, then if $\Vol(E)>0$, then $\exists \chi>a$, $p_0\in\mathbb N$, and a $k$-disk $D$ s.t.   $\forall p\in p_0\mathbb{N}$, $$\Vol_D([\chi_{A_{k,p}}(x,\iota(T_xD))> \chi])>0.$$
\end{lemma}
\begin{proof}
	Fix $\chi>a$ and $p_0\in\mathbb N$ s.t.   $\Vol(E')>0$ where we set  $$E':=\left\{x\in E: \lambda_{k,p_0}(x)> \chi\right\}.$$ 
    Let $\{0\}=F_0\subset F_1\subset \ldots\subset F_s=\wedge^k T_\cdot \mathbf M$, $s\geq 1$ be the Lyapunov flag associated to $A_{k,p_0}$. By definition the  Lyapunov space $F_{s-1}(x)$ is   given by  $F_{s-1}(x)=\{v\in \wedge^k T_xM: \chi_{A_{k,p_0}}(x,v)< \lambda_{k,p_0}(x):=\sup_{w\in \wedge^kT_M}\chi_{A_{k,p_0}}(x,w)\}$. The  map $x\mapsto F_{s-1}(x)$ is measurable (for the Borel algebra associated to the usual topology on Grassmanian), and so it is continuous on a subset $E''\subseteq E'$ with $\Vol(E'')>0$. We may choose a Lebesgue density point  $x$ of $E''$. Let $U$ be a small neighborhood of $x$, such that  $F_{s-1}(y)$ is almost constant for $y\in U\cap E''$. 

We may disintegrate the volume by a family of affine $k$-planes $H$ (in some local chart at $x$ which can be assumed containing $U$)  satisfying   $\iota(H)\notin \{F_{s-1}(y), \ y\in U\cap E''\}$. By Fubini's theorem, there is a $k$-disk disk $D$ contained in one of these planes with $\Vol_D(E'')>0$. This completes the proof as  we have  for all $x\in D\cap E''$, $$\chi_{A_{k,p_0}}(x,\iota(T_xD))> \chi.$$ 
\end{proof}

\subsection{Sequence of finite-time expanding subsets  with a sub-exponential volume}

\begin{lemma}\label{preBC}
Let $D$ and $p_0$ be given by Lemma \ref{theDiskD}, and let $p\in p_0\mathbb N$. There exist $\chi'>\chi>a$, $\mathcal{N}\subseteq \mathbb N$ with $\#\mathcal{N}=\infty$ and subsets of $D$, $(B_n^p)_{n\in\mathcal{N}}$, s.t.   for all $n\in\mathcal{N}$,
\begin{enumerate}
	\item $\Vol_D(B_n^p)\geq \frac{2}{n^2}$,
	\item $\forall n\in\mathcal{N}, \forall x\in B_n^p$, $|A_{k,p}^{n}(x,\iota(T_xD))|\geq e^{\chi' n}$.
\end{enumerate}
\end{lemma}
\begin{proof}
	This is a consequence of the Borel-Cantelli lemma applied to the set $E''$ from the proof of Lemma \ref{theDiskD}.
\end{proof}

\section{The tree description of the dynamics of a disk}\label{treeSect}
In this section we aim to describe the dynamics of a disk using a tree associated with its Yomdin partitions. In \cite{BurguetViana} it was done for curves, and here we introduce an extension to high-dimensional disks (which may exhibit much more complicated geometry, lack of conformality, and non-trivial boundary).
We first generalize the notion of bounded geometry to higher-dimensional disks.

\subsection{Bounded geometry}\label{bddGeoSection}\text{ }

\medskip
\noindent\textbf{Notation:} By a $C^{r-1,1}$ map $F$, we mean a $C^{r-1}$ map $F$ with Lipschitz $C^{r-1}$ derivatives and we define the following semi-norm for $k\leq r$: $$\|d^{k}_\cdot F \|:=\max_{|\underline{\alpha}|= k-1}\mathrm{Lip}(|\partial^{\underline{\alpha}}_\cdot F|),$$ 
where $\mathrm{Lip}(\cdot)$ denotes the Lipschitz constant. We also denote $\|F\|_{r}=\max\limits_{1\leq k\leq r} \|d^k_\cdot F\|$.

\begin{definition}[Bounded couple]\label{bddCouple}
A couple of $C^{r-1,1}$ map $(\sigma,\theta)$, where $\sigma: [0,1]^k\to \mathbf M$ is a $C^{r}$ embedding  and $\theta:[0,1]^k\to[0,1]^k$ is a $C^{r-1,1}$ map, is called {\em a bounded couple} if $$\forall s=1,\cdots, r-1,\ \|d^{s}_\cdot\left(t\mapsto \wedge^k d_{\theta(t)}\sigma \right)\| \leq \frac{1}{10 k}\sup_t|\wedge^k d_{\theta(t)} \sigma|.$$
\end{definition}

\begin{lemma}\label{densityJuy} Let $(\sigma,\theta)$ be a bounded couple, then $$\forall t,t'\in [0,1]^k, \ |\wedge^k d_{\theta(t)}\sigma-\wedge^k d_{\theta(t')}\sigma|\leq \frac{1}{10}|\wedge^k d_{\theta(t)}\sigma|.$$
\end{lemma}

\begin{proof}
    Let $t_*$ be such that $|\wedge^k d_{\theta(t_*)} \sigma|=\sup_{s'}|\wedge^k d_{\theta(s')} \sigma|$. Then for any $t,t'\in [0,1]^k$,
\begin{align}\label{bbdDistort1.5July5}\Big||\wedge^k d_{\theta(t')} \sigma|-|\wedge^k d_{\theta(t)} \sigma|\Big|\leq&
\Big|\wedge^k d_{\theta(t)} \sigma-\wedge^k d_{\theta(t')} \sigma\Big|\nonumber\\
\leq& |t-t'|\cdot\sup_\zeta\|d_\zeta\left(\cdot\mapsto \wedge^k d_{\theta(\cdot)}\sigma \right)\|\nonumber\\
\leq& |t-t'|\cdot \frac{1}{10k} |\wedge^k d_{\theta(t_*)} \sigma|\nonumber\\
\leq &\sqrt{k}\cdot \frac{1}{10k} |\wedge^k d_{\theta(t_*)} \sigma|.
\end{align}
Then, by letting $t'=t_*$, we obtain 
\begin{equation}\label{bbdDistort1.5July25}
\left(1-\frac{1}{10\sqrt{k}}\right)|\wedge^k d_{\theta(t_*)} \sigma|
\leq |\wedge^k d_{\theta(t)} \sigma|
\leq |\wedge^k d_{\theta(t_*)} \sigma|.
\end{equation}
Therefore, by bootstrapping \eqref{bbdDistort1.5July25} and plugging it back into \eqref{bbdDistort1.5July5}, we get that for every $t,t'\in [0,1]^k$,
\begin{equation}\label{bbdDistort1.5July35}
    \Big|\wedge^k d_{\theta(t)} \sigma-\wedge^k d_{\theta(t')} \sigma\Big|\leq|t-t'|\cdot  |\wedge^k d_{\theta(t)} \sigma|\cdot \frac{1}{\sqrt{k}(10\sqrt{k}-1)}.
\end{equation}
\end{proof}

We may control the oscillation of the tangent space of a bounded couple as follows. We refer to Appendix \ref{anglee} for the definition of the angle $\angle H,H'$ between two  vector spaces $H$ and $H'$.  
\begin{lemma}\label{angle}
Let $(\sigma, \theta)$ be a bounded couple. Then for any $x,y\in \mathrm{Im}(\theta)$ we have $$\angle T_{x}\sigma, T_y\sigma<\pi/6.$$ 
\end{lemma}

\begin{proof}It follows directly from Lemma \ref{densityJuy} and  Lemma \ref{angl} in the appendix.
\end{proof}

\begin{lemma}[Bounded distortion property]\label{bounddis}
Let $(\sigma, \theta)$ be a bounded couple. Then for any $x,y\in \mathrm{Im}(\theta)$ we have $$\frac{|\wedge^k d_{x} \sigma|}{|\wedge^k d_{y} \sigma|}<\sqrt 2.$$ 
\end{lemma}
\begin{proof}
By \eqref{bbdDistort1.5July25}, we have for all $t,t'\in [0,1]^k$ 
$$\frac{|\wedge^k d_{\theta(t)} \sigma|}{|\wedge^k d_{\theta(t')} \sigma|}\leq \frac{1}{1-\frac{1}{10\sqrt k}}<\sqrt 2.$$ 
\end{proof}

\begin{definition}[strongly $\epsilon$-bounded couple]\label{StrBddCpl}
Let $\epsilon
>0$. A bounded couple $(\sigma,\theta)$ is called a {\em strongly $\epsilon$-bounded couple} if $\|\sigma\circ \theta \|_{r}\leq \frac{\epsilon}{\sqrt k}$.
\end{definition}

\medskip

\begin{definition}[Admissible family]\label{admissibleFamily} Let $\epsilon>0$ and let $(\sigma,\theta')$ be a bounded couple.  A finite family $\Theta$ of maps $\left\{\theta:[0,1]^k\circlearrowleft\right\}_{\theta\in \Theta}$ is said to be a $(\sigma, \theta',\epsilon)$-admissible family  when
\begin{enumerate}
\item  $(\sigma, \theta)$ is a strongly $\epsilon$-bounded couple for any $\theta\in \Theta$,
\item  $\coprod^k_{\theta\in \Theta}\mathrm{Im}(\theta)\subset\mathrm{Im}(\theta') $, where $\coprod^k$ denotes a union of $k$-manifolds which is disjoint up to 
to their boundaries,
\item There is  $x\in \mathrm{Im}(\sigma\circ \theta')$ and a $1$-Lipschitz map $\psi:H\to H^\bot$ with $H\subset T_x\mathbf M$ being the tangent space  of $\sigma$ at $x$  such that $$\Gamma_\psi(\epsilon)\subset \sigma(\Theta):=\coprod\nolimits^k_{\theta\in \Theta} \mathrm{Im}(\sigma\circ \theta)\subset \Gamma_\psi(2\epsilon),$$
 where $\Gamma_\psi(\delta):=\exp_x\{(v,\psi(v))\in H\oplus H^{\bot}: \ | v|\leq \delta\}$ for $\delta>0$ with $\exp_x$ being the exponential map at $x$ (such a set $\Gamma_\psi(\delta)$ is called a $\delta$-graph). 
\end{enumerate}
\end{definition}

\medskip
\noindent\textbf{Remark:} There is a constant $C_d>1$ such that: 
\begin{enumerate}
\item For any strongly $\epsilon$-bounded couple, 
$$\Vol_k(\mathrm{Im}(\sigma\circ \theta))\leq C_d\epsilon^k.$$
\item For any $(\sigma, \theta',\epsilon)$-admissible family $\Theta$ of a bounded couple $(\sigma, \theta')$,  
$$C_d^{-1}\epsilon^k\leq\Vol_k(\sigma(\Theta))\leq  C_de^{k}\sharp \Theta.$$
\end{enumerate}

\subsection{The tree of the disk dynamics  and the Yomdin partitions}\label{YomdinPtns}

We consider a $C^r$ smooth  diffeomorphism $g:\mathbf M\circlearrowleft$ and a $C^r$ smooth  embedded disk $\sigma:[0,1]^k\rightarrow \mathbf M$ with $ \mathbb N\ni r>1$. 
We state a global reparametrization lemma to describe the dynamics on  the image of $\sigma$ by generalizing the case $k=1$ which was established in \cite{BurguetViana}. We will apply this lemma to $g=f^p$ for large $p$ with $f$ being the $C^r$ smooth system under study.\\

\medskip
\noindent\textbf{Notation:} 
\begin{enumerate}
\item We denote the image of $\sigma$ by $D$. 
 \item For any $n\in\mathbb N$ we let $\sigma_n$ (resp. $D_n$) be the $C^r$ embedded disk defined as 
$\sigma_n=g^n\circ \sigma$ (resp. $D_n=g^n[D]$). 
\item For two maps $\theta, \theta':[0,1]^k\circlearrowleft$ we write $\theta>\theta'$ when there is a map $\phi:[0,1]^k\circlearrowleft$ with $\|d_\cdot\phi\|<1/2$ and $\theta' =\theta\circ \phi$. 
\item We denote by $G$ (resp. $F$) the map induced by $g$ (resp. $f$) on the projective space $\mathbb P \wedge^k T\mathbf M$ of the $k$-th exterior tangent bundle.
\item For $\hat x=(x,w_x)\in \mathbb P\wedge^k T\mathbf M$ we also write $w_x$ for a representative with unit norm. Then we  let $\mathfrak l(\hat x)$, $\mathfrak l'(x)$ be the following integers:
 \begin{align*}
\mathfrak l ( \hat x):=&\lfloor\log |\wedge^kd_xg(w_x)|\rfloor,\\
\mathfrak l'(x):=&\lceil\log^+\max_{k'<k}\|\wedge^{k'}d_x g \|\rceil.     
 \end{align*}
 \item For $x\in D$, we let  $\widetilde x\in \mathbb P\wedge^k T_x\mathbf M$ denote the class of the   tangent space to $D$ at $x$, formally $\widetilde x=\iota (T_x D)$.
 \end{enumerate}

\subsubsection*{The tree description} We code the dynamics of $g$ on the image of $\sigma$ by a   directed rooted tree $\mathcal T$, with all edges pointing away from the root. Moreover, the nodes of our tree will be colored, either in blue or in red. The level of a node is the number of edges along the unique path between it and the root node. We let $\mathcal T_n$ (resp. $\underline{\mathcal T}_n$, $\overline{\mathcal T}_n$) be the set of nodes  (resp. blue, red nodes) of level $n$. For all $\ell\leq n-1$  and for all $\mathbf i^n\in \mathcal T_n$, we also let  $\mathcal T_\ell\ni \mathbf i^n_{\ell}$ be the node of level $\ell$ leading to $\mathbf i^{n}$. 
 We assign to each node $\mathbf i^n\in \mathcal T_n$ a family of maps 
 $\Theta_{\mathbf i^n}$ such that we have  for some constant $A_{r,d}$ depending  only on $r$ and $d$, which is specified afterwards:
 \begin{itemize}
 \item either $\mathbf i^n\in \overline{\mathcal T}_n$, then $\Theta_{\mathbf i^n}$ is a $(\sigma_n, \theta'_{\mathbf i^n},\epsilon)$-admissible family with $\# \Theta_{\mathbf i^n}<A_{r,d}$ where $\theta'_{\mathbf i^n}$ satisfies $\theta_{\mathbf i^{n-1}}>\theta'_{\mathbf i^n}>\theta_{\mathbf i^n}$  for some $\theta_{\mathbf i^{n-1}}\in \Theta_{\mathbf i^n_{n-1}}$ and for all $\theta_{\mathbf i^n}\in \Theta_{\mathbf i^n}$,
\item or $\mathbf i^n\in \underline{\mathcal T}_n$, then   $\Theta_{\mathbf i^n}=\{\theta_{\mathbf i^n} \}$ is a singleton with   $(\sigma_n, \theta_{\mathbf i^n}) $ being a strongly  $\epsilon$-bounded couple with $\theta_{\mathbf i^{n-1}}>\theta_{\mathbf i^n}$ for some $\theta_{\mathbf i^{n-1}}\in \Theta_{\mathbf i^n_{n-1}}$. In this case we put $\theta'_{\mathbf i^n}=\theta_{\mathbf i^n}$.

\end{itemize}

\subsubsection*{Choice of the scale $\epsilon$}
Recall that  $\exp_x$ denotes the exponential map at $x$ and let $R_\mathrm{inj}$ be the radius of injectivity of $(\mathbf M, \|\cdot\|)$. For $\frac{R_\mathrm{inj}}{2}>\epsilon>0$ we let $g^x_{2\epsilon}= g\circ  \exp_x(2\epsilon\cdot): \{w\in T_x\mathbf M, \ \|w\|\leq 1\}\rightarrow \mathbf M$. Then   $\|d^s g^x_{2\epsilon}\|_\infty \leq (2\epsilon)^{s} \sup_{\stackrel{w\in T_x\mathbf M,}{  |w|\leq 2\epsilon}}\|d^s_w (g\circ \exp_x)\|$.  In particular there is $\epsilon_0=\epsilon_0(g)<\frac{R_\mathrm{inj}}{2}$  depending only on $\mathbf M$ and $\|d^s_\cdot g\|_\infty$, $s=1, \ldots, r$, such that  $\|d^s_\cdot g_{2\epsilon}^x\|_\infty\leq 3\epsilon \|d_xg\|$
 and $\|d^{s'}_\cdot \wedge^kd g_{2\epsilon}^x\|_\infty\leq 3\epsilon^k \|\wedge^kd_xg\|$
 for all $s=1,\ldots,r$ (resp. $s'=1,\cdots, r-1$), all $x\in \mathbf M$ and all $\epsilon\leq \epsilon_0(g)$.  We may also choose $\epsilon_0(g)>0$ sufficiently small so that for any $j=1,\ldots,d$,
\begin{equation}\label{smallosc}\frac{\|\wedge^{j}d_xg\|}{\|\wedge^{j}d_yg\|}\leq 2
 \end{equation} whenever $x$ and $y$ are $\epsilon_0(g)$-close.

\begin{prop}\label{firstLemmaOfSection} Let $\epsilon_0(g)\geq \epsilon>0$  and let  $\sigma:[0,1]^k\rightarrow \mathbf{M}$ such that $(\sigma, \mathrm{Id}_{[0,1]^k})$ is a strongly $\epsilon$-bounded couple. Then there is a tree $\mathcal T$ as above 
such that we  have for some  other universal constant  $C_{r,d}$: 
\begin{enumerate}
\item $\Theta_{\mathbf i^0}=\left\{\theta_{\mathbf i^0}\right\}=\mathrm{Id}_{[0,1]^k}$
and for any $\mathbf i^{n-1}\in \mathcal T_{n-1}$ we have 
$$\sigma_n\left( \Theta_{\mathbf i^{n-1}}\right)=\coprod\nolimits^k_{\mathbf i^n\in \mathcal T_n, \ \mathbf i^n_{n-1}=\mathbf i^{n-1}}\sigma_n\left( \Theta_{\mathbf i^n}\right).$$ 
In particular
$$\coprod\nolimits^k_{\mathbf i^n\in \mathcal T_n} \sigma_n(\Theta_{\mathbf i^n})=\mathrm{Im}(\sigma_n).$$

 \item $\forall  \mathbf i^{n-1}\in \mathcal T_{n-1}$ and for all $(\mathfrak l_n,\mathfrak l'_n)\in \mathbb Z\times \mathbb N$  we have
\begin{align*}
    \# \Big\{ \mathbf i^n\in \overline{\mathcal T}_n: &\ \mathbf i^n_{n-1}= \mathbf i^{n-1} 
    \text{ and } \\
    & \exists x\in \sigma(\Theta_{\mathbf i^n}), \ s.t.    \ \mathfrak l(G^{n-1}\widetilde x) =\mathfrak l_n\Big\}\leq 2C_{d}^2e^{\mathfrak l_n}, \\
    \# \Big\{ \mathbf i^n\in \underline{\mathcal T}_n:& \  \mathbf i^n_{n-1}= \mathbf i^{n-1}
    \text{ and } \\
    &\exists x\in \sigma(\Theta_{\mathbf i^n}),  \ s.t.    \ \mathfrak l'(g^{n-1} x )=\mathfrak l'_n \Big\}\leq C_{r,d} M_g^{\frac{3k^2}{r-1}}e^{\mathfrak l'_n}.
\end{align*}
\end{enumerate} 
\end{prop}

This statement is a global version of Proposition 7 of \cite{Burguetphysical}. The proof is the content of \textsection \ref{prop37proof}.

\begin{definition}[Yomdin partition]\label{yomd} For $x\in D$ and $0\leq\ell\leq n$ we let 
$\varpi_\ell(x)$ be the unique set of the form $\mathrm{Im}(\sigma_\ell\circ \theta_{\mathbf i^\ell})$, $\theta_{\mathbf i^\ell}\in \Theta_{\mathbf i^\ell}$, containing $g^\ell x$.   When it is clear we also write $\varpi_\ell(x)$ for the  map $\sigma_\ell\circ \theta_{\mathbf i^\ell}$.
\end{definition}

\medskip
\noindent\textbf{Remarks:}\text{ }
\begin{enumerate}
\item The Yomdin partition is in fact only a partition up to a set of zero disk volume, as charts may intersect on their boundaries, a subset of co-dimension at least $1$. This is enough for us, as the only reference measure which we will use will be the disk volume. We will continue to treat it as a partition henceforth.
\item By applying Lemma \ref{bounddis} to the bounded couples $(\sigma_0, \theta_{\mathbf i^0})=(\sigma, \mathrm{Id}_{[0,1]^k})$ and $(\sigma_\ell, \theta_{\mathbf i^l})$ we get the following bounded distortion property for any $x\in D$ and any $\ell\in \mathbb N$:
\begin{equation}\label{bdist}\forall y,z\in g^{-\ell}[\varpi_\ell(x)]\subset D, \ \frac{|\wedge^k dg^\ell(\iota(T_yD) )|}{|\wedge^k dg^\ell(\iota(T_zD) )|}\leq 2.\end{equation}

\end{enumerate}

\begin{lemma}\label{YomdinPtnShrinksUni}
   For every $\ell\leq n$ and $x\in D$,
   $$\mathrm{diam}(g^{-(n-\ell)}[\varpi_{n}(x)])\leq \epsilon 2^{\ell-n}.$$
\end{lemma}
\begin{proof}
Let $\theta_{\mathbf i^{n}_\ell}\in \Theta_{{\mathbf i^{n}_\ell}} $ (resp. $\theta_{\mathbf i^{n}}\in \Theta_{\mathbf i^{n}})$  with $x\in \mathrm{Im}\left(\sigma \circ \theta_{\mathbf i^{n}_\ell} \right)$ (resp. $x\in \mathrm{Im}\left(\sigma \circ \theta_{\mathbf i^{n}} \right)$). We have  $\theta_{\mathbf i^n_\ell}>\theta'_{\mathbf i^n_{\ell +1}}>\cdots>\theta'_{\mathbf i^{n}}>\theta_{\mathbf i^{n}}$. Therefore there is $\phi:[0,1]^k\circlearrowleft$ with $\|d_\cdot \phi\|\leq 2^{\ell-n}$ such that $\theta_{\mathbf i^{n}}= \theta_{\mathbf i^{n}_\ell}\circ \phi$. Observe that $g^{-(n-\ell)}[\varpi_{n}(x)]=\mathrm{Im}(\sigma_{\ell}\circ \theta_{\mathbf{i}^{n}})$.  In particular 
\begin{align*}
\mathrm{diam}(g^{-(n-\ell)}[\varpi_{n}(x)])& \leq \sqrt k\|d_\cdot(\sigma_{\ell}\circ \theta_{\mathbf{i}^{n}})\|\\
&\leq \sqrt k \|d_\cdot(\sigma_{\ell}\circ \theta_{\mathbf{i}^{\ell}})\|\cdot \|d_\cdot\phi\|\\
&
\leq \epsilon   2^{\ell-n}.
\end{align*}

\end{proof}

\section{Proof of Proposition \ref{firstLemmaOfSection}}\label{prop37proof}

Proposition  \ref{firstLemmaOfSection} follows by induction on $n$ from the following lemma.

\begin{lemma}\label{mai}Let $(\sigma, \theta)$ be a strongly $\epsilon$-bounded couple. Then there are maps $\phi:[0,1]^k\circlearrowleft$, disjoint families $\overline{\Theta}_{\phi}$  and $\underline{\Theta}_\phi$  (maybe empty) of maps $\widetilde\theta:[0,1]^k\circlearrowleft$  such that 
\begin{enumerate}
\item $(g\circ \sigma, \phi)$ is a bounded couple for any $\phi$,
\item $(g\circ \sigma, \phi\circ \widetilde\theta)$ is a strongly $\epsilon$-bounded couple for any $\phi$ and any $\widetilde\theta\in \overline{\Theta}_{\phi}\cup\underline{\Theta}_\phi $,
\item $\widetilde \theta<\phi<\theta$ for any $\phi$ and any $\widetilde\theta\in \overline{\Theta}_{\phi}\cup\underline{\Theta}_\phi $, 
\item $\overline{\Theta}_{\phi}$ is a disjoint union of  $(g\circ \sigma, \phi,\epsilon)$-admissible families, each with cardinality less than $ A_{r,d}$,
\item $\coprod^k_{\phi, \  \theta'\in \overline{\Theta}_{\phi}\cup  \underline{\Theta}_\phi } \mathrm{Im}(\phi\circ \widetilde\theta)=\mathrm{Im}(\theta)$,
\item $\# \left(\bigcup_{\phi}\underline{\Theta}_\phi\right)\leq B_{r,d}M_g^{\frac{3k^2}{r-1}} \max_{0\leq l<k}\|\wedge^ld_\cdot  g\|_{\mathrm{Im}(\sigma\circ \theta)}$, for some universal constant $B_{r,d}>0$.
\end{enumerate}
\end{lemma}

\begin{proof}[Proof of Proposition \ref{firstLemmaOfSection} assuming Lemma \ref{mai}]
 The proof is by induction, where the case $n=0$ is trivial. Assume that the statement holds for the index $n$, and we prove it for $n+1$. Let $\mathbf i^n\in \mathcal T_n$ and let $\theta_{\mathbf i^n}\in \Theta_{\mathbf i^n}$. We apply Lemma \ref{mai} to the strongly $\epsilon$-bounded couple $(\sigma_n, \theta_{\mathbf i^n})$. Let $\phi$, $\overline{\Theta}_\phi$, $\underline\Theta_\phi$ be the corresponding reparametrizations. Then the red children of $\mathbf i^n$ (i.e., nodes $\mathbf i^{n+1}$ in $\overline{\mathcal T}_{n+1}$ with $\mathbf i^{n+1}_n=\mathbf i^n$) are given by the $(\sigma_{n+1}, \phi,\epsilon)$-admissible families in $\overline{\Theta}_\phi$ and, in this case, we set $\theta'_{\mathbf i^{n+1}}=\phi$. Then, to any $\theta'\in \underline{\Theta}_\phi$, $\phi$,  we associate a blue child $\mathbf i^{n+1}$ of $\mathbf i^n$ with $\Theta_{\mathbf i^{n+1}}=\{\theta'\}$. Note that $\mathrm{diam}(\sigma_n(\Theta_{\mathbf i^n}))<\epsilon$, in particular $e^{\mathfrak l'(g^nx)}\geq \|\wedge^ld_{g^nx} g\|\geq  \|\wedge^ld_\cdot g\|_{\mathrm{Im}(\sigma_n( \Theta_{\mathbf{i}^n})}/2$
for any $l\leq k$ and for any $x\in \sigma(\Theta_{\mathbf i^n})$. It follows from the last item of Lemma \ref{mai} that 
 \begin{align*}\# \Big\{ \mathbf i^{n+1}\in \underline{\mathcal T}_{n+1}:& \  \mathbf i^{n+1}_{n}= \mathbf i^{n}
    \text{ and } \\
    &\exists x\in \sigma(\Theta_{\mathbf i^{n+1}}),  \ s.t.    \ \mathfrak l'(g^{n} x )=\mathfrak l'_{n+1} \Big\}\\
    &\leq \sum_{\overset{\theta_{\mathbf i^{n}}\in \Theta_{\mathbf i^n}, }{\max_{l<k}\|\wedge^ld g\|_{\mathrm{Im}(\sigma_n\circ \theta_{\mathbf i^n})}\leq 2 e^{\mathfrak l'_{n+1}} }}\# \left(\bigcup_{\phi(\theta_{\mathbf i^{n}})}\underline{\Theta}_{\phi(\theta_{\mathbf i^{n}}) } \right)\\
    &\leq 2A_{r,d} B_{r,d} M_g^{\frac{3k^2}{r-1}}e^{\mathfrak l'_{n+1}}.
 \end{align*}
The upper bound on the number of red children $\mathbf i^{n+1}$ in Proposition \ref{firstLemmaOfSection} follows from the fact that the volume of $\sigma_{n+1}(\Theta_{\mathbf i^{n+1}})$ is bounded from below by $C_d^{-1}\epsilon^k$ as it contains an $\epsilon$-graph in this case (see the remark after Definition \ref{admissibleFamily}).

\end{proof}

The proof of Lemma \ref{mai} involves the following form of Yomdin-Gromov algebraic lemma. Recall $\|\phi\|_r=\max_{1\leq k\leq r}\|d^k_\cdot \phi\|$. 

\begin{lemma}[Algebraic Lemma]\label{alg}\cite{GromovLemma,buryom,wilkie,novikov} Let $P:[0,1]^k\to \mathbb R^d$ be a polynomial map with total degree less than or equal to $r$ and let $Y$ be  a bounded semi-algebraic set of $\mathbb{R}^d$. 
Then there is a constant $B_{r,d}$ depending only $r,d,\mathrm{deg}(Y)$, and $\mathrm{diam}(Y)$, and there are semi-algebraic analytic injective maps $\theta_i:(0,1)^{k_i}\to [0,1]^k$, $k_i\leq k$,  $i\in I$, such that 
\begin{enumerate}
\item $\# I\leq B_{r,d}$, 
\item $\|\theta_i\|_r, \|P\circ \theta_i\|_r\leq \frac{1}{100d}$, 
\item $\coprod_{i\in I}^k\mathrm{Im}(\theta_i)=P^{-1}[Y]$. 
\end{enumerate}
\end{lemma}

\medskip
\noindent\textbf{Remark:}\text{ }
\begin{enumerate}
\item The maps $\theta_i$ may be continuously extended on $[0,1]^{k_i}$ as $\theta_i$ satisfies $\|d_\cdot \theta_i\|\leq \|\theta_i\|_r\leq 1$. 
\item In the following we may only focus on the reparametrizations $\theta_i$ with $k_i=k$ as the image of the others have zero $k$-volume. 
\item By the invariance of domain theorem the image of each map $\theta_i:(0,1)^k\to \mathbb R^k$ is open  and each $\theta_i$ is a homeomorphism onto its image.
\item The boundary of a semi-algebraic set  has zero Lebesgue measure, therefore so does $\theta_i(\partial [0,1]^{k})=\partial \theta_i\left((0,1)^{k}\right)$.
\end{enumerate}

We will make use of the two following well-known multivariate formulas for the derivatives of a product and a composition of $C^r$  functions on $\mathbb R^d$. For  positive integers $m,p,q$ we let $M_{p,q}(\mathbb R)$ be the set of real valued $p\times q$ matrices  and we denote $A\cdot B\in M_{p,m}(\mathbb R)$ the product of two matrices $A\in M_{p,q}(\mathbb R)$ and $B\in  M_{q,m}(\mathbb R)$. We have with the standard multi-index notations:
\begin{itemize}
\item \textit{General Leibniz rule:} Let $u:\mathbb R^d\rightarrow M_{p,q}(\mathbb R)$ and $v:\mathbb R^d\rightarrow M_{q,m}(\mathbb R)$ be $C^r$ maps, then for any $\alpha=(\alpha_1, \cdots,\alpha_d )\in \mathbb N^d$ with $|\alpha|:=\sum_i\alpha_i\leq r$, we have 
\begin{eqnarray}\label{leib}
 \partial^\alpha (u\cdot v)=\sum_{\beta : \beta\leq \alpha}\binom{\alpha}{\beta}(\partial^\beta u)\cdot(\partial^{\alpha-\beta}v).
\end{eqnarray}
\item \textit{ Fa\` a di Bruno's formula:}  Let $u:\mathbb R^d\rightarrow \mathbb R$ and $v=(v_1,\cdots,v_d):\mathbb R^e\rightarrow \mathbb R^d$ be $C^r$ maps, then for any $\alpha\in \mathbb N^e$ with $|\alpha|\leq r$, we have 
\begin{eqnarray}\label{bruno} \partial^\alpha (u\circ v)=\sum_{\beta\in \mathbb N^d, \ |\beta|\leq |\alpha|}(\partial^\beta u)\circ v \times P_\beta\left(\left(\partial^{\gamma}v_i\right)_{ \gamma, i}\right), 
\end{eqnarray}
where $P_\beta\left(\left(\partial^{\gamma}v_i\right)_{ \gamma, i}\right)$ is a universal polynomial, in $\partial^{\gamma}v_i$ for  $i=1,\cdots, d$ and $\gamma \in \mathbb N^e$ with $|\gamma|\leq |\alpha|$, of total degree less than or equal to $|\alpha|$. \\
\end{itemize}

\begin{lemma}\label{compbounded}
Let $(\sigma, \theta)$ be a bounded couple. There is a constant $E_{r,d}>1$ such that  
for any $C^r$ map $\psi:[0,1]^k\circlearrowleft$ with $\|\psi\|_r\leq \frac{1}{E_{r,d}}$, the pair $(\sigma, \theta\circ\psi )$ is a bounded couple.
\end{lemma}

\begin{proof}
By Fa\`a di Bruno's formula, we have for some constant $E_{r,d}>1$ when $\|\psi\|_r\leq 1$
\begin{align*}
\|\wedge^kd_{\phi\circ \psi(\cdot)}\sigma\|_{r-1}&\leq
\frac{E_{r,d}}{\sqrt 2}\|\wedge^kd_{\phi(\cdot)}\sigma\|_{r-1}\|\psi\|_r \\
&\leq 
\frac{E_{r,d}}{10k\sqrt 2}\|\wedge^kd_{\phi(\cdot)}\sigma\|\|\psi\|_r.
\end{align*}
By Lemma \ref{bounddis} we have $$\|\wedge^kd_{\phi(\cdot)}\sigma\|\leq \sqrt 2 \|\wedge^kd_{\phi\circ \psi(\cdot)}\sigma\| $$
Therefore $(\sigma, \theta\circ\psi )$ is a bounded couple whenever $\|\psi\|_r\leq \frac{1}{E_{r,d}}$

\end{proof}
\begin{proof}[Proof of Lemma \ref{mai}]
Without loss of generality we may assume 
$\mathbf M=\mathbb R^d$, $\|g\|_{r}\leq \|d_\cdot g\|$,  $\|\wedge^kd_\cdot g\|_{r-1} \leq \| \wedge^kd_\cdot g\|$ and $\epsilon=1$  by an appropriate rescaling, namely $g\to \epsilon^{-1}g(\epsilon\cdot)$ and $\sigma\to \epsilon^{-1}\sigma$. \\

\fbox{%
\begin{minipage}{0.75\textwidth}
\noindent\textbf{Step 1} : We first explicit a family of maps $\phi$ such that \begin{enumerate}
\item $\coprod_{\phi}^k\mathrm{Im}(\phi)=\mathrm{Im}(\theta)$,
\item $(g\circ \sigma,\phi)$ is a bounded couple, 
\item $\|d^r_\cdot(g\circ \sigma\circ \phi)\|\leq 1/4$,
\item  $\# \{\phi\}\leq C_{r,d}M_g^{\frac{3k^2}{r-1}}$. 
\end{enumerate}
\end{minipage}
}\\
In the next computations  we write $a\lesssim b$ to mean $a\leq C_{r,d}\, b$ for some constant $C_{r,d}$ depending only on $r$ and $d$. Let 
$\Gamma:t\mapsto \wedge^k d_{\theta(t)}(g\circ \sigma)$ 
and let $s$ with $|\Gamma(s)|=\inf_t| \Gamma(t)|$.  We have 
\begin{align*}
\|d^{r-1}_\cdot \Gamma\|&\leq \|d^{r-1}_\cdot (\wedge^k d_{\sigma\circ \theta(t)}g\cdot  \wedge^k d_{\theta(t)}\sigma)\|\\
&\lesssim \|\wedge^k d_{\sigma\circ \theta(\cdot)}g\|_{r-1}\| \wedge^k d_{\theta(\cdot)}\sigma\|_{r-1}\text{ by Leibniz rule \eqref{leib}}\\
&\lesssim \|\wedge^k d_{\sigma\circ \theta(\cdot)}g\|_{r-1} | \wedge^k d_{\theta(s)}\sigma| \text{, $\because$ $(\sigma, \theta)$ is a bounded couple}\\
&\lesssim \|\wedge^k d_\cdot g\|_{r-1}\max\{1,\|\sigma\circ \theta\|_r\}^{r-1} | \wedge^k d_{\theta(s)}\sigma| \text{ by 
\eqref{bruno}}\\
&\lesssim \|\wedge^k d_\cdot g\|_{r-1}| \wedge^k d_{\theta(s)}\sigma| \text{,  $\because$ $(\sigma, \theta)$ is a strongly  $1$-bounded couple} \\
&\lesssim \|\wedge^k d_\cdot g\| | \wedge^k d_{\theta(s)}\sigma| \  \text{by the choice of $\epsilon$}\\
&\lesssim \|d_\cdot g\|^{k} | \wedge^k d_{\theta(s)}\sigma|\\
&\lesssim \|d_\cdot g\|^{k}\|d_\cdot g^{-1}\|^k |\wedge^k d_{\theta(s)}(g\circ \sigma)|\\
&\lesssim M_g^{2k}\inf_t| \Gamma(t)|.
\end{align*}

We partition $[0,1]^k$ into subcubes of size $\lesssim \left(M_g^{2k}\right)^{\frac{-1}{r-1}}$. By composing $\theta$ with an affine reparametrization of one of these sub-cubes, we get  maps $\psi$ satisfying $\|d^r(g\circ \sigma\circ \psi)\|\leq 1/4$ and 
$\|d^{r-1}(\Gamma\circ \psi)\|\lesssim \inf_t| \Gamma(t)|$. Let $P$ be the interpolation polynomial of $\Gamma\circ \psi$ at $0$. Note that
$$\|\Gamma\circ \psi-P\|_{r-1}\lesssim \inf_t| \Gamma(t)|.$$

We apply the Algebraic Lemma (Lemma \ref{alg}) to each $P_\mathfrak l:=\frac{P(t)}{e^\mathfrak l\|\wedge^k d_{\psi(\cdot)} \sigma)\|} $,  $\mathfrak l\in \mathbb N^*$, with $Y=B(0,1) \setminus B(0,1/e)$ (recall $B(0,a)$ denotes the Euclidean ball of radius $a>0$). Let $\theta^\mathfrak l_i$ be the obtained reparametrizations. 
We have
\begin{align}\label{appor} \|(\Gamma\circ \psi-P)\circ\theta^{\mathfrak l}_i\ \|_{r-1}&\lesssim \|\Gamma\circ \psi-P\|_{r-1}\max\{1, \|\theta_i^\mathfrak l\|_r\}^{r-1}
\\
&\leq \frac{\inf_t| \Gamma(t)|}{100d}\end{align}

and 
\begin{align}\label{appor1}\forall t, \ e^{\mathfrak l-1}\|\wedge^k d_{\psi(\cdot)}\sigma\|\leq |P\circ \theta_{i}^\mathfrak l(t)|\leq e^\mathfrak l\| \wedge^k d_{\psi(\cdot)}\sigma\|.
\end{align}
From \eqref{appor} we get $$ \forall t, \ \frac{10}{11}|P\circ \theta_{i}^\mathfrak l(t)|\leq |\Gamma\circ \psi\circ \theta_i^\mathfrak l(t)| \leq \frac{10}{9} |P\circ \theta_{i}^\mathfrak l(t)|$$
then by using \eqref{appor1}
$$\forall t,\ e^{\mathfrak l-2}\|\wedge^kd_{\psi(\cdot)} \sigma\|\leq |\Gamma\circ \psi\circ \theta_i^\mathfrak l(t) |\leq e^{\mathfrak l+2} \|\wedge^k d_{\psi(\cdot)}\sigma\|.$$

Finally we obtain 
\begin{align*}\|\Gamma\circ \psi\circ \theta^{\mathfrak l}_i\|_{r-1}&\leq \|P\circ \theta^\mathfrak l_i\|_{r-1}+\|(\Gamma\circ \psi-P)\circ\theta^{\mathfrak l}_i\|_{r-1}\\
&\leq e^\mathfrak l \frac{\|\wedge^k d_{\psi(\cdot)}\sigma\|}{100d}+\frac{\inf_t| \Gamma(t)|}{100d}\\
&\leq  \frac{\|\Gamma\circ\psi\circ \theta_i^\mathfrak l\| e^{ 2}}{100d}+\frac{\inf_t| \Gamma(t)|}{100d}\\
&\leq \frac{\|\Gamma\circ\psi\circ \theta_i^\mathfrak l\|}{10k}.
\end{align*}

This proves $(g\circ \sigma, \psi\circ \theta_i^\mathfrak l)$ is a bounded couple.  Finally we check: 
 \begin{align*}
 \|d^r_\cdot(g\circ \sigma \circ \psi\circ \theta_i^\mathfrak l)\|&\lesssim\|g \|_r \max\{1,\|\sigma\circ \theta\|_r\}^r
 \text{ by  
\eqref{bruno}} \\
 &\lesssim\|g \|_r \text{, $\because$$(\sigma, \theta)$ is a strongly  $1$-bounded couple} \\
& \lesssim\|d_\cdot g\|  \text{ by the choice of $\epsilon$.}
 \end{align*}

 By composing $\theta_{i}^\mathfrak l$ with affine contractions $\rho$ of rate $\lceil \|d_\cdot g\|^{-\frac{1}{r}}\rceil$ as above, we get maps $\phi=\theta\circ \psi\circ \theta_i^\mathfrak l\circ \rho$ satisfying the three first items of Step 1. Finally observe that 
 \begin{align*}\#\{\phi\}&\leq \#\{\psi\}\cdot  \#\{\theta_i^\mathfrak l\}\cdot  \#\{\rho\}\\
 &\lesssim M_g^{\frac{2k^2}{r-1}}\lceil\|d_\cdot g\|^{\frac{1}{r}}\rceil^k\\
 &\lesssim M_g^{\frac{3k^2}{r-1}}.
 \end{align*}

\fbox{%
\begin{minipage}{0.75\textwidth}
\noindent\textbf{Step 2:} 
We construct the families $\underline{\Theta}_\phi$ and estimate their cardinality. \end{minipage}
} \\

We set $\widetilde{\sigma}=g\circ \sigma\circ \phi$.   
Let $\mathcal C$  be the partition of $\mathbb{R}^d$ into cubes of size $1$ with vertices in $\mathbb Z^d$. We may assume that there is a $k$-face $F_k(C)$ of each  cube $C$ which is tangent to  
$\mathrm{Im}(d_{\phi(x_C)} (g\circ \sigma))$ 
with $x_C$ being the center of $C$. Without loss of generality we may assume this face is given by $[0,1]^k\times \{0\}^{d-k}+a$ for some $a\in \mathbb Z^d$.  For $C\in \mathcal C$ and $q\in \mathbb N$ we let $C^q$ be the cube of size $2q+1$ centered at $C$. 

   Let $Q$ be the interpolation polynomial of $\widetilde{\sigma}$ at $0$. Recall $\|d^r\widetilde{\sigma}\|\leq 1/4$.  Then $\|\widetilde{\sigma}-Q\|\leq 1$, in particular 
for any $C\in \mathcal C$ we have  $$\widetilde{\sigma}^{-1}C\subset Q^{-1}C^1\subset \widetilde{\sigma}^{-1}C^2.$$
 We apply now the Algebraic   Lemma to $Q$ with $Y$ being each cube $C$ in $\mathcal C$. Let $\Theta_C$ be the family of reparametrizations  obtained in this way. Note that $\|\widetilde\sigma\circ \theta\|_r\leq 1/4$. Without loss of generality we may also assume $\|\theta\|_r\leq \frac{1}{E_{r,d}}$   so that  $(g\circ\sigma, \phi\circ \theta)$ is a strongly $1$-bounded couple for any $\theta\in \Theta_C$ by Lemma \ref{compbounded}.

We let   $\partial \mathrm{Im}(\widetilde{\sigma})$ be the boundary of the disk $\mathrm{Im}(\widetilde{\sigma})$. Note that $\partial \mathrm{Im}(\widetilde{\sigma})=\widetilde{\sigma}(\partial [0,1]^k)$. Then we have 
 \begin{align*}
&\# \{C\in \mathcal C: \ C\cap \mathrm{Im}(\widetilde{\sigma})\neq \varnothing \text{ and } C^3\cap \partial \mathrm{Im}(\widetilde{\sigma})\neq \varnothing  \}\\
&\lesssim \mathrm{Cov}(\mathrm{Im}(\widetilde{\sigma}|_{\partial [0,1]^k }),1)\\
&\lesssim\max_{0\leq l<k}\|\wedge^ld_\cdot \widetilde \sigma\|, \text{ by Lemma \ref{cover}}\\
 &\lesssim \max_{0\leq l<k}\|\wedge^ld_\cdot g\|_{\mathrm{Im}(\sigma\circ \phi)}, \text{\  $\because \|d_\cdot (\sigma\circ\phi)\|\leq 1$.}
 \end{align*}

We let $\underline{\Theta}_{\phi}$ be the union of $\Theta_C$ over $C\in \mathcal C$ with  $C\cap \mathrm{Im}(\widetilde{\sigma})\neq \varnothing$  and $C^3\cap \partial \mathrm{Im}(\widetilde{\sigma})\neq \varnothing$. \\
 
\fbox{%
\begin{minipage}{0.75\textwidth}
\noindent\textbf{Step 3:} for $\Theta_C\ni \theta \notin  \underline{\Theta}_{\phi}$, the intersection $\mathrm{Im}(\widetilde \sigma )\cap C$ is contained in graphs of $1$-Lipschitz maps $\psi_j, j\in   J_C$  over $F_k(C)$.\end{minipage}
} \\

Let $C\in \mathcal C$ with  $C\cap \mathrm{Im}(\widetilde{\sigma})\neq \varnothing$ and $C^3\cap \partial \mathrm{Im}(\widetilde{\sigma})= \varnothing$. 
Recall that $\widetilde{\sigma}=g\circ \sigma\circ \phi$ with $(g\circ \sigma, \phi)$ being a bounded couple. By Lemma \ref{angle} and Lemma \ref{graph},  the disk $
\mathrm{Im}(\widetilde \sigma )$ is locally at $x\in \mathrm{Im}(\widetilde \sigma 
)\cap C$  a graph of a $1$-Lipschitz map $g:U\subset \mathbb R^k\to \mathbb R^{n-k}$, which may be extended on 
$F_k(C)$  as $C^3\cap \partial \mathrm{Im}(\widetilde{\sigma})= \varnothing$. 
Then $\mathrm{Im}(\widetilde \sigma )\cap C$ is contained in the disjoint  union of 
such graphs. These graphs are not contained in $C$ a priori but in $C^3$. \\

\fbox{%
\begin{minipage}{0.75\textwidth}\noindent\textbf{Step 4: }We construct the admissible families in $\overline{\Theta}_\phi$. \end{minipage} }\\

We enumerate the cubes  $C$ satisfying  $C^3\cap \partial \mathrm{Im}(g\circ \sigma\circ \theta)= \varnothing$ as $C_1,C_2,\cdots,C_N$. For any $i$, we let $\psi_j, j\in   J_i$ be the Lipschitz maps over $F_k(C_i)$  given by the previous step 
whose disjoints  graphs cover (and intersect) $\mathrm{Im}(\widetilde \sigma)\cap C_i$.

We continue to construct the admissible families by induction on $i=1,...,N$. 
For any $j\in J_1$ we let $\Theta'_{j,1}$ be the union of $\theta\in \bigcup_C\Theta_C$ with $\mathrm{Im}(\widetilde\sigma \circ \theta)\cap \Gamma_{\psi_j}\neq \varnothing$. Assume for simplicity that $C_1=[0,1]^d$. Then $\widetilde{\sigma}(\Theta'_{j,1})$ is contained in $[-1/4, 5/4]^k\times \mathbb R^{d-k}$. The number of such $\theta$'s is bounded by $\sum_{C\in \mathcal C, \ C\cap C_1^3\neq \varnothing}\# \Theta_C\leq 7^dB_{r,d}=:A_{r,d}$. The families $ \Theta'_{j,1}$ are $(g\circ \sigma,\phi,\epsilon)$-admissible family. 

Then for any $j\in J_2$ we let $\Theta'_{j,2}$ be the union of $\theta\in \bigcup_C\Theta_C\setminus (\bigcup_{j\in J_1 } \Theta'_{j,1})$ with $\mathrm{Im}(\widetilde{\sigma}\circ \theta)\cap \Gamma_{\psi_j}\neq \varnothing$ and so on. Observe that for $j\in J_2$, the graph $\Gamma_{\psi_j}$ may intersect the image of the previously built admissible family $ \Theta'_{j',1}$, $j'\in J_1$ (assume for example  $C_2=[1,2]\times [0,1]^{d-1}$). However $\widetilde{\sigma}(\Theta'_{j,2})$ still contain the graph of $\psi_j$ over $[5/4, 7/4]\times [0,1]^{k-1}\times \{0\}^{d-k}$.  One may continue this way by induction to construct the required $(g\circ \sigma, \phi,\epsilon)$-admissible families $ \Theta'_{j,i}$, $j\in J_i$, $i=1,\cdots, N$.

\end{proof}

\section{Positive density of geometric times}\label{posDensOfGeoTimes}
We consider a $C^r$ diffeomorphism  $f:\mathbf M \circlearrowleft$ with $\Vol([\lambda_k>\frac{3k^2}{r-1}R(f)])>0$ (as assumed in Theorem \ref{crcr}). 
 Let $p$, $D$ and $\chi>a=\frac{3k^2}{r-1}R(f)$ given by Lemma \ref{theDiskD}. Without loss of generality we can assume that $D=\mathrm{Im}(\sigma)$ where  $(\sigma, \mathrm{Id}_{[0,1]^k})$ is a strongly $\epsilon$-bounded couple (by composing $\sigma$ with an affine contraction if necessary). 
We consider the tree $\mathcal T$ associated  to $g=f^p$ and to the $k$-disc $D$ given by Proposition \ref{firstLemmaOfSection}. 

\begin{definition}
For all $n\in \mathbb N^*\cup\{\infty\}$ and for all $x\in D$ we  define: $$\mathfrak l^n( x)=(\mathfrak l_i(\widetilde x), \mathfrak l'_i(x))_{0\leq i<n}=(\mathfrak l(G^{i}\widetilde 
x),\mathfrak l'(g^ix) )_{0\leq i<n}.$$
\end{definition}

\begin{lemma}\label{lebol}
For any $\mathbf i^n\in \mathcal T_n$ we have
\begin{enumerate} 
\item If $x,y\in \mathrm{Im}(\sigma\circ \theta'_{\mathbf i^n})$, then   $|\mathfrak l_{n-1}(\widetilde x)-\mathfrak l_{n-1}(\widetilde y)|\leq 1$,
\item if $x,y\in \sigma(\Theta_{\mathbf i^n})$, then $|\mathfrak l_i(\widetilde x)-\mathfrak l_i(\widetilde y)|,|\mathfrak l'_i(x)-\mathfrak l'_i(y)|\leq 1$ for all $0\leq i<n$.
\end{enumerate}
\end{lemma}
\begin{proof}\text{ }

\begin{enumerate}
\item There is $\theta_{\mathbf i^n_{n-1}}\in \Theta_{\mathbf i^n_{n-1}}$ such that $\theta'_{\mathbf i^n}<\theta_{\mathbf i^n_{n-1}}$. 
 Observe that $(\sigma_{n}, \theta'_{\mathbf i^n})$ and $(\sigma_{n-1}, \theta_{\mathbf i^n_{n-1}})$ are bounded couples. Therefore, by the bounded distortion property, we have for any $t,s\in \mathrm{Im}(\theta'_{\mathbf i^n})$
 $$\frac{|\wedge^k d_{t} \sigma_n|}{|\wedge^k d_{s} \sigma_n|}<\sqrt 2\text{ and }\frac{|\wedge^k d_{t} \sigma_{n-1}|}{|\wedge^k d_{s} \sigma_{n-1}|}<\sqrt 2,$$
therefore we get with $x=\sigma(t), y=\sigma(s)$ \begin{equation}\label{fran}\frac{|\wedge^kdg (G^{n-1}(\widetilde x))|}{|\wedge^kdg (G^{n-1}(\widetilde y))|}\leq 2\ (<e),
\end{equation}
 finally $$|\mathfrak l_{n-1}(\widetilde x)-\mathfrak l_{n-1}(\widetilde y)|\leq 1.$$ 
 \item For any $x,y\in \sigma(\Theta_{\mathbf i^n})$, the points $g^{i}x$ and $g^iy$ are $\epsilon$-close for any $0\leq i<n$. According to the choice of $\epsilon$ in (\ref{smallosc}), we have

 $$\frac{\max_{j<k}\|\wedge^{j}d_{g^ix} g \|\vee 1}{\max_{j<k}\|\wedge^{j}d_{g^iy} g \|\vee 1}\leq 2,$$
 therefore $$\forall i<n,\ |\mathfrak l'_{i}( x)-\mathfrak l'_{i}(y)|\leq 1.$$
 The inequalities $|\mathfrak l_{i}( \widetilde x)-\mathfrak l_{i}(\widetilde y)|\leq 1$ follows from item (1) as we have $\theta_{\mathbf i^n}<\theta'_{\mathbf i^n_i}$ for any $\theta_{\mathbf i^n}\in \Theta_{\mathbf i^n}$.
 \end{enumerate}
\end{proof}

\begin{definition}
For a parameter $\widehat{\chi}>0$, $n\in \mathbb N^*\cup\{\infty\}$ and  a  sequence $\mathfrak l^n=(\mathfrak l_i,\mathfrak l'_i)_{0\leq i<n}\in \mathbb R^{2n}$, an integer $\ell\in [0,n]$   is said to be a $(\widehat{\chi}, \mathfrak l^n)$-hyperbolic time  when $$\forall 0\leq i\leq \ell, \ \sum_{i\leq j<\ell} \mathfrak l_j-\mathfrak l'_j\geq (\ell-i)\widehat{\chi}.$$
\end{definition}

\begin{lemma}\label{BGtimesAreHyp}
  Let $x\in D$ and $\hat \chi>0$. For any $(\widehat{\chi}, \mathfrak l^\infty(x))$-hyperbolic time $\ell$, we have 
  $$\forall v\in T_{g^i x}D_i \ \forall 0\leq i\leq \ell, \ |d_{g^{i}x}g^{\ell-i}v |\geq e^{(\ell-i)\widehat{\chi}}|v|. $$
\end{lemma}

\begin{proof}
We have for all $0\leq j< \ell$ (recall \textsection \ref{YomdinPtns}),

\begin{align}\label{estim}
e^{\mathfrak l(G^j\widetilde x)-\mathfrak l'(g^jx)}&\leq \frac{|\wedge^kd_{g^jx}g(G^j\widetilde x)| }{\max_{k'<k}\|\wedge^{k'}d_{g^jx} g \|\vee 1}\\
&\leq\frac{|\wedge^kd_{g^jx}g(G^j\widetilde x)| }{\max_{k'<k}\|\wedge^{k'}d_{g^jx} g \|}\nonumber\\
&\leq \frac{|\wedge^kd_{g^jx}g(G^j\widetilde x)| }{\|\wedge^{k-1}d_{g^jx} g|_{T_{g^jx}D_j} \|}\nonumber\\
&\leq \inf\{|d_{g^jx}g(u)|: \ u\in T_{g^jx}D_j \text{ and } \ |u|=1\}.\nonumber
\end{align}
Observe then that 
\begin{align}\label{estim222}
|d_{g^{i}x}g^{\ell-i}v |&\geq |v|\cdot \prod_{i\leq j<\ell}\inf\{|d_{g^jx}g(u)|: \ u\in T_{g^jx}D_j \text{ and } \ |u|=1\} \nonumber\\
&\geq|v|\cdot \prod_{i\leq j<\ell}e^{\mathfrak l_j(x)-\mathfrak l'_j(x)}, \text{ }\because \text{\eqref{estim}} \nonumber\\
&\geq|v|\cdot e^{(\ell-i)\widehat{\chi}}, \text{ }\because \ell\text{ is a }(\widehat{\chi}\text, \mathfrak l^\infty(x))\text{-hyperbolic \ time}.
\end{align}

\end{proof}

\medskip
\noindent\textbf{Remark:} By Lemma \ref{lebol}, if $x,y\in \sigma(\Theta_{\mathbf i^n})$ for some $\mathbf i^n\in \mathcal T_n$, then $|\mathfrak l_i(x)-\mathfrak l_i(y)|,|\mathfrak l'_i(x)-\mathfrak l'_i(y)|\leq 1$,  for all $i<n$. Therefore, if $\ell\in [0,n]$ is $(\widehat{\chi}, \mathfrak l^\infty (x))$-hyperbolic then $\ell$ is  $(\widehat{\chi}-2, \mathfrak l^\infty (y))$-hyperbolic. We apply this remark later with $\widehat \chi=p\chi+2$.

\medskip
We recall Pliss Lemma in this context.
\begin{lemma}\label{Pliss}
Let $\mathfrak l^n=(\mathfrak l_i,\mathfrak l'_i)_{0\leq i<n}$ be a sequence with $\sum_{0\leq i<n}\mathfrak l_i-\mathfrak l'_i\geq n\widehat{\chi}'>0$. Then for any $0<\widehat{\chi}<\widehat{\chi}'$ we have for $A=\sup_i(\mathfrak l_i-\mathfrak l'_i)\leq \log\|\wedge^kd_\cdot g\|$,
$$\frac{1}{n}\# \Big\{0\leq \ell<n\ : \ \text{$\ell $ is $(\widehat{\chi}, \mathfrak l^n)$-hyperbolic}\Big\}\geq \frac{\widehat{\chi}'-\widehat{\chi}}{A-\widehat{\chi}}.$$
\end{lemma}

\medskip

\begin{definition}
For all $x\in D$ and all $n\in \mathbb  N$ we  define: 
\begin{align*}
\mathfrak u_n(x)&=1 \text{ if $x\in \sigma(\Theta_{\mathbf i^n})$ with $\mathbf i^n\in \overline{\mathcal T}_n$},\\
&=0  \text{ if not.}
\end{align*}

Given $n\in \mathbb N$, we set
$$\mathfrak u^n(   
x)=\left(\mathfrak u_i( x)\right)_{0\leq i<n}.$$

\end{definition}

\medskip
\noindent\textbf{Notation:} 
\begin{enumerate}
\item For two sequences $\mathfrak l^n$ and $\mathfrak u^n$ and for  $0\leq \ell <n$ we denote by $\mathfrak  l^n_\ell$ and $\mathfrak u^n_\ell$ respectively the sequence of suffices given by the $n-\ell$ last terms of $\mathbf l^n$ and $\mathfrak u^n$. Note that  $\mathfrak  l^n_0=\mathfrak  l^n$ 
and $\mathfrak  u^n_0=\mathfrak  u^n$.
\item To simplify the notations, we write $\sigma_{\mathbf i^\ell}$ for $\sigma(\Theta_{\mathbf i^\ell})$ and $\widetilde{\sigma}_{\mathbf i^\ell}$ for $g^\ell\left(\sigma_{\mathbf i^\ell}\right)=\sigma_\ell(\Theta_{\mathbf i^\ell})$.
\item We also denote by $\Vol_{\overline{\sigma}}$ the normalized volume probability measure 
  on a  disjoint finite union $\overline\sigma$ of disks. 
\end{enumerate}

\medskip

\begin{definition}
For a $2n$-tuple of   integers $\mathfrak  
l^n=(\mathfrak l_i,\mathfrak l'_i)_{0\leq i<n}$ and for $0\leq \ell<n$
we consider then $$\mathcal H(\mathfrak l_\ell^n):=g^\ell\left(\left\{x\in D \ :\mathfrak l^n_\ell( x)=\mathfrak  l^n_\ell\right\}\right).$$
\end{definition}

\begin{definition}
For a $n$-tuple    $\mathfrak  
u^n=\left(\mathfrak u_i\right)_{0\leq i<n}\in \{0,1\}^n$  and for $0\leq \ell<n$ 
we consider  $$\mathcal B(\mathfrak u_\ell^n):=g^\ell\left(\left\{x\in D \ :\mathfrak u_\ell^n(  x)=\mathfrak  u^n_\ell\right\}\right).$$

\end{definition}

The following lemma is a key estimate in the proof of Proposition \ref{lebgeo}.
\begin{lemma}\label{hnotbg} Given $\mathfrak  l^n$ and $\mathfrak u^n$, we consider a triple of  integers  $(\ell_0,\ell_1,\ell_2)$ with  $0\leq \ell_0 < \ell_1\leq \ell_2<n$ satisfying :
\begin{enumerate}
\item $\mathfrak u_{\ell_0}=1$,
\item $\ell_1$ is $(\widehat{\chi}, \mathfrak l^n)$-hyperbolic, 
\item $\mathfrak u_{\ell}=0$  for any   $\ell\in [\ell_0+1, \ell_1]$.
\end{enumerate}
Then, for some universal constant $D_{r,d}>0$ we have   
\begin{align*}
    &\sup_{\mathbf i^{\ell_0} }\Vol_{\widetilde{\sigma}_{i^{\ell_0}}}\left( \mathcal H(\mathfrak l_{\ell_0}^n)\cap \mathcal B(\mathfrak u_{\ell_0}^n)\right)\\
    &\leq D_{r,d}^{\ell_1-\ell_0}M_g^{\frac{3k^2(\ell_1-\ell_0)}{r-1}} e^{-(\ell_1-\ell_0)\widehat{\chi}}\sup_{\mathbf i^{\ell_2}} \Vol_{\widetilde{\sigma}_{\mathbf i^{\ell_2}}}\left( \mathcal H(\mathfrak  l^{n}_{\ell_2})\cap \mathcal B(\mathfrak u^{n}_{\ell_2})\right).
\end{align*}
\end{lemma}
\begin{proof}
Observe that 
\begin{align*}
&\Vol_{\widetilde{\sigma}_{\mathbf i^{\ell_1}}}\left( \mathcal H(\mathfrak  l^{n}_{\ell_1})\cap \mathcal B(\mathfrak u^{n}_{\ell_1})\right)\\
&\leq  \sum_{\mathbf i^{\ell_2}: \ \mathbf i^{\ell_2}_{\ell_1}=\mathbf{i}^{\ell_1}}\frac{\Vol_k(g^{\ell_1}(\sigma_{\mathbf i^{\ell_2}}))}{\Vol_k(\widetilde{\sigma}_{\mathbf i^{\ell_1}})} \frac{\Vol_k\left(g^{\ell_1}(\sigma_{\mathbf i^{\ell_2}})\cap \mathcal  H(\mathfrak  l^{n}_{\ell_1})\cap \mathcal B(\mathfrak u_{\ell_1}^{n})\right)}{\Vol_k(g^{\ell_1}(\sigma_{\mathbf i^{\ell_2}}))}.\\
\end{align*}
It follows from  the  bounded distortion property that for all $\mathbf i^{\ell_2}$ we have 
$$ \frac{\Vol_k\left(g^{\ell_1}(\sigma_{\mathbf i^{\ell_2}})\cap \mathcal H(\mathfrak l^{n}_{\ell_1})\cap \mathcal B(\mathfrak u_{\ell_1}^{n})\right)}{\Vol_k(g^{\ell_1}(\sigma_{\mathbf i^{\ell_2}}))}\leq 4 \frac{\Vol_k\left(g^{\ell_2}(\sigma_{\mathbf i^{\ell_2}})\cap \mathcal H(\mathfrak l^{n}_{\ell_2})\cap \mathcal B(\mathfrak u_{\ell_2}^{n})\right)}{\Vol_k(g^{\ell_2} (\sigma_{\mathbf i^{\ell_2}}))}.$$
Therefore, 
\begin{align*}
\Vol_{\widetilde{\sigma}_{\mathbf i^{\ell_1}}}\left( \mathcal H(\mathfrak l^{n}_{\ell_1})\cap \mathcal B(\mathfrak u^{n}_{\ell_1})\right) 
&\leq 4\sum_{\mathbf i^{\ell_2}: \ \mathbf i^{\ell_2}_{\ell_1}=\mathbf{i}^{\ell_1}}\frac{\Vol_k(g^{\ell_1}(\sigma_{\mathbf i^{\ell_2}}))}{\Vol_k(\widetilde{\sigma}_{\mathbf i^{\ell_1}})}  \Vol_{\widetilde{\sigma}_{\mathbf i^{\ell_2}}}\left(\mathcal H(\mathfrak l^{n}_{\ell_2})\cap \mathcal B(\mathfrak u_{\ell_2}^{n})\right)\\
&\leq 4\sup_{\mathbf i^{\ell_2}}\Vol_{\widetilde{\sigma}_{\mathbf i^{\ell_2}}}\left(\mathcal H(\mathfrak l^{n}_{\ell_2})\cap \mathcal B(\mathfrak u_{\ell_2}^{n})\right).
\end{align*}
To conclude the desired statement, we are only left to show that for some constant $D_{r,d}$
\begin{align*}
    &\sup_{\mathbf i^{\ell_0} }\Vol_{\widetilde{\sigma}_{i^{\ell_0}}}\left( \mathcal H(\mathfrak l_{\ell_0}^n)\cap \mathcal B(\mathfrak u_{\ell_0}^n)\right)\\
    &\leq \frac{D_{r,d}^{\ell_1-\ell_0}}{4}M_g^{\frac{3k^2(\ell_1-\ell_0)}{r-1}} e^{-(\ell_1-\ell_0)\widehat{\chi}}\sup_{\mathbf i^{\ell_1}} \Vol_{\widetilde{\sigma}_{\mathbf i^{\ell_1}}}\left( \mathcal H(\mathfrak l^{n}_{\ell_1})\cap \mathcal B(\mathfrak u^{n}_{\ell_1})\right).
\end{align*}

 For any $\ell \in [\ell_0+1,\ell_1]$ and for any $g^{\ell_0}x\in \mathcal B(\mathfrak u_{\ell_0}^n)$, we have $\mathfrak u_\ell(x)=0$. Therefore  by Proposition \ref{firstLemmaOfSection},
  \begin{align*}
      &\#\{ \mathbf i^{\ell_1}:\ \mathbf i^{\ell_1}_{\ell_0}=\mathbf i^{\ell_0} \text{ and } g^{\ell_0} (\sigma_{\mathbf i^{\ell_1}})\cap\mathcal H(\mathfrak l_{\ell_0}^n)\cap \mathcal B(\mathfrak u_{\ell_0}^n)\neq \varnothing \}\\
      &\leq \left(C_{r,d}M_g^{\frac{3k^2}{r-1}}\right)^{\ell_1-\ell_0}e^{\sum_{\ell_0<j\leq\ell_1}\mathfrak l'_j}.
  \end{align*}
Then, in particular,

\begin{align*}
\Vol_k&\left(\widetilde{\sigma}_{\mathbf i^{\ell_0}}\cap  \mathcal H(\mathfrak l_{\ell_0}^{n})\cap \mathcal B(\mathfrak u^{n}_{\ell_0})\right)\leq \sum_{\mathbf i^{\ell_1}:\ \mathbf i^{\ell_1}_{\ell_0}=\mathbf i^{\ell_0} } \Vol_k(g^{\ell_0}(\sigma_{\mathbf i^{\ell_1}})\cap  \mathcal H(\mathfrak l_{\ell_0}^{n})\cap \mathcal B(\mathfrak u^{n}_{\ell_0}))\\
&\leq 4\left(2C_{r,d}M_g^{\frac{3k^2}{r-1}}\right)^{\ell_1-\ell_0} e^{\sum\limits_{\ell_0<j\leq \ell_1}\mathfrak l'_j}\sup_{\mathbf{i}^{\ell_1}}\Vol_k(g^{\ell_0}(\sigma_{\mathbf i^{\ell_1}})\cap  \mathcal H(\mathfrak l_{\ell_0}^{n})\cap \mathcal B(\mathfrak u^{n}_{\ell_0}))\\
&\leq 4\left(4C_{r,d}M_g^{\frac{3k^2}{r-1}}\right)^{\ell_1-\ell_0} e^{\sum\limits_{\ell_0<j\leq \ell_1}\mathfrak l'_j-\mathfrak l_j}\sup_{\mathbf{i}^{\ell_1}}\Vol_k\left(\widetilde{\sigma}_{\mathbf i^{\ell_1}}\cap  \mathcal H(\mathfrak  l_{\ell_1}^{n})\cap \mathcal B(\mathfrak u^{n}_{\ell_1})\right)\\
&\leq  4\left(4C_{r,d}M_g^{\frac{3k^2}{r-1}}\right)^{\ell_1-\ell_0}  e^{(\ell_0-\ell_1)\widehat{\chi} }\sup_{\mathbf{i}^{\ell_1}}\Vol_k\left(\widetilde{\sigma}_{\mathbf i^{\ell_1}}\cap  \mathcal H(\mathfrak l_{\ell_1}^{n})\cap \mathcal B(\mathfrak u^{n}_{\ell_1})\right),\\ 
&\because \ell_1\text{ is }\widehat{\chi}\text{-hyperbolic}.
\end{align*}
As $\mathfrak u_{\ell_0}=1$, we have $\mathbf i^{\ell_0}$ belongs to $\overline{\mathcal T}_{\ell_0}$. By the remark following Definition \ref{admissibleFamily}, we have  $\Vol_k(\widetilde{\sigma}_{\mathbf i^{\ell_0}})\geq C_d^{-1} \epsilon^k\geq C_d^{-2}\Vol_k(\widetilde{\sigma}_{\mathbf i^{\ell_1}}) $ for some  constant $C_d$,  so that we finally get for any $i^{\ell_0}$ :
\begin{align*}
 \Vol_{\widetilde{\sigma}_{\mathbf i^{\ell_0}}} \left( \mathcal H(\mathfrak l_{\ell_0}^{n})\cap \mathcal B(\mathfrak u^{n}_{\ell_0})\right)
&=\frac{\Vol_k\left(\widetilde{\sigma}_{\mathbf i^{\ell_0}}\cap \mathcal  H(\mathfrak l_{\ell_0}^{n})\cap \mathcal B(\mathfrak u^{n}_{\ell_0})\right)}{\Vol_k(\widetilde{\sigma}_{\mathbf i^{\ell_0}})}\\
&\leq  4C_d\epsilon^{-k} \left(4C_{r,d}M_g^{\frac{3k^2}{r-1}}\right)^{\ell_1-\ell_0}  \\
&\cdot e^{(\ell_0-\ell_1)\widehat{\chi} }\sup_{\mathbf{i}^{\ell_1}}\Vol_k\left(\widetilde{\sigma}_{\mathbf i^{\ell_1}}\cap \mathcal   H(\mathfrak l_{\ell_1}^{n})\cap \mathcal B(\mathfrak u^{n}_{\ell_1})\right)\\
&\leq 4C_{d}^2\left(4C_{r,d}M_g^{\frac{3k^2}{r-1}}\right)^{\ell_1-\ell_0}  \\
&\cdot e^{(\ell_0-\ell_1)\widehat{\chi} }\sup_{\mathbf{i}^{\ell_1}}\Vol_{\widetilde{\sigma}_{\ell_1}}\left(  \mathcal H(\mathfrak l_{\ell_1}^{n})\cap \mathcal B(\mathfrak u^{n}_{\ell_1})\right).
\end{align*}
This concludes the proof with $D_{r,d}=4^3C_d^2C_{r,d}$. 
\end{proof}

\medskip
\begin{definition}
For all sequences $\mathfrak l^n$ and $\mathfrak u^n$, we set  with $\widehat \chi:=p \chi+2$
 $$ H(\mathfrak l^n):=\{\ell \in [0,n): \ \ell \text{ is 
 $(\widehat{\chi},\mathfrak l^n)$-hyperbolic}\}$$
 and $$\mathrm{BG}(\mathfrak u^n):=\{\ell\in [0,n):\ \mathfrak u_\ell=1\}.$$
\end{definition}

\begin{lemma}\label{hyptime}
For all sequences  $\mathfrak l^n$ and $\mathfrak u^n$, we have $$\Vol_{\sigma}\left( \mathcal H(\mathfrak l^n)\cap \mathcal B(\mathfrak u^n)\right)\leq 4D_{r,d}^n\left(M_g^{\frac{3k^2}{p(r-1)}}e^{-\chi}\right)^{p\# (H(\mathfrak l^n)\setminus \mathrm{BG}(\mathfrak u^n))}.$$

\end{lemma}

\begin{proof}

Fix  sequences $\mathfrak l^n$ and $\mathfrak u^n$.\\
 
 \noindent\textbf{Claim:} There are integers $0\leq \ell_0<\ell'_0<\ell_1<\ell'_1<\cdots <\ell'_q< n$ such that :
 \begin{itemize}
 \item $\forall i=0,\cdots, q, \ \ell_i\in  \mathrm{BG}(\mathfrak u^n)$, 
 \item $\forall i=0,\cdots, q, \ \ell'_i\in H(\mathfrak l^n)\setminus \mathrm{BG}(\mathfrak u^n)$,
 \item $H(\mathfrak l^n)\setminus \mathrm{BG}(\mathfrak u^n)\subset \bigcup_{0\leq i\leq q}(\ell_i,\ell'_i]\subset [0,n)\setminus \mathrm{BG}(\mathfrak u^n) $. 
 \end{itemize}

\begin{proof}[Proof of the Claim]
 Let $\ell'_{-1}=-1$. 
We set inductively for all $i$: $$\ell_i=\max \left\{\ell\in \mathrm{BG}(\mathfrak u^n): 
 \ell <\ell' \ \forall \ell'_{i-1}<\ell' \in H(\mathfrak l^n)\setminus \mathrm{BG}(\mathfrak u^n)\right\} $$ and $$\ell'_i=\max\left\{ \ell'\in  H(\mathfrak l^n)\setminus \mathrm{BG}(\mathfrak u^n): \ \ell' <\ell \ \forall \ell_i<\ell\in \mathrm{BG}(\mathfrak u^n)\right\} .$$
 We let $q$ be the last integer for which $\ell'_q$ is well defined.
\end{proof}

 By applying Lemma \ref{hnotbg} successively to the triplets $(\ell_0,\ell'_0, \ell_1)$, $(\ell_1,\ell'_1,\ell_2)$, $(\ell_2,\ell'_2,\ell_3)$, $ \cdots$, $(\ell_q,\ell'_q,n-1)$ it follows that 
\begin{align*}
\sup_{\mathbf i^{\ell_0}}\Vol_{\widetilde{\sigma}_{\mathbf i^{\ell_0}}}\left( \mathcal H(\mathfrak l_{\ell_0}^n)\cap \mathcal B(\mathfrak u_{\ell_0}^n)\right)&\leq \left(D_{r,d}M_g^{\frac{3k^2}{r-1}}e^{-p\chi-2}\right)^{\sum_{i=0}^q\ell'_i-\ell_i}\\
&\leq D_{r,d}^n\left(M_g^{\frac{3k^2}{p(r-1)}}e^{-\chi}\right)^{p\sum_{i=0}^q\ell'_i-\ell_i}\\
&\leq D_{r,d}^n\left(M_g^{\frac{3k^2}{p(r-1)}}e^{-\chi}\right)^{p\# (H(\mathfrak l^n)\setminus \mathrm{BG}(\mathfrak u^n))}.
\end{align*}
Then by using again the bounded distortion property  we get 
\begin{align*}\Vol_{\sigma}\left( \mathcal H(\mathfrak l^n)\cap \mathcal B(\mathfrak u^n)\right)&\leq 4 \sup_{\mathbf i^{\ell_0}}\Vol_{\widetilde{\sigma}_{\mathbf i^{\ell_0}}}\left( \mathcal H(\mathfrak l_{\ell_0}^n)\cap \mathcal B(\mathfrak u_{\ell_0}^n)\right),\\
& \leq 4D_{r,d}^n\left(M_g^{\frac{3k^2}{p(r-1)}}e^{-\chi}\right)^{p\# (H(\mathfrak l^n)\setminus \mathrm{BG}(\mathfrak u^n))}.
\end{align*}

\end{proof}

\begin{definition}
For all $x\in D$, all $\beta \in (0,1)$ and all $n\in \mathbb  N$ we  define: 
\begin{align*}
\underline{\mathfrak u}^\beta_n(x)&=1 \text{ if $x\in \sigma(\theta_{\mathbf i^n})$ with $\mathbf i^n\in \overline{\mathcal T}_n$ and $\Vol_k(\sigma_n\circ \theta_{\mathbf  i^n})\leq \frac{\beta C_d^{-1}\epsilon^k}{ 400A_{r,d}}$},\\
&=0  \text{ if not.}
\end{align*}

Given $n\in \mathbb N$, we set

$$\underline{\mathfrak u}^n_\beta(   
x)=\left(\underline{\mathfrak u}^\beta_i( x)\right)_{0\leq i<n}.$$

\end{definition}

\medskip
\noindent\textbf{Remark:} In particular 
$\underline{\mathfrak u}^\beta_n(x)=1$ implies $\mathfrak u_n(x)=1$.
\medskip

\begin{definition}
For an $n$-tuple    $\underline{\mathfrak  
u}^n=\left(\underline{\mathfrak u}_i\right)_{0\leq i<n}\in \{0,1\}^n$  and for $0\leq \ell<n$ 
we consider  $$\underline{\mathcal B}^\beta(\underline{\mathfrak u}_\ell^n):=g^\ell\left(\left\{x\in D \ :\underline{\mathfrak u}_\beta^n(  x)=\underline{\mathfrak  u}^n\right\}\right).$$

\end{definition}

\begin{lemma}\label{lastlem}Given $\underline{\mathfrak u}^n$, for any $0\leq \ell_0<\ell_1<n$, with $\underline{\mathfrak u}_{\ell_1}=1$, we have 
$$\sup_{\theta_{\mathbf i^{\ell_0}} } \Vol_{\sigma_{\ell_0}\circ \theta_{\mathbf i^{\ell_0}}}(\underline{\mathcal B}^\beta(\underline{\mathfrak u}_{\ell_0}^n))\leq \frac{\beta}{25}\sup_{\theta_{\mathbf i^{\ell_1}} : \ \theta_{\mathbf i^{\ell_1}}<\theta_{\mathbf i^{\ell_0}} }\Vol_{\sigma_{\ell_1}\circ \theta_{\mathbf i^{\ell_1}}}(\underline{\mathcal B}^\beta(\underline{\mathfrak u}_{\ell_1}^n)).$$
\end{lemma}

\begin{proof}
By  the bounded distortion property, recall that
$$\sup_{\theta_{\mathbf i^{\ell_0}} } \Vol_{\sigma_{\ell_0}\circ \theta_{\mathbf i^{\ell_0}}}(\underline{\mathcal B}^\beta(\underline{\mathfrak u}_{\ell_0}^n))\leq 4\sup_{\theta_{\mathbf i^{\ell_1-1}} : \ \theta_{\mathbf i^{\ell_1-1}}<\theta_{\mathbf i^{\ell_0}} }\Vol_{\sigma_{\ell_1-1}\circ \theta_{\mathbf i^{\ell_1-1}}}(\underline{\mathcal B}^\beta(\underline{\mathfrak u}_{\ell_1-1}^n)).$$

Therefore it is enough to show that
$$\sup_{\theta_{\mathbf i^{\ell_1-1}}} \Vol_{\sigma_{\ell_1-1}\circ \theta_{\mathbf i^{\ell_1-1}}}(\underline{\mathcal B}^\beta(\underline{\mathfrak u}_{\ell_1-1}^n))\leq \frac{\beta}{100}\sup_{\theta_{\mathbf i^{\ell_1}} : \ \theta_{\mathbf i^{\ell_1}}<\theta_{\mathbf i^{\ell_1-1}} }\Vol_{\sigma_{\ell_1}\circ \theta_{\mathbf i^{\ell_1}}}(\underline{\mathcal B}^\beta(\underline{\mathfrak u}_{\ell_1}^n)).$$

We have for all $\theta_{\mathbf i^{\ell_1-1}}$: 
\begin{align*}
&\Vol_k\left(\mathrm{Im}\left(\sigma_{\ell_1-1}\circ \theta_{\mathbf i^{\ell_1-1}}\right) \cap \underline{\mathcal B}^\beta(\underline{\mathfrak u}_{\ell_1-1}^n) \right) \\
\leq &\sum_{\overset{\theta_{\mathbf i^{\ell_1}}: \ \mathbf i^{\ell_1}\in \overline{\mathcal T}_{\ell_1}, \  \theta_{\mathbf i^{\ell_1}}<\theta_{\mathbf i^{\ell_1-1}}}{\Vol_k(\sigma_{\ell_1}\circ \theta_{\mathbf i^{\ell_1}} )\leq \frac{\beta C_d^{-1}}{400A_{r,d}}\epsilon^{k} } } \Vol_k\left(\mathrm{Im}\left(\sigma_{\ell_1-1}\circ \theta_{\mathbf i^{\ell_1}}\right) \cap \underline{\mathcal B}^\beta(\underline{\mathfrak u}_{\ell_1}^n) \right)&\\
\leq& \sum_{\overset{\theta'_{\mathbf i^{\ell_1}}: \ \mathbf i^{\ell_1}\in \overline{\mathcal T}_{\ell_1},}{\theta'_{\mathbf i^{\ell_1}}<\theta_{\mathbf i^{\ell_1-1}}}}\left(  \sum_{\overset{\theta_{\mathbf i^{\ell_1}}: \ \theta_{\mathbf i^{\ell_1}}<\theta'_{\mathbf i^{\ell_1}} }{{\Vol_k(\sigma_{\ell_1}\circ \theta_{\mathbf i^{\ell_1}} )\leq \frac{\beta C_d^{-1}}{400A_{r,d}}\epsilon^{k} } } } \Vol_k\left(\mathrm{Im}\left(\sigma_{\ell_1-1}\circ \theta_{\mathbf i^{\ell_1}}\right) \cap \underline{\mathcal B}^\beta(\underline{\mathfrak u}_{\ell_1}^n) \right)\right).
\end{align*}

 Let $x'_{\ell_1}\in \mathrm{Im}(\sigma\circ \theta'_{\mathbf i^{\ell_1}})$.  By Equation \eqref{fran}  we have 
 
\begin{align*}&\Vol_k\left(\mathrm{Im}\left(\sigma_{\ell_1-1}\circ \theta_{\mathbf i^{\ell_1-1}}\right) \cap \underline{\mathcal B}^\beta(\underline{\mathfrak u}_{\ell_1-1}^n) \right) \\
\leq& \sum_{\overset{\theta'_{\mathbf i^{\ell_1}}: \ \mathbf i^{\ell_1}\in \overline{\mathcal T}_{\ell_1},}{\theta'_{\mathbf i^{\ell_1}}<\theta_{\mathbf i^{\ell_1-1}}}}2|\wedge^kdg (G^{\ell_1-1}(\widetilde{x}'_{\ell_1}))|^{-1} \\
\cdot&\left(  \sum_{\overset{\theta_{\mathbf i^{\ell_1}}: \ \theta_{\mathbf i^{\ell_1}}<\theta'_{\mathbf i^{\ell_1}} }{{\Vol_k(\sigma_{\ell_1}\circ \theta_{\mathbf i^{\ell_1}} )\leq \frac{\beta C_d^{-1}}{400A_{r,d}}\epsilon^{k} } } } \Vol_k\left(\mathrm{Im}\left(\sigma_{\ell_1}\circ \theta_{\mathbf i^{\ell_1}}\right) \cap \underline{\mathcal B}^\beta(\underline{\mathfrak u}_{\ell_1}^n) \right)\right)
\end{align*}
\begin{align*}
\leq &\sum_{\overset{\theta'_{\mathbf i^{\ell_1}}:\  \mathbf i^{\ell_1}\in \overline{\mathcal T}_{\ell_1},}{\theta'_{\mathbf i^{\ell_1}}<\theta_{\mathbf i^{\ell_1-1}}}}2|\wedge^kdg (G^{\ell_1-1}(\widetilde{x}'_{\ell_1}))|^{-1}\frac{\beta}{400}\Vol_k(\mathrm{Im}(\sigma_{\ell_1}\circ \theta'_{\mathbf i^{\ell_1}} ))\\
\cdot& \sup_{\theta_{\mathbf i^{\ell_1}} : \ \theta_{\mathbf i^{\ell_1}}<\theta_{\mathbf i^{\ell_1-1}} }\Vol_{\sigma_{\ell_1}\circ \theta_{\mathbf i^{\ell_1}}}(\underline{\mathcal B}^\beta(\underline{\mathfrak u}_{\ell_1}^n))\\
\leq &\sum_{\theta'_{\mathbf i^{\ell_1}}: \ \theta'_{\mathbf i^{\ell_1}}<\theta_{\mathbf i^{\ell_1-1}}}  \frac{\beta}{100}\Vol_k(\mathrm{Im}(\sigma_{\ell_1-1}\circ \theta'_{\mathbf i^{\ell_1}} ))\sup_{\theta_{\mathbf i^{\ell_1}} : \ \theta_{\mathbf i^{\ell_1}}<\theta_{\mathbf i^{\ell_1-1}} }\Vol_{\sigma_{\ell_1}\circ \theta_{\mathbf i^{\ell_1}}}(\underline{\mathcal B}^\beta(\underline{\mathfrak u}_{\ell_1}^n))\\
\leq & \frac{\beta}{100}\Vol_k(\sigma_{\ell_1-1}\circ \theta_{\mathbf i^{\ell_1-1}} )\sup_{\theta_{\mathbf i^{\ell_1}} : \ \theta_{\mathbf i^{\ell_1}}<\theta_{\mathbf i^{\ell_1-1}} }\Vol_{\sigma_{\ell_1}\circ \theta_{\mathbf i^{\ell_1}}}(\underline{\mathcal B}^\beta(\underline{\mathfrak u}_{\ell_1}^n)).
\end{align*}

\end{proof}

Applying Lemma \ref{lastlem} inductively yields the following lemma.
\begin{lemma} \label{ouf}For any $\underline{\mathfrak u}^n$
$$\Vol_{\sigma}(\underline{\mathcal B}^\beta(\underline{\mathfrak u}^n))\leq \left(\frac{\beta}{25}\right)^{\# \mathrm{BG}(\underline{\mathfrak u}^n)}.$$
\end{lemma}

\begin{definition}
For $x\in D$ we define  
\begin{align*}\mathfrak E_{\widehat{\chi},p,\beta}(x):= \{m:\ & \text{$m$ is $(\widehat{\chi}, \mathfrak l^\infty(x))$-hyperbolic,} \\  & \text{$\mathfrak{u}_m(x)=1$ and $\underline{\mathfrak{u}}^\beta_m(x)=0$} \}.
\end{align*}
\end{definition}

\begin{definition}
    Given a set $B\subseteq \mathbb{N}$ and $n\in\mathbb N^*$, we set
$$d_n(B):=\frac{1}{n}\# \left(B\cap [0,n-1]\right).$$
\end{definition}

\medskip
Recall Lemma \ref{preBC} and the definition of $\chi'>\chi>\frac{3k^2}{r-1}R(f)$. 

 \begin{prop}\label{lebgeo} For all $\beta\in (0,1)$, for all $p\geq \frac{8}{\chi'-\chi}$ and for all $n$, we have 
  \begin{align*}&\Vol_\sigma\left(\left\{x: \ d_{n}(\mathfrak E_{p\chi+2,p,\beta}(x))< \beta \textrm{ and } |A_{k,p}^{np}(\iota (T_xD))|\geq e^{np \chi'}\right\}\right) \\
  &\leq C_{r,d,f}(p)^{n}\left(e^{\frac{3k^2}{r-1}R(f)}e^{-\chi}\right)^{np\cdot\left(\frac{\chi'-\chi}{4k\log M_f}-2\beta\right)}+e^{-\beta n}, \end{align*}
  where $C_{r,d,f}(p)$ satisfies $\lim_{p\to\infty}\frac{\log C_{r,d,f}(p)}{p}=0$.
  \end{prop}
 \begin{proof}
First, note that if  $|A_{k,p}^{np}(\iota (T_xD))|\geq e^{np\chi'}$, then there is $0\leq q<p$ such that 
$$\prod_{i=0}^{n-1}\frac{|\wedge^kd_{f^{ip+q}x}f^p(F^{ip+q}\widetilde x)|}{\max\{1,\max_{j\leq k-1}\|\wedge^jd_{f^{ip+q}x}f^p\|\}}\geq e^{np\chi'}. $$
Without loss of generality we can assume that $q=0$. Then we get that,
$$\sum_{i=0}^{n-1} \mathfrak l_i(x)-\mathfrak l'_i(x)\geq n(p\chi'-2).$$ 

Set
$$\mathcal E_n^\beta:=\left\{x\in D: \ d_{n}(\mathfrak E_{p\chi+2,p,\beta}(x))< \beta \textrm{ and }  \sum_{i=0}^{n-1} \mathfrak l_i(x)-\mathfrak l'_i(x)\geq n(p\chi'-2)\right\}.$$ 

\medskip 
Since the entries of $\mathfrak l$ and $\mathfrak l'$ belong to $[-k\log M_g-1, k\log M_g+1]$, the number of sequences $(\mathfrak l^n, \mathfrak u^n, \underline{\mathfrak u}^n)$ with $\mathcal H(\mathfrak l^n)\neq \varnothing$, $\mathcal B(\mathfrak u^n)\neq \varnothing$, $\underline{\mathcal B}(\underline{\mathfrak u}^n)\neq \varnothing$ is bounded from above by $(2k\log M_g+3)^{2n} 4^n$. 

\medskip
Fix such sequences $\mathfrak l^n$, $\mathfrak u^n$  and  $\underline{\mathfrak u}^n$  with $\mathcal H(\mathfrak l^n)\cap \mathcal B(\mathfrak u^n)\cap \underline{\mathcal B}^\beta(\underline{\mathfrak u}^n)\cap \mathcal E_{n}^\beta\neq \varnothing$ 
and $\#\mathrm{BG}(\underline{\mathfrak u}^n)\leq \beta n$. Let $q_*\in \mathbb N$ with $\beta\in [\frac{1}{q_*+1}, \frac{1}{q_*})$. By Lemma \ref{ouf} the volume of the union of $\underline{\mathcal B}(\underline{\mathfrak u}^n)$ over the others $\underline{\mathfrak u}^n$ is less than  (up to a multiplicative constant independent of $n$)
\begin{align*}\sum_{q=q_*}^n {n \choose [n/q]}\cdot \left(\frac{1}{25q}\right)^{ \lceil\frac{n}{q}\rceil}& \lesssim \sum_{q=q_*}^n e^{n \left(\frac{1}{q}\log q-(1-\frac{1}{q})\log (1-\frac{1}{q})\right) }\cdot \left(\frac{1}{25q}\right)^{ \frac{n}{q}}\\
&\lesssim \sum_{q=q_*}^n e^{\frac{n}{q}(2-\log 25)} \\
&\lesssim e^{\beta(2-\log 25)n}\leq e^{-\beta n}.
\end{align*}

 By Pliss's lemma (Lemma \ref{Pliss}) applied to the sequence $\left(\mathfrak l_i-\mathfrak l'_i\right)_{i<n}$ with $\widehat{\chi}=p\chi+2$ and $\widehat{\chi}'=p\chi'-2$, we have $\# H(\mathfrak l^n)\geq n\frac{p(\chi'-\chi)-4}{2k\log M_g}$. Therefore, for  $p\geq \frac{8}{\chi-\chi'}$ and for all $x\in \mathcal H(\mathfrak l^n)\cap \mathcal B(\mathfrak u^n)\cap \mathcal E_{n}^\beta$ we get 
 \begin{align*}
     \# \left(H(\mathfrak l^n)\setminus \left(\mathrm{BG}(\mathfrak u^n)\setminus \mathrm{BG}(\underline{\mathfrak{u}}^n\right)\right)\geq &\# H(\mathfrak l^n) -nd_n(\mathfrak E_{p\chi+2,p,\beta}(x))\\
     \geq &n\cdot \left(\frac{\chi'-\chi}{4k\log M_f}-\beta\right).
 \end{align*}

Then,
\begin{align*} \# \left(H(\mathfrak l^n)\setminus \mathrm{BG}(\mathfrak u^n)\right)&\geq \# \left(H(\mathfrak l^n)\setminus \left(\mathrm{BG}(\mathfrak u^n)\setminus \mathrm{BG}(\underline{\mathfrak{u}}^n\right)\right)-\#\mathrm{BG}(\underline{\mathfrak{u}}^n)\\
&\geq n\cdot \left(\frac{\chi'-\chi}{4k\log M_f}-2\beta\right).
\end{align*}

By Lemma \ref{hyptime} we get 
 $$\Vol_{\sigma}\left( \mathcal H(\mathfrak l^n)\cap \mathcal B(\mathfrak u^n)\right)\leq 4D_{r,d}^n\left(M_g^{\frac{3k^2}{p(r-1)}}e^{-\chi}\right)^{-p\chi'n\cdot \left(\frac{\chi'-\chi}{4k\log M_f}-2\beta\right)}.$$

Therefore, for all $n$, \begin{align*}&\Vol_\sigma\left(\left\{x: \ d_{n}(E_{\widehat{\chi},p}(x))< \beta \textrm{ and } \|A_{k,p}^{np}(\iota (T_xD))\|\geq e^{np \chi}\right\}\right) \\
&\leq  (2pk\log M_f+3)^{2n} (4D_{r,d})^n \left(M_g^{\frac{3k^2}{p(r-1)}}e^{-\chi}\right)^{pn\cdot\left(\frac{\chi'-\chi}{4k\log M_f}-2\beta\right)}.
\end{align*}
This concludes the proof with $$C_{r,d,f}(p):=4D_{r,d}(2pk\log M_f+3)^{2}\left(\frac{M_{f^p}}{e^{pR(f)}}\right)^{\frac{3k^2(\chi'-\chi)}{4k(r-1)\log M_f}-2\beta} .$$
\end{proof}

\medskip

\begin{definition}
For  $x\in D$, we set \begin{align*}\mathrm{BG}^\beta_n(x):=&\Big\{ 0\leq \ell<n: \ 
 \mathfrak u_\ell(x)=1, \underline{\mathfrak u}_\ell^\beta(x)=0 \\ 
 &\text{ and $\ell$ is $(p\chi, \mathfrak l^{\infty}(y))$ for any $y\in \varpi_\ell(x)$} \Big\}.
\end{align*}

\end{definition}

\begin{definition}\label{En} For any $\beta\in (0,1)$ and any $p\in \mathbb N^*$ we let 
    $$E_n^{p,\beta}:=\left\{x\in B_n^p : \   
   d_n(\mathrm{BG}^\beta_n(x)) \geq\beta\right\}.$$
\end{definition}

\begin{cor}\label{seten}For any  $\beta \in (0, \frac{\chi'-\chi}{8k\log M_f})$ there is $p$, such that  for all $n\in \mathcal N$ large enough, 
$$\Vol_D(E_n^{p,\beta})\geq \frac{1}{n^2}.$$
\end{cor}

\begin{proof}
Recall the set $B_{n}^p$ given  in Lemma \ref{preBC} satisfies $\Vol_D(B_n^p)\geq \frac{2}{n^2}$ for $n\in \mathcal N$. As observed after Lemma \ref{BGtimesAreHyp}, we have $\mathfrak E_{p\chi+2,p,\beta}(x)\subset \mathrm{BG}^\beta_n(x)$ for $x\in D$. 
By  Proposition \ref{lebgeo}, for $\beta<\frac{\chi'-\chi}{8k \log M_f}$ we may choose $p$ so large that the set of points $x$ in $B_{n}^p$ with $d_{n}(\mathfrak E_{p\chi+2,p,\beta}(x))< \beta \textrm{ and } |A_{k,p}^{np}(\iota (T_xD)) | \geq e^{np \chi'}$ has exponentially small measure in $n$, in particular measure less than $\frac{1}{n^2}$ for  $n$ large enough.
\end{proof}

\medskip
\noindent\textbf{Remark:} From now on, we fix some $\beta \in (0, \frac{\chi'-\chi}{8k\log M_f})$ and we choose $p$ as in Corollary \ref{seten}. The set $E_{n}^{\beta, p}$ and $\mathrm{BG}^\beta_n(x)$ will be then simply denoted by $E_n$ and $\mathrm{BG}_n(x)$ respectively. The integers in $\mathrm{BG}_n(x)$ are referred to as $\mathrm{BG}$ times.

\section{Invariant measure given by disintegration on disks}\label{constMu}

In this section we construct an invariant measure which is given by a disintegration on $C^{r-1,1}$ embedded $k$-disks. To achieve that, we wish to first construct a probability measure on the space of {\em Measured Disks}- which is a space of disks endowed with a probability measure. We restrict to BG times where the elements of the Yomdin partition belong to a pre-compact subset of the space of Measured Disks, where we can take limits.

\subsection{Space of Measured Disks}

\begin{definition}[Space of Measured Disks]
    Let $$\mathcal{S}:=\{(\varpi,\varrho):\varpi:[0,1]^k\to \mathbf M\text{ is }C^{r-1,1}\text{ and }\varrho\in\mathbb{P}([0,1]^k)\},$$ endowed with the product topology of the $C^{r-1,1}$ topology and the weak-$*$ topology, be the {\em space of Measured Disks}.
\end{definition}

\begin{definition}\text{ }
\begin{enumerate}
    \item For $x\in D$ and $\ell\in \mathbb N$,  we say that the atom  $\varpi_\ell(x)$ of the Yomdin partition  containing $g^\ell(x)$ is $\mathrm{BG}$, when  $\ell $ belongs to $\mathrm{BG}_n(x)$ for some (any) $n> \ell $.
    \item Let $P^{\ell}$ be the Yomdin partition of $D$ into elements of the form \\$g^{-\ell}[\mathrm{Im}(\varpi_\ell(\cdot))]$.
\end{enumerate}

\end{definition}

\subsection{Sequence of measures on the space of Measured Disks}\text{ }

\medskip
Recall Lemma \ref{preBC} and the definition of the set $\mathcal{N}$.

\begin{definition}
For $n\in \mathcal{N}$, and for $0\leq\ell\leq n$, 
\begin{enumerate}
    \item 
    $$E_n^{\ell}:=\bigcup\{g^{-\ell}[\mathrm{Im}(\varpi_\ell(x))]: x\in E_n \}\subset D,$$ 
    \item Let $P_n^\ell$ be the restriction of $P^\ell$ to elements centered at $E_n$ points,
  \item Given $x\in E_n^n$, write
$$\mathrm{BG}_n^M(x):=\{\ell\leq n: \exists \ell'\in[\ell,\ell+M],\ y\in \varpi_\ell(x)\cap E_n\text{ s.t.  }\ \varpi_{\ell'}(y)\text{ is $\mathrm{BG}$}\},$$
\item $$P_n^{\ell,M}:=\{P\in P_n^\ell: \ell\in \mathrm{BG}_n^M(x)\text{ for some (any) }x\in P\cap E_n\},$$
 \item We denote $\Vol_D$ by $\lambda$, $\lambda(F\cap \cdot)$ by $\lambda_F$ for any $F\subset D$   and $\frac{1}{\lambda(E_n^n)}\lambda_{E_n^n}$ by $\lambda_n$.
 \end{enumerate}
\end{definition}

\medskip
\noindent\textbf{Remark:}
 In the above definition,  $\varpi_{\ell}(x)$ denotes the map $\sigma_\ell\circ \theta_{\mathbf i^\ell}$ with $x\in \mathrm{Im}(\sigma\circ \theta_{\mathbf i^\ell})$ (see Definition \ref{yomd}). According to the Remark following the Algebraic Lemma (Lemma \ref{alg}), $\theta_{\mathbf i^\ell}$, thus $\sigma_\ell\circ \theta_{\mathbf i^\ell}$ is a (topological) embedding, therefore we may push the measure $\frac{1}{\lambda(P_n^\ell(x)\cap E_n^n)}\lambda_{P_n^\ell(x)\cap E_n^n}\circ g^{-\ell}$ on $\mathbf M$ by the inverse of $\varpi_{\ell}(x)$ to get $\varrho_{\ell, n,x}$.

\medskip 

 \begin{definition}[The subspace $\mathcal S_M$]
 
 Given $M\in \mathbb{N}$, we write 
 $$\mathcal{S}_M:=\overline{\{(\varpi_\ell(x),\varrho)\in \mathcal{S}:  \ell \in \mathrm{BG}_n^M(x), \ x\in E_n^n \}}^\mathcal{S}\Subset \mathcal{S},$$
 where $\Subset$ denotes ``compactly contained".
\end{definition}

\medskip

 \begin{definition}[The measures $p_n^M$ on $\mathcal S_M$]\label{pN}
    \begin{align*}
        p_n^{M}:=&\int \frac{1}{n}\sum_{\ell\in \mathrm{BG}_n^M(x)}\delta_{(\varpi_\ell(x), \varrho_{\ell,n,x})}d\lambda_n,
    \end{align*}
 where $\varrho_{\ell,n,x}:=\Big(\frac{1}{\lambda(P_n^\ell(x)\cap E_n^n)}\lambda_{P_n^\ell(x)\cap E_n^n}\circ g^{-\ell}\Big)\circ \varpi_\ell(x)$.

\end{definition}

\medskip
\begin{lemma}\label{sameEmpLims}
Given $x\in E_n$ and $y\in P_n^n(x)$,
\begin{enumerate}
    \item For all $M$, $\mathrm{BG}_n^M(x)=\mathrm{BG}_n^M(y)$, 
    \item  For all $n$ large enough and for all $h\in \mathrm{Lip}(\mathbf M)$, $$\Big|\frac{1}{n}\sum_{\ell\leq n}\delta_{g^{\ell }(x)}(h)-\frac{1}{n}\sum_{\ell\leq n}\delta_{g^{\ell }(y)}(h)\Big|\leq  \frac{\|h\|_\mathrm{Lip}}{\sqrt n}.$$
\end{enumerate}
\end{lemma}
\begin{proof}
   The first item is clear by the definition of $\mathrm{BG}_n^M(\cdot)$, which is saturated by partition elements. We continue to show the second item. Let $\eta>0$. By Lemma \ref{YomdinPtnShrinksUni},
    \begin{align*}
        &\Big|\frac{1}{n}\sum_{\ell\leq n}\delta_{g^{\ell }(x)}(h)-\frac{1}{n}\sum_{\ell\leq n}\delta_{g^{\ell }(y)}(h)\Big|\\
        \leq&\Big|\frac{1}{n}\sum_{\ell\leq n(1-\eta)}\delta_{g^{\ell }(x)}(h)-\frac{1}{n}\sum_{\ell\leq n(1-\eta)}\delta_{g^{\ell }(y)}(h)\Big|+2\eta\|h\|_\infty\\
        \leq &\mathrm{Lip}(h)\cdot \frac{1}{2^{n\eta}}+2\eta\cdot \|h\|_\infty.
    \end{align*}
Then if we choose $\eta=\frac{1}{3\sqrt n}$, the desired result is following, for all $n$ s.t.  $\frac{1}{2^{\frac{\sqrt{n}}{3}}}\leq \frac{1}{3\sqrt n}$.
\end{proof}

\medskip
\noindent\textbf{Remark:}
 From here onwards, we saturate the set $E_n$ into $E_n^n$, and associate the two sets (recall Definition \ref{En}). Lemma \ref{sameEmpLims} above justifies it.

\medskip
\begin{prop}\label{sizeOfPnM} For all $n\in \mathcal{N}$ and $M\geq0$, 
$$p_n^{M}(1)\in [\beta,1].$$
\end{prop}
\begin{proof} Recall Definition \ref{En} and its  following remark, then
\begin{align*}
   1\geq  p_n^M(1)&=\int\frac{1}{n}\#\mathrm{BG}_n^M(x)\,   d\lambda_n(x) \\ 
   &\geq \int d_n\left(\mathrm{BG}_n(x)\right)\, d\lambda_n(x) \geq \beta.
\end{align*}
\end{proof}

\begin{definition}\label{muN}\text{ }

 $$\mu^M_n:=\int_{\mathcal S} \varrho\circ \varpi^{-1}dp_n^M((\varpi,\varrho)).$$
\end{definition}

\begin{claim}\label{muLikeP} For every $n,M$, $\mu_n^M(1)=p_n^M(1)\in [\beta,1]$.
\end{claim}
\begin{proof}
This is a direct corollary of Definition \ref{muN}, as $\varrho\circ\varpi^{-1}(1)=1$ for all $(\varpi,\varrho)\in \mathcal{S}$.
\end{proof}

\begin{lemma}
    $$\mu_n^M=\int \frac{1}{n}\sum_{\ell \in \mathrm{BG}_n^M(x)}\delta_{g^\ell(x)}d\lambda_n.$$
\end{lemma}
\begin{proof}
       \begin{align*}
    \int\frac{1}{n}\sum_{\ell\in \mathrm{BG}_n^M(x)}\delta_{g^\ell(x)}d\lambda_n=&\frac{1}{\lambda(E_n)}\frac{1}{n}\sum_{\ell=0}^{n-1}\int_{[\ell\in \mathrm{BG}_n^M(\cdot)]}\delta_{g^\ell(x)}d\lambda \\
    =&\frac{1}{\lambda(E_n)}\frac{1}{n}\sum_{\ell=0}^{n-1}\lambda_{[\ell\in \mathrm{BG}_n^M(\cdot)]}\circ g^{-\ell}\\
    =&\frac{1}{\lambda(E_n)}\frac{1}{n}\sum_{\ell=0}^{n-1}\sum_{P\in P_n^{\ell,M}}\lambda_{E_n\cap P}\circ g^{-\ell}\\
    =&\frac{1}{\lambda(E_n)}\frac{1}{n}\sum_{\ell=0}^{n-1}\sum_{P\in P_n^{\ell,M}}\lambda(E_n\cap P)\frac{\lambda_{E_n\cap P}\circ g^{-\ell}}{\lambda(E_n\cap P)}
    \end{align*}
    \begin{align*}
    =&\frac{1}{\lambda(E_n)}\frac{1}{n}\sum_{\ell=0}^{n-1}
    \sum_{P\in P_n^{\ell,M}}\lambda(E_n\cap P)\int_{\mathcal S} \varrho\circ \varpi^{-1}d\delta_{(\varpi_\ell(x_P),\varrho_{\ell,n,x_P})}\\
=&\int_{\mathcal S} \varrho\circ \varpi^{-1}  d\Big(\int \frac{1}{n}\sum_{\ell\in \mathrm{BG}_n^M(x)}\delta_{(\varpi_\ell(x),\varrho_{\ell,n,x})}d\lambda_n\Big)\\
=&\int_{\mathcal S} \varrho\circ \varpi^{-1}  dp_n^M=\mu_n^M,
\end{align*}
where $x_P$ is any point in $P\cap E_n$.
\end{proof} 

\medskip
Recall Lemma \ref{preBC}.
\begin{prop}\label{subSeq}
    There exists a subsequence $\mathcal{N}\ni n_j\uparrow\infty$ s.t.  $\forall M\geq0$, $$p_{n_j}^M\to p^{M}\in[\beta,1]\cdot\mathbb{P}(\mathcal{S}_M)\subseteq [\beta,1]\cdot\mathbb{P}(\mathcal{S})\text{ and }\mu_{n_j}^M\to \mu^{M}\in [\beta,1]\cdot\mathbb{P}(\mathbf M).$$
    Moreover, $$\mu^M=\int \varrho\circ \varpi^{-1}dp^M.$$
\end{prop}
\begin{proof}

  By Proposition \ref{sizeOfPnM}, $p_n^M$ is a sequence of bounded measures carried by $\mathcal{S}_M$, which is a pre-compact subset of $\mathcal{S}$. 
 
 Therefore, using the diagonal argument, there exists a subsequence $\mathcal{N}\ni n_j\uparrow\infty$ s.t.  $\forall M\geq0$, $p_{n_j}^M\to p^{M}\in[\beta,1]\cdot\mathbb{P}(\mathcal{S}_M)\subseteq [\beta,1]\cdot\mathbb{P}(\mathcal{S})$.

Then, given any continuous function $h\in C(\mathbf{M})$, we define $\widehat{h}(\varpi,\varrho):=\varrho\circ \varpi^{-1}(h)=\int_{[0,1]^k} h\circ \varpi\, d\varrho$, which lies in $C(\mathcal{S})$. 
 Then,
\begin{align}\label{formonomono}
    \int h\, d\mu_{n_j}^M&=\int \varrho\circ \varpi^{-1}(h)\, dp_{n_j}^M \\
    &=\int \widehat{h}\, dp_{n_j}^M\xrightarrow[j\to\infty ]{}\int \widehat{h}\, dp^M\nonumber
\end{align}
Therefore 
\begin{align*}
    \mu^M(h)&=\int \varrho\circ \varpi^{-1}(h)\, dp^M.
\end{align*}
\end{proof}

\subsection{Invariant measure with disintegration by a family of disks}

\begin{lemma}\label{monomonomono}
    For all $n$ and $M$, $p_n^{M+1}\geq p_n^M$, $p^{M+1}\geq p^M$, $\mu_n^{M+1}\geq \mu_n^M$, and $\mu^{M+1}\geq \mu^M$.
\end{lemma}
\begin{proof}
    It is enough to check that for all $n$ and $M$, $p_n^{M+1}\geq p_n^M$, and the rest follows by Proposition \ref{subSeq}.
    Indeed, by Definition \ref{pN}, for any $n$, when increasing $M$, we increase the sum in the definition of $p_n^M$.
\end{proof}

\begin{cor}[The measures $\mu$ and $\mathbf p$]\label{muAndP}
   $p^{M}\uparrow \mathbf{p}$ and $\mu^{M}\uparrow \mu$, where $p(\mathcal{S})\in [\beta,1]$ and $\mu=\int \varrho\circ \varpi^{-1}d\mathbf p$.
\end{cor}
\begin{proof}
By Lemma \ref{monomonomono}, both limits must exist as bounded increasing sequences, and  $\mu^M=\int \varrho\circ \varpi^{-1}dp^M$ for every $M$. Then as in \eqref{formonomono}, we conclude that $\mu=\int \varrho\circ \varpi^{-1}d\mathbf{p}$.
\end{proof}

\begin{lemma}\label{muLikeMuNull}
$\mu\ll \sum_{i\geq0 }\mu^0\circ f^{-i}$.
\end{lemma}
\begin{proof}
   Recall $g=f^p$. For every $M\in \mathbb{N}$, 
    $$\mu^M\leq \sum_{0\leq i\leq M}\mu^0\circ f^{-i p}\leq 2^{Mp}\sum_{0\leq i\leq Mp}\frac{1}{2^i}\mu^0\circ f^{-i}< 2^{Mp}\sum_{ i\geq 0}\frac{1}{2^i}\mu^0\circ f^{-i} .$$
Hence $\mu^M\ll \sum_{ i\geq 0}\frac{1}{2^i}\mu^0\circ f^{-i}\equiv \eta$. Since $\mu^{M+1}\geq \mu^M$, we get that $\mu=\lim_M\uparrow \mu^M\ll \eta$. 
\end{proof}

\medskip
\noindent\textbf{Remark:} Lemma \ref{muLikeMuNull} implies that $\mu^0$ sees all ergodic components of the measure $\mu$.

\begin{theorem}\label{dom}
$\mu\circ g^{-1}=\mu$.
\end{theorem}
\begin{proof}
For any $n\in \mathcal{N}$, for any $M\geq 0$,
\begin{align*}
    \mu_n^{M+1}\circ g^{-1}=&\int \frac{1}{n}\sum_{\ell\in \mathrm{BG}^{M+1}_n(x)}\delta_{g^\ell(x)}\circ g^{-1}d\lambda_n\\
    =&\int \frac{1}{n}\sum_{\ell\in \mathrm{BG}^{M+1}_n(x)}\delta_{g^{\ell+1}(x)}d\lambda_n\\
    \geq &\int \frac{1}{n}\sum_{\ell+1\in \mathrm{BG}^{M}_n(x)}\delta_{g^{\ell+1}(x)}d\lambda_n\\
    \geq&\int \frac{1}{n}\sum_{\ell\in \mathrm{BG}^{M}_n(x)}\delta_{g^{\ell}(x)}d\lambda_n-\frac{1}{n}=\mu_n^M- \frac{1}{n}.
\end{align*}
By sending $n_j\to\infty$ and then $M\to\infty$, we conclude $\mu\circ g^{-1}\geq \mu$. Since $\mu(1)$ and $\mu\circ g^{-1}(1)$ have the same value, $\mu=\mu\circ g^{-1}$.
\end{proof}

\begin{lemma}\label{fromGtoF}
   There exists a probability measure $\widetilde{\mathbf{p}}\in \mathbb{P}(\mathcal{S})$ s.t.  $$\frac{1}{p}\sum_{j=0}^{p-1}\left(\frac{1}{\mu(1)}\mu\right)\circ f^{-j}=\int_S \varrho\circ \varpi^{-1}d\widetilde{\mathbf{p}}((\varpi,\varrho)).$$
\end{lemma}
\begin{proof}
For $0\leq j\leq p-1$, let $F_j((\varpi,\varrho)):=(f^{j}\circ \varpi,\varrho)$ which is an invertible and continuous map from $\mathcal{S}$ to $\mathcal{S}$. Then set $\widetilde{\mathbf{p}}:=\frac{1}{p}\sum_{j=0}^{p-1}(\frac{1}{\mathbf{p}(1)}\mathbf{p})\circ F_j^{-1}$ (recall Claim \ref{muLikeP}).
\end{proof}

\begin{definition}\label{muHat}
    $$\widehat{\mu}:=\int_S \varrho\circ \varpi^{-1}d\widetilde{\mathbf{p}}((\varpi,\varrho)).$$
\end{definition}

\begin{cor} $\widehat{\mu}$ is an $f$-invariant probability measure on $\mathbf M$.
\end{cor}
\begin{proof}
    It follows directly from Lemma \ref{fromGtoF} (recall that $g=f^p$).
\end{proof}

\begin{definition}
    Let $\mathcal{W}:=\{
    \varpi:[0,1]^k\to \mathbf M\text{ where }\varpi\text{ is }C^{r-1,1}\}$ be the {\em space of disks}.
\end{definition}

\begin{theorem}[Invariant measure with disintegration by an invariant family of disks]\label{theCondsDef}
    There exists a probability measure $\widehat{\mathbf{p}}$ on $\mathcal{W}$ s.t.  $$\widehat{\mu}=\int_{\mathcal{W}} \varrho_\varpi\circ \varpi^{-1} d\widehat{\mathbf{p}}(\varpi),$$ where $\varrho_\varpi$ is a probability measure on $[0,1]^k$.
\end{theorem}
\begin{proof}
    Let $\alpha$ be the measurable partition of $\mathcal{S}$ given by the atoms $\alpha((\varpi,\varrho)):=\{\varpi\}\times \mathbb{P}([0,1]^k)$, and disintegrate $\widetilde{\mathbf{p}}$:
    $$\widetilde{\mathbf{p}}=\int \widetilde{\mathbf{p}}_{\alpha((\varpi,\varrho))}d\widetilde{\mathbf{p}}\equiv \int \widetilde{\mathbf{p}}_{\varpi}d\widehat{\mathbf{p}},$$
    where $\pi:\mathcal{S}\to \mathcal{W}$, $\pi((\varpi,\varrho)):=\varpi$ and $\widehat{\mathbf{p}}:=\widetilde{\mathbf{p}}\circ\pi^{-1}$. Then set $\varrho_\varpi:= \int \varrho d\widetilde{\mathbf{p}}_{\varpi}$ and so,
    $$\widehat{\mu}=\int \varrho\circ \varpi^{-1}d\widetilde{\mathbf{p}}=\int \left(\int \varrho\circ \varpi^{-1} d\widetilde{\mathbf{p}}_\varpi \right)d\widehat{\mathbf{p}}=\int \varrho_\varpi\circ\varpi^{-1} d\widehat{\mathbf{p}}.$$
\end{proof}

\section{Smooth Conditionals}\label{smoothCondSect}

In this section we study the invariant measure which we constructed in \textsection \ref{constMu}, and prove that its disintegration by Measured Disks in fact is given by smooth disk-measures (the ``conditionals"). The proof is decomposed into two sub-sections, where we first show that the conditionals are absolutely continuous, and then prove also domination in the following sub-section. Proving that the conditionals are absolutely continuous relies on finding a natural family of smooth disk-measures to which we compare the conditionals. The comparison is done by the atoms of finer-and-finer Yomdin partitions, where we are able to get tight bounds on the number of atoms which ``compare well", and the total measure they cover. The key idea which allows us to get these estimates is the observation that if a certain point visits with some fixed positive density times where a fixed portion of the atoms do not ``compare well", then the point belongs to a set of exponentially small measure. Using the fact that $\frac{1}{n_j}\log\lambda(E_{n_j})\to 0$, we can compute the portion of ``bad points". Note that the dynamics of the disk is being used here, as for purely analytic reasons, one should not expect to bound the $\lambda_n$-portion of a subset of $E_n$ whose $\lambda$-measure ``is large" relatively to $E_n$, as these are morally two singular measures.

\subsection{Absolute Continuity of the Conditionals}\label{ACofConds}
For a $C^{1}$ map $\varpi:[0,1]^k\to \mathbf M$ with a positive $k$-volume, i.e. $\lambda(\varpi):=\int_{[0,1]^k}|\wedge^kd_t\varpi|\,dt>0 $, we define $\lambda_\varpi$  by 
\begin{equation}\label{pushforwardmeasure}
    \forall h\in C(\mathbf M), \ \int h\,d\lambda_\varpi:=\int_{[0,1]^k}h\circ \varpi(t)\cdot |\wedge^kd_t\varpi|\,dt.
\end{equation}

In addition, the tangent space to $\mathrm{Im}(\varpi)$ is well-defined $\lambda_\varpi$-a.e. since the set of points in $\mathrm{Im}(\varpi)$ where it intersects itself transversely is of measure zero.

In this section of the paper we prove that for $\widehat{\mathbf{p}}$-a.e. $\varpi$, the probability  measure $\varrho_\varpi\circ\varpi^{-1}$ is absolutely continuous w.r.t. $\lambda_{\varpi}$, where $\widehat{\mu}=\int_{\mathcal{W}} \varrho_\varpi\circ \varpi^{-1} d\widehat{\mathbf{p}}(\varpi)$.

\begin{definition}\label{DefNuVarpiEll}
    For any $\varpi$ as above 
    and $h\in C(\mathbf{M})$, set for $\ell\geq0$,
    $$\nu_\varpi^{(\ell)}(h):=\int \rho_\varpi^{(\ell)}(x)\cdot h(x)d\lambda_\varpi(x),$$
    where 
$$\rho_\varpi^{(\ell)}(x):=\frac{1}{\int \prod_{j\leq \ell}e^{Q_\varpi^{(j)}(y,x)}d\lambda_\varpi(y)},$$
    and 
    $$Q_\varpi^{(j)}(y,x):= \log | \wedge^kd_{g^{-j}y}g^{-1}|_{g^{-j}[\varpi]}| - \log | \wedge^kd_{g^{-j}x}g^{-1}|_{g^{-j}[\varpi]}|.$$   
\end{definition}

\medskip
\noindent\textbf{Remarks:}\text{ }
\begin{enumerate}
    \item $Q_\varpi^{(j)}(y,x)$ is well-defined $\lambda_\varpi$-a.e. by the preceding discussion.
    \item Note, $\nu_\varpi^{(\ell)}=\frac{\lambda_{g^{-\ell}[\varpi]}\circ g^{-\ell}}{\lambda(g^{-\ell}[\varpi])}$.
\end{enumerate}

\begin{prop}\label{nuEsts}
	For $\widehat{\mathbf{p}}$-a.e. $\varpi$, the limit $\lim_{\ell\to\infty}\rho_\varpi^{(\ell)}=\rho_\varpi$ exists (recall Definition \ref{DefNuVarpiEll}), and converges uniformly exponentially fast.  
      In addition, $|\log \rho_\varpi|\leq  1$.
\end{prop}
\begin{proof}
We prove that for $\widehat{\mathbf{p}}$-a.e. $\varpi$, for all $s,t\in [0,1]^k$ with $|\wedge^kd_t\varpi|, |\wedge^kd_s\varpi|\neq 0$, for all $j\geq0$, $|\widehat Q_\varpi^{(j)}(s,t)|\leq \frac{1}{2^{j+1}}$, where $$\widehat Q_\varpi^{(j)}(s,t):=\log\frac{|\wedge^k d_s g^{-j-1}\circ \varpi |}{|\wedge^k d_s g^{-j}\circ \varpi|}-\log\frac{|\wedge^k d_t g^{-j-1}\circ \varpi |}{|\wedge^k d_t g^{-j}\circ \varpi|}.$$
It will conclude the proof as for $\lambda_\varpi$-a.e. $x$ and $y$, there are $t,s$ with $\varpi(t)=x$ and $\varpi(s)=y$, and $|\wedge^kd_t\varpi|, |\wedge^kd_s\varpi|\neq 0$, so that $\widehat Q_\varpi^{(j)}(s,t)=Q_\varpi^{(j)}(x,y)$.

Given $M\in \mathbb{N}$, aside for at most measure $\delta>0$, for $p^M$-a.e. $(\varpi,\varrho)$,
$$\varpi=\lim_{n\to\infty}^{C^{r-1,1}}\varpi_{\ell_n}(x_n),$$
where $\ell_n\geq \delta n$, and $\varpi_{\ell_n}(x_n)\in \mathcal{S}_M$. Therefore, since for every $j$, $\widehat Q_\varpi^{(j)}$ depends $C^1$-continuously on $\varpi$, it is enough to show that for all $n$, $|\widehat Q_{\varpi_{\ell_n}(x_n)}^{(j)}(s,t)|\leq \frac{1}{2^j}$.

Fix $n$. Given $s,t\in [0,1]^k$, write $t_j:=\varpi_{\ell_n-j}(x_n)^{-1}(g^{-j}(\varpi_{\ell_n}(x_n)(t)))$ and $s_j:=\varpi_{\ell_n-j}(x_n)^{-1}(g^{-j}(\varpi_{\ell_n}(x_n)(s)))$. Then,
$$\widehat Q_{\varpi_{\ell_n}(x_n)}^{(j)}(s,t)=\widehat  Q^{(0)}_{\varpi_{\ell_n-j}(x_n)}(s_j,t_j).$$

By arguing as in Lemma \ref{YomdinPtnShrinksUni} we can write $\varpi_{\ell_n-j}:=\varpi_{\ell_n-j}(x_n)=\text{Im}(\sigma_{\ell_n-j}\circ \theta_{\ell_n-j})$ and $\varpi_{\ell_n}:=\varpi_{\ell_n}(x_n)=\text{Im}(\sigma_{\ell_n}\circ\theta_{\ell_n})$ with $ \theta_{\ell_n}=\theta_{\ell_n-j}\circ \phi_j$ where $\|d\phi_j\|\leq \frac{1}{2^j}$. By the definition of $s_j$ and $t_j$, we can write $s_j=\phi_j(s)$ and $t_j=\phi_j(t)$. Thus, 
\begin{equation}\label{lastref}|s_j-t_j|\leq \frac{\sqrt{k}}{2^j},\end{equation}
and 
\begin{align*}
    &\widehat Q^{(0)}_{\varpi_{\ell_n-j}(x_n)}(s_j,t_j)=\\
    &\log\frac{|\wedge^k d_{s_j} g^{-1}\circ \varpi_{\ell_n-j} |}{|\wedge^k d_{s_j} \varpi_{\ell_n-j}|}-\log\frac{|\wedge^k d_{t_j} g^{-1}\circ \varpi_{\ell_n-j} |}{|\wedge^k d_{t_j} \varpi_{\ell_n-j}|}=\\
    &\log\frac{|\wedge^k d_{\theta_{\ell_n-j}(s_j)} g^{-1}\circ \sigma_{\ell_n-j} |}{|\wedge^k d_{\theta_{\ell_n-j}(s_j)} \sigma_{\ell_n-j}|}-\log\frac{|\wedge^k d_{\theta_{\ell_n-j}(t_j)} g^{-1}\circ \sigma_{\ell_n-j} |}{|\wedge^k d_{\theta_{\ell_n-j}(t_j)} \sigma_{\ell_n-j}|}=\\
    &\log\frac{|\wedge^k d_{\theta_{\ell_n-j}(s_j)} g^{-1}\circ \sigma_{\ell_n-j} |}{|\wedge^k d_{\theta_{\ell_n-j}(t_j)} g^{-1}\circ \sigma_{\ell_n-j}|}+\log\frac{|\wedge^k d_{\theta_{\ell_n-j}(t_j)} \sigma_{\ell_n-j} |}{|\wedge^k d_{\theta_{\ell_n-j}(s_j)} \sigma_{j}|}=\\
    &\log\frac{|\wedge^k d_{\theta_{\ell_n-j}(s_j)}  \sigma_{\ell_n-j-1} |}{|\wedge^k d_{\theta_{\ell_n-j}(t_j)}  \sigma_{\ell_n-j-1}|}+\log\frac{|\wedge^k d_{\theta_{\ell_n-j}(t_j)} \sigma_{\ell_n-j} |}{|\wedge^k d_{\theta_{\ell_n-j}(s_j)} \sigma_{\ell_n-j}|},
\end{align*}
as $g^{-1}\circ \sigma_{\ell_n-j}$ is simply $\sigma_{\ell_n-j-1}$. By \eqref{bbdDistort1.5July35}, we get 
\begin{equation}\label{numerator}
    |\widehat Q^{(0)}_{\varpi_{\ell_n-j}(x_n)}(s_j,t_j)|\leq2\cdot \frac{\sqrt k}{2^j}\cdot \frac{1}{5k}\leq \frac{1}{2^{j+1}}.
\end{equation}

\end{proof}

\medskip
\noindent\textbf{Remark:} Note that Proposition \ref{nuEsts} holds also for $\varpi'$ an embedded $k$-disk with $\mathrm{Im}(\varpi')\subseteq \mathrm{Im}(\varpi)$.

\begin{definition}\label{DefNuVarpi} For any $\varpi\in \mathrm{Supp}(\widehat{\mathbf{p}})$ s.t. $\rho_\varpi$ is well-defined (recall Proposition \ref{nuEsts} and Definition \ref{DefNuVarpiEll}) and $h\in C(\mathbf{M})$, set 
     $$\nu_\varpi(h):=\int \rho_\varpi(x)\cdot h (x)d\lambda_\varpi(x),$$
    where 
$$\rho_\varpi=\lim_{\ell\to\infty}\rho_\varpi^{(\ell)}.$$
\end{definition}

\medskip
\noindent\textbf{Remark:} Note that in particular, $\int \rho_\varpi(x)d\lambda_\varpi(x)=1$, and so $\nu_\varpi$ is a probability measure.

\begin{definition}\label{badSet25}
	Given $n,N\in\mathbb{N}$ with $N|n$, and $\delta>0$ with $\delta \frac{n}{N}\in \mathbb{N}$, and $x\in E_n$ s.t.  $$\frac{1}{n}\#\Big\{\ell\leq n: \frac{\lambda(P_n^{\ell+N}(x)\cap E_n)}{\lambda(P_n^\ell(x)\cap E_n)}\geq \frac{1}{\delta^3} \frac{\lambda(P_n^{\ell+N}(x))}{\lambda(P_n^\ell(x))}\Big\}\geq\delta,$$
set $\underline{c}^t(x)\in \{0,1\}^\frac{n}{N}$ which has exactly $\delta \frac{n}{N}$-many $1$'s to be the first chain in lexicographic order such that for $0\leq t<N$ $$c_{i}^t(x)=1\Rightarrow \frac{\lambda(P_n^{(i+1)N+t}(x)\cap E_n)}{\lambda(P_n^{i N+t}(x)\cap E_n)}\geq \frac{1}{\delta^3}\frac{\lambda(P_n^{(i+1)N+t}(x))}{\lambda(P_n^{iN+t}(x))}.$$
For $\underline{c}\in \{0,1\}^\frac{n}{N}$ with exactly $\delta\frac{n}{N}$-many $1$'s, set
	$$[\underline{c}]_{n,N,\delta,t}:=\Big\{x\in E_n: \underline{c}^t(x)=\underline{c}\Big\}.$$
Let $\mathcal{C}_{n,N,\delta}:=\Big\{\underline{c}\in \{0,1\}^\frac{n}{N}:\text{ has exactly }\delta \frac{n}{N}\text{-many }1\text{'s}\Big\}$.
We often abuse notation, and omit the dependence on $n,N$  and $\delta$ 
when it is clear from context. For a lean notation, we write $$[\underline{c}^t]:=[\underline{c}]_{n,N,\delta,t}\text{ and }\mathcal{C}:=\mathcal{C}_{n,N,\delta}.$$
\end{definition}

\medskip
\noindent\textbf{Remark:}
 For a lean notation, from here onwards we assume that $N|n$ and $\delta \frac{n}{N}\in\mathbb{N}$. This can always be arranged by taking integer values which do not affect the proof but carry notation.

\medskip

\begin{prop}[Main estimate]\label{keyAC}
	Given $n,N\in\mathbb{N}$ and $\delta>0$ as in Definition \ref{badSet25}, 
    for all $\delta>0$ small enough, for all $\underline{c}\in \mathcal{C}$ and $0\leq t< N$,
	$$\lambda([\underline{c}^t])\leq e^{-3\delta\log\frac{1}{\delta}\frac{n}{N}}\text{, and }\#\mathcal{C}\leq e^{2\delta\log\frac{1}{\delta} \frac{n}{N}}.$$
\end{prop}
\begin{proof}
	The bound on $\#\mathcal{C}$  is  standard (see e.g. \cite[Lemma~16.19]{flum}). We continue to bound $\lambda([\underline{c}^t]) $ for $\underline{c}\in \mathcal{C}$ and a fixed $0\leq t<N$.

Let $E_{\underline{c}^t}^{i}:=\bigcup\{P\in P_n^{Ni+t}:P\cap [\underline{c}^t]\neq\varnothing \}$. Note, this is a nested sequence of sets in $i$. In addition, when $i_j$ is the $j$-th $i$ so $c_i^t=1$,
\begin{align*}
	\frac{\lambda\Big(E_{\underline{c}^t}^{i_j+1}\Big)}{\lambda\Big(E_{\underline{c}^t}^{i_j}\Big)}=&\frac{\sum_{\overset{P\in P_n^{N i_j+t}:}{P\cap[\underline{c}^t]\neq\varnothing}}\lambda(P) \sum_{\overset{P'\in P_n^{N (i_j+1)+t}:}{P'\cap[\underline{c}^t]\neq\varnothing, P'\subseteq P}}\frac{\lambda(P')}{\lambda(P)}}{\sum_{\overset{P\in P_n^{N i_j+t}:}{P\cap[\underline{c}^t]\neq\varnothing}}\lambda(P)}\\
	\leq &\frac{\sum_{\overset{P\in P_n^{N i_j+t}:}{P\cap[\underline{c}^t]\neq\varnothing}}\lambda(P) \sum_{\overset{P'\in P_n^{N (i_j+1)+t}:}{P'\cap[\underline{c}^t]\neq\varnothing, P'\subseteq P}}\delta^3\frac{\lambda(P'\cap E_n)}{\lambda(P\cap E_n)}}{\sum_{\overset{P\in P_n^{N i_j+t}:}{P\cap[\underline{c}^t]\neq\varnothing}}\lambda(P)}\leq \delta^3.
\end{align*}

Then,
\begin{align*}
\frac{\lambda\Big(E_{\underline{c}^t}^{\frac{n}{N}-1}\Big)}{\lambda\Big(E_{\underline{c}^t}^{0}\Big)}=&\prod_{i=0}^{\frac{n}{N}-2} \frac{\lambda\Big(E_{\underline{c}^t}^{i+1}\Big)}{\lambda\Big(E_{\underline{c}^t}^{i}\Big)}\leq \prod_{j\leq \frac{n\delta}{N}} \frac{\lambda\Big(E_{\underline{c}^t}^{i_j+1}\Big)}{\lambda\Big(E_{\underline{c}^t}^{i_j}\Big)}
\leq  e^{-3\delta\log\frac{1}{\delta}\frac{n}{N}}.
\end{align*}
	Then since $\lambda\Big(E_{\underline{c}^t}^{0}\Big) \leq \lambda(D)\leq 1$, and since $[\underline{c}^t]\subseteq E_{\underline{c}^t}^{\frac{n}{N}-1} $, we have 
	$$\lambda([\underline{c}^t])\leq e^{-3\delta\log\frac{1}{\delta}\frac{n}{N}}.$$
\end{proof}

\begin{definition}
	Given $n,N\in\mathbb{N}$ and $\ell\leq n-N$, set $$\check{E}_n^{\ell,N,\delta}:=\Big\{x\in E_n: \frac{\lambda(P_n^{\ell+N}(x)\cap E_n)}{\lambda(P_n^\ell(x)\cap E_n)}\geq \frac{1}{\delta^3} \frac{\lambda(P_n^{\ell+N}(x))}{\lambda(P_n^{\ell}(x))}\Big\}.$$ 
\end{definition}

\begin{definition} Given $\delta\in\mathbb{Q}^+$, set
\begin{enumerate}
\item $N_n:=\lfloor\sqrt{n}\rfloor$,
\item $\check{A}_{n,\delta}:= 
E_n\setminus \bigcup_{t<N_n,\underline{c}^t\in\mathcal{C}_{n,N_n,\delta}}[\underline{c}^t]
 $.
\end{enumerate}
\end{definition}

\begin{lemma}\label{A_ndelta}
For all $\delta>0$ and all $n$ large enough,
$$\lambda_n(\check{A}_{n,\delta})\geq 1-\frac{N_n e^{-\delta\log\frac{1}{\delta}\frac{n}{N_n}}}{\lambda(E_n)}\geq  1-\delta.$$
\end{lemma}
\begin{proof}
	By Theorem \ref{keyAC},
\begin{align*}\label{checkLimEq}
	\lambda(\check{A}_{n,\delta})=&\lambda\Big(
	E_n\setminus \bigcup_{t< N_n,\underline{c}^t\in\mathcal{C}_{n,N_n,\delta}}[\underline{c}^t]
	 \Big)\geq\lambda(E_n)-
	 N_n\cdot e^{-\delta\log\frac{1}{\delta}\frac{n}{N_n}}\nonumber\\
	&\geq \lambda(E_n)\cdot\Big(1-\frac{\sqrt{n} e^{-\delta\log\frac{1}{\delta}\sqrt{n}}}{\lambda(E_n)}\Big)\geq(1-\delta)\lambda(E_n),
\end{align*}
	for all $n$ large enough, since $\lambda(E_n) \geq \frac{1}{n^2}$.\footnote{The sub-exponential decay of $\lambda(E_n)$ is enough, but then we may need to choose a more careful $N_n\sim \frac{1}{\frac{-1}{n}\log\lambda(E_n)+\frac{\log n}{n}}\to\infty$. We need to satisfy $\frac{N_n}{\lambda(E_n)}e^{-c_\delta \frac{n}{N_n}}\to 0$, while $N_n\to\infty$.}
\end{proof}

\begin{lemma}\label{goodElls}
	For all $\delta>0$ and $n$ large enough, $$\frac{1}{n}\#\{\ell\leq n-N_n: \lambda_n(E_n\setminus\check{E}_n^{\ell,N_n,\delta})\leq1-\sqrt{2\delta}\}\leq\sqrt{2\delta}.$$
\end{lemma} 
\begin{proof}

Note that for $x\in\check{A}_{n,\delta}$  $$\frac{1}{n}\sum_{\ell\leq n-N_n}\mathbb{1}_{\check{E}_n^{\ell,N_n,\delta}}(x)\leq \delta.$$
For all $n$ large enough, by Lemma \ref{A_ndelta}, 
$$ \lambda_n(E_n\setminus\check{A}_{n,\delta})\leq \delta$$
Therefore 
\begin{align*}
	\frac{1}{n}\sum_{\ell\leq n-N_n}\lambda_n(\check{E}_n^{\ell,N_n,\delta})=&\int \frac{1}{n}\sum_{\ell\leq n-N_n}\mathbb{1}_{\check{E}_n^{\ell,N_n,\delta}}\, d\lambda_n\\
	\leq& \lambda_n(E_n\setminus\check{A}_{n,\delta})+\int_{\check{A}_{n,\delta}} \frac{1}{n}\sum_{\ell\leq n-N_n}\mathbb{1}_{\check{E}_n^{\ell,N_n,\delta}}\, d\lambda_n\\
    \leq & 2\delta.
\end{align*}
Finally, by the Markov inequality,
\begin{equation*}\label{qBound}
	\frac{1}{n}\#\{\ell\leq n-N_n: \lambda_n(E_n\setminus\check{E}_n^{\ell,N_n,\delta})\leq1-\sqrt{2\delta}\}\leq\sqrt{2\delta}.
\end{equation*}
\end{proof}

\begin{lemma}\label{goodPs}
	For all $\delta>0$ and $n$ large enough, for every $\ell\leq n-N_n$ s.t.  $\lambda_n(E_n\setminus\check{E}_n^{\ell,N_n,\delta})\geq 1-\sqrt{2\delta}$, 
	\begin{align*}
    &\lambda\left(\bigcup \left\{P\cap E_n: P\in P_n^\ell, \frac{\lambda(P\cap (E_n\setminus \check{E}_n^{\ell,N_n,\delta}))}{\lambda(P\cap E_{n})}\leq 1-\sqrt{\sqrt{{2\delta}}}\right\}\right)\\
    &\leq \lambda(E_n) \sqrt{\sqrt{{2\delta}}}.
    \end{align*}
\end{lemma}
\begin{proof}
 By the Markov inequality, for every $\ell\leq n-N_n$ s.t.  $\lambda_n(E_n\setminus\check{E}_n^{\ell,N_n,\delta})\geq 1-\sqrt{2\delta}$ (recall Lemma \ref{goodElls}),
\begin{align*}
&\lambda\left(\bigcup\left\{P\cap E_n: P\in P_n^\ell, 1-\frac{\lambda(P\cap (E_n\setminus \check{E}_n^{\ell,N_n,\delta}))}{\lambda(P\cap E_{n})}\geq \sqrt{\sqrt{{2\delta}}}\right\}\right)\\
&\leq\frac{1}{\sqrt{\sqrt{{2\delta}}}}\sum_{P\in P_n^\ell}\lambda(P\cap E_n)-\lambda(P\cap (E_n\setminus \check{E}^{\ell,N_n,\delta}))=\frac{\lambda(\check{E}_n^{\ell,N,\delta})}{\sqrt{\sqrt{2\delta}}}\\
&=\frac{\lambda(E_n)\left(1-\left(1-\frac{\lambda(\check{E}_n^{\ell,N_n,\delta})}{\lambda(E_n)}\right)\right)}{\sqrt{\sqrt{{2\delta}}}}\leq \lambda(E_n) \sqrt{\sqrt{{2\delta}}}.
\end{align*}
\end{proof}

\begin{definition}\label{pCheck}\text{ }Given $\delta>0$, $n$, $\ell\leq n-N_n$, set
	\begin{enumerate}
	\item $\check{P}_n^{\ell,N_n,\delta}:=\Big\{P\in P_n^\ell: \frac{\lambda(P\cap (E_n\setminus \check{E}_n^{\ell,N_n,\delta}))}{\lambda(P\cap E_n)}\geq 1-\sqrt{\sqrt{2\delta}}\Big\}$,
	\item Given $x\in E_n$,  $$\varrho_{\ell}^\delta(x):= \frac{\lambda_{P_n^\ell(x)\cap (E_n\setminus \check{E}_n^{\ell,N_n,\delta})}\circ g^{-\ell}}{\lambda(P_n^\ell(x)\cap E_n)} \circ \varpi_\ell(x),$$
	\item 	$$\check{p}_n^{M,\delta}:=\int_{\check{A}_{n,\delta}}\frac{1}{n}\sum_{\overset{\ell\in \mathrm{BG}_ {n-N_n} ^M(x):}{P_n^\ell(x)\in 
\check{P}_n^{\ell,N_n,\delta} }}\delta_{(\varpi_\ell(x),\varrho_\ell^\delta(x))}d\lambda_n.$$
	\end{enumerate}
\end{definition}

\medskip
\noindent\textbf{Remark:}\text{ }
\begin{enumerate}
	\item The measures $\varrho_\ell^\delta(x)$ are well-defined, as $\varpi_\ell(x)$ is invertible.
	\item The measures $\varrho_\ell^\delta(x)$ may not be probability measures, but are sub-probability measures (which is a still a compact collection of measures).
	\item It is clear that $0<\delta'\leq\delta\Rightarrow \varrho_\ell^{\delta'}(x)\geq \varrho_\ell^\delta(x) $.
	\item By the diagonal argument, we may assume w.l.o.g. that for any $M\in \mathbb N$ and for any $\delta\in \mathbb Q^+$ the measures $\{\check{p}_n^{M,\delta}\}_{n}$ converge to a limit on the same subsequence $\{n_j\}_j=\mathcal{N}$ from Lemma \ref{preBC}.
\end{enumerate}

\begin{definition}\text{ }
\begin{enumerate}
	\item Let $\mathbb{P}_{\leq1}([0,1]^k)$ be the set of sub-probability measures on $[0,1]^k$,
\item Let $\pi:\mathcal{W}\times\mathbb{P}_{\leq 1}([0,1]^k) \to \mathcal{W}$ be the projection onto $\mathcal{W}$,
\item  Let $\check{p}^{M,\delta}=\lim_{j} \check{p}_{n_j}^{M,\delta} $ and  $\check{q}^{M,\delta}=\check{p}^{M,\delta}\circ \pi^{-1}$,
\item Let $\check{p}^{M,\delta}=\int \check{p}_{\varpi}^{M,\delta}d\check{q}^{M,\delta}(\varpi)$ be the disintegration of $\check{p}^{M,\delta}$ 
w.r.t. the measurable partition  $\left\{\pi^{-1}(\{\varpi\}): \ \varpi\in \mathcal W\right\}$,
\item Set $\varrho_\varpi^{M,\delta}:= \int \varrho d\check{p}_{\varpi}^{M,\delta}$.
\end{enumerate}	
\end{definition}

\begin{lemma}\label{checkLim2}
	For all $0<\delta'\leq\delta$, for all $M\in \mathbb{N}$, 
    $$\check{q}^{M,\delta'}\geq \check{q}^{M,\delta}.$$
\end{lemma}
\begin{proof}
By Lemma \ref{A_ndelta},    
\begin{align*}
	\check{p}_n^{M,\delta'}\circ\pi^{-1}-\check{p}_n^{M,\delta}\circ\pi^{-1}=&\int_{\check{A}_{n,\delta'}}\frac{1}{n}\sum_{\overset{\ell\in \mathrm{BG}_ {n-N_n} ^M(x):}{P_n^\ell(x)\in 
\check{P}_n^{\ell,N_n,\delta'} }}\delta_{\varpi_\ell(x)}\, d\lambda_n\\
&-\int_{\check{A}_{n,\delta}}\frac{1}{n}\sum_{\overset{\ell\in \mathrm{BG}_{n-N_n}^M(x):}{P_n^\ell(x)\in 
\check{P}_n^{\ell,N_n,\delta} }}\delta_{\varpi_\ell(x)}\,d\lambda_n\\
\geq &-\lambda_n(\check{A}_{n,\delta'})-\lambda_n(\check{A}_{n,\delta})\\
+& 
	\int \frac{1}{n}\sum_{\overset{\ell\in \mathrm{BG}_{n-N_n}^M(x):}{P_n^\ell(x)\in 
	\check{P}_n^{\ell,N_n,\delta'}\setminus 
	\check{P}_n^{\ell,N_n,\delta}}} \delta_{
	\varpi_\ell(x)
	}d\lambda_n\\
    \geq &
-\frac{N_ne^{-\delta'\log\frac{1}{\delta'}\frac{n}{N_n}}}{\lambda(E_n)}-\frac{N_n e^{-\delta\log\frac{1}{\delta}\frac{n}{N_n}}}{\lambda(E_n)}\\
	+&\int \frac{1}{n}\sum_{\overset{\ell\in \mathrm{BG}_{n-N_n}^M(x):}{P_n^\ell(x)\in 
	\check{P}_n^{\ell,N_n,\delta'}\setminus 
	\check{P}_n^{\ell,N_n,\delta}}} \delta_{
	\varpi_\ell(x)
	}d\lambda_n\\
	\geq &-2 \frac{N_n e^{-\delta'\log\frac{1}{\delta'}\frac{n}{N_n}}}{\lambda(E_n)}\geq -2 n^\frac{5}{2} e^{-\delta'\log\frac{1}{\delta'}\sqrt{n}} \xrightarrow[n\to\infty]{}0.
\end{align*}

\end{proof}

\begin{lemma}\label{checkLim}
For all $M\in\mathbb{N}$ and $\delta\in \mathbb{Q}^+$, 
	$$0\leq p^M\circ\pi^{-1}-\check{q}^{M,\delta}\leq 4\sqrt{\sqrt{2\delta}}.$$
\end{lemma}
\begin{proof}
	Recall that by Lemma \ref{A_ndelta}, $\lambda_n(\check{A}_{n,\delta})\geq1-\delta$
 	for all $n$ large enough. Then, by Lemma \ref{goodElls} and by Lemma \ref{goodPs},
\begin{align*}
	\check{p}_n^{M,\delta}\circ\pi^{-1}\geq& -\sqrt{2\delta}-\sqrt{\sqrt{2\delta}}
    -\frac{N_n}{n} +\int_{\check{A}_{n,\delta}}\frac{1}{n}\sum_{\ell\in \mathrm{BG}_n^M(x)}\delta_{\varpi_\ell(x)}d\lambda_n \\
	\geq& -4\sqrt{\sqrt{2\delta}}+ \int\frac{1}{n}\sum_{\ell\in \mathrm{BG}_n^M(x)} \delta_{\varpi_\ell(x)} d\lambda_n =-4\sqrt{\sqrt{2\delta}} +p_n^M\circ\pi^{-1}.
\end{align*}
Since $\pi$ is continuous, $\check{q}^{M,\delta}=\check{p}^{M,\delta}\circ\pi^{-1}\geq -4\sqrt{\sqrt{2\delta}} +p^M\circ \pi^{-1}$. The inequality $\check{p}_n^{M,\delta}\circ\pi^{-1}\leq p_n^M\circ\pi^{-1} $ is clear. 
\end{proof}

\begin{cor}\label{increasingQs}
$$\lim_{\delta\to 0}\uparrow \check{q}^{M,\delta}=p^M\circ \pi^{-1}.$$
\end{cor}
\begin{proof}
	It is a consequence of Lemma \ref{checkLim} and Lemma \ref{checkLim2}.
\end{proof}

Recall Theorem \ref{theCondsDef} and the definition of the conditional measures $\widehat{\mu}=\int \varrho_\varpi\circ \varpi^{-1}d\widehat{\mathbf{p}}(\varpi)$.

\begin{theorem}\label{entForm}
	For 
	$\widehat{\mathbf{p}}$-a.e. $\varpi$,
	$$\varrho_\varpi\circ \varpi^{-1}
	\ll\nu_\varpi.$$
\end{theorem}
\begin{proof}\text{ }

\noindent\textbf{Step 1:}
Let $h\in \mathrm{Lip}(\mathbf{M})$ with $0\leq h\leq 1$. We show that for all $M\in \mathbb{N}$, 
for $\check{p}^{M,\delta}$-a.e. $(\varpi,\varrho)$, $\varrho\circ\varpi^{-1}\leq \frac{
1}{\delta^3}\nu_\varpi
$. 

We start by studying $\check{p}_n^{M,\delta}$. Recall Lemma \ref{YomdinPtnShrinksUni}. Then, for $\check{p}_n^{M,\delta}$-a.e. $(\varpi,\varrho)$ $\exists x\in E_n$ and $\sqrt{n}\leq \ell\leq n-N_n$ 
s.t. for $P:=P^\ell_n(x)=g^{-\ell}[\varpi]$
,
\begin{align}\label{toRevisit}
	\varrho\circ \varpi^{-1}(h)=& \frac{\lambda_{P_n^\ell(x)\cap (E_n\setminus \check{E}_n^{\ell,N_n,\delta})}\circ g^{-\ell}}{\lambda(P_n^\ell(x)\cap E_n)}(h)\\
	=&\sum_{\overset{P'\in P_n^{\ell+N_n}:}{P'\subseteq P, P'\cap \check{E}_n^{\ell,N_n,\delta}=\varnothing}}\frac{\lambda(P'\cap E_n)}{\lambda(P\cap E_n)}\cdot \frac{\lambda_{P'\cap E_n}\circ g^{-\ell}}{\lambda(P'\cap E_n)}(h)\nonumber\\
	=&\pm\mathrm{Lip}(h)\cdot \frac{1}{2^{N_n}}+ \sum_{\overset{P'\in P_n^{\ell+N_n}:}{P'\subseteq P, P'\cap \check{E}_n^{\ell,N_n,\delta}=\varnothing}}\frac{\lambda(P'\cap E_n)}{\lambda(P\cap E_n)}\cdot h(x_{g^{\ell}[P']})\nonumber\\
	\leq& \mathrm{Lip}(h)\cdot \frac{1}{2^{N_n}}+ \sum_{\overset{P'\in P_n^{\ell+N_n}:}{P'\subseteq P, P'\cap \check{E}_n^{\ell,N_n,\delta}=\varnothing}}\frac{1}{\delta^3}\frac{\lambda(P')}{\lambda(P)}\cdot h(x_{g^{\ell}[P']})\label{toRevisitthree}\\
	\leq& \mathrm{Lip}(h)\cdot \frac{1}{2^{N_n}}+ \frac{1}{\delta^3}\frac{1}{\lambda(P)}\sum_{\overset{P'\in P_n^{\ell+N_n}:}{P'\subseteq P}}\lambda(P') \cdot h(x_{g^{\ell}[P']})\label{toRevisittwo}\\
	=& \pm\frac{2}{\delta^3}\mathrm{Lip}(h)\cdot \frac{1}{2^{N_n}}+ \frac{1}{\delta^3}\nu_\varpi^{(\ell)}
    (h).\nonumber
\end{align}

Recall, $\ell\geq \sqrt{n}$. Therefore, for any $\tau>0$, the property $\varrho\circ \varpi^{-1}(h)\leq \tau + \frac{
1}{\delta^3}\nu_\varpi^{(\ell)}
(h) $ is a closed property, hence for $\check{p}^{M,\delta}$-a.e. $(\varpi,\varrho)$, for all $\tau>0$,
$$\varrho\circ \varpi^{-1}(h)\leq \tau + \frac{
1}{\delta^3}\nu_\varpi
(h).$$
Since $\tau>0$ was arbitrary, and $0\leq h\leq 1$ is any Lipschitz function, \begin{equation}\label{eqStep1}
\text{for }\check{p}^{M,\delta}\text{-a.e. }(\varpi,\varrho)\text{, }\varrho\circ \varpi^{-1}\leq\frac{
1}{\delta^3}\nu_\varpi
.
 \end{equation}

\medskip
\noindent\textbf{Step 2:} By \eqref{eqStep1}, for $\check{q}^{M,\delta}$-a.e. $\varpi$,
\begin{equation}
	\varrho_\varpi^{M,\delta}\circ\varpi^{-1}=\int \varrho\circ\varpi^{-1}d\check{p}^{M,\delta}_\varpi\leq \frac{
    1}{\delta^3}\nu_\varpi
\end{equation}
Recall, for $\check{p}^{M,\delta}$-a.e. $(\varpi,\varrho)$, $\varrho(1)\in [1-2\delta^\frac{1}{4},1]$, and so for $\check{q}^{M,\delta}$-a.e. $\varpi$, $\varrho_\varpi^{M,\delta} (1)\in [1-2\delta^\frac{1}{4},1] $. 

In addition, $\varrho_\varpi^{M,\delta}$ increases as $\delta$ decreases, and as $M$ increases (recall the remark following Definition \ref{pCheck}). In particular, by Corollary \ref{increasingQs}, for $\mathbf{p}\circ \pi^{-1}$-a.e. $\varpi$,
$$\lim_{M\uparrow\infty}\uparrow\lim_{\delta\downarrow 0}\uparrow\varrho_\varpi^{M,\delta}=\int \varrho d\mathbf p_\varpi,$$
where $\mathbf p=\int \mathbf p_\varpi d\mathbf p\circ\pi^{-1}$ is the disintegration of $\mathbf{p}$ w.r.t. the measurable partition  $\left\{\pi^{-1}[\{\varpi\}]: \ \varpi\in \mathcal W\right\}$.
Hence, since an increasing limit of absolutely continuous measures which are bounded is absolutely continuous, 
$$\text{for }\mathbf p\circ \pi^{-1}\text{-a.e. }\varpi,
\int \varrho \circ \varpi^{-1}d\mathbf p_\varpi\ll \nu_\varpi.$$
Recalling the definition of $\widetilde{\mathbf{p}}$ from Lemma \ref{fromGtoF} and then of $\widehat{\mathbf{p}}$ and $\varrho_\varpi$ from Theorem \ref{theCondsDef}, we are done.
\end{proof}

By combining Theorem \ref{entForm} and Theorem \ref{theCondsDef} one easily concludes the proof of Theorem \ref{crcr}.

\medskip
\noindent\textbf{Remark:} In \eqref{toRevisit} we provide an estimate that a posteriori in \textsection \ref{DomCondSect} we are able to improve and use to obtain additional properties. If one just wants to check that $\widehat \mu$  is G-$u$-Gibbs measure, then we may conclude from \eqref{toRevisittwo} without referring  to Proposition \ref{nuEsts} by using  the bounded distortion property in \eqref{bdist} : 
$\frac{\lambda(P')}{\lambda(P)}\leq 2\frac{\lambda_\varpi(g^\ell[P'])}{\lambda_\varpi(g^\ell[P])}=2\frac{\lambda_\varpi(g^\ell[P'])}{\lambda(\varpi)}$. Indeed 
\begin{align*}
\frac{1}{\delta^3}\frac{1}{\lambda(P)}\sum_{\overset{P'\in P_n^{\ell+N_n}:}{P'\subseteq P}}\lambda(P') \cdot h(x_{g^{\ell}[P']})&\leq \frac{1}{\delta^3}\sum_{\overset{P'\in P_n^{\ell+N_n}:}{P'\subseteq P}}\lambda_\varpi(g^\ell[P']) \cdot h(x_{g^{\ell}[P']}),\\
&\leq \pm\frac{2}{\delta^3}\mathrm{Lip}(h)\cdot \frac{1}{2^{N_n}}+ \frac{2}{\delta^3}\lambda_\varpi
    (h).
\end{align*}

\subsection{Dominating Conditionals}\label{DomCondSect}

In this subsection we post-hoc revisit \textsection \ref{ACofConds}, and conclude domination in addition to absolute continuity.

\begin{prop}\label{EisConstProp}
	We may assume w.l.o.g. that $\{E_{n_j}\}_{j\geq0}$ is a constant sequence $\{E\}$, where $E$ is closed. 
\end{prop}
\begin{proof}
Let $D$, and let $(\varpi_\ell)_{\ell \in \mathbb N}$ be the disks associated with the Yomdin partition. Let $(E_n)_{n\in \mathcal N}\subset D$ be as in the previous sections. We will show that there exists another disk $D'$ and a subset $E$  with $\lambda_{D'}(E)>0$ such that the previous constructions hold for $D'$ and $E'_n=E$ for all $n\in \mathbb N$ large enough. By potentially restricting to a subset, we may assume further w.l.o.g. that $E$ is closed.

Note, given $M\in \mathbb{N}$, aside for at most measure $\delta>0$, for $p^M$-a.e. $(\varpi,\varrho)$,
$$\varpi=\lim_{n\to\infty}^{C^{r-1,1}}\varpi_{\ell_n}(x_n),$$
where $\ell_n\geq \delta n$, and $\varpi_{\ell_n}(x_n)\in \mathcal{S}_M$. Fix $m\in \mathbb{N}$, and write $L:=\mathrm{Lip}(\Xi)$, where
$$\Xi(x):= \frac{1}{p}\sum_{i=0}^{p-1}\log^+\max_{ s\leq k-1}\|\wedge^s d_{f^i(x)}g\|.$$

Let $t_n\in [0,1]^k$ and $\ell_n'\in [\ell_n,\ell_n+M]$ s.t. $\varpi_{\ell_n'}(y_n)$ is BG and $g^{\ell_n}(y_n)=\varpi_{\ell_n}(x_n)(t_n)$.

By \eqref{estim222}, for all $n$ large enough so $\delta n\geq m$, 
\begin{equation}\label{234}
    \frac{1}{m}\log \Big|C_{k,p}^{(-
    m)}\Big(g^{\ell_n}(y_n), T_{g^{\ell_n}(y_n)}\varpi_{\ell_n}(x_n)\Big)\Big|\leq -\widehat{\chi}+\frac{2pk M}{m}\log M_f,
\end{equation}
where $$C_{k,p}
(y,\omega):=\frac{\wedge^k d_y g\omega}{
\max_{s\leq k-1}\|\wedge^sd_{
y}g\|^\frac{1}{p}\vee 1}$$ is a linear cocycle on $\wedge^k T\mathbf{M}$ over $(\mathbf{M},
g)$. 

We continue to show that the estimate of \eqref{234} carries to all of $\varpi_{\ell_n}(x_n)$ uniformly, and so, since $n$ is arbitrary, it carries to every point in $\varpi$.  

Given $a,b\in [0,1]^k$, write $a_j:=\varpi_{\ell_n-j}(x_n)^{-1}(g^{-j}(\varpi_{\ell_n}(x_n)(a)))$ and $b_j:=\varpi_{\ell_n-j}(x_n)^{-1}(g^{-j}(\varpi_{\ell_n}(x_n)(b)))$. Then  $|a_j-b_j|\leq \frac{\sqrt{k}}{2^j}$ (recall \eqref{lastref}), and since $\|\varpi_{\ell_n-j}(x_n)\|_{r}\leq 1$, we get that for all $a,b\in [0,1]^k$, for all $j$
\begin{equation}\label{new22}
    d(g^{-j}(\varpi_{\ell_n}(x_n)(a)),g^{-j}(\varpi_{\ell_n}(x_n)(b)))\leq \frac{\sqrt{k}}{2^{j}},
\end{equation}
and so
\begin{equation}
  \sum_{j=0}^{m-1}|\Xi(g^{-j}(\varpi_{\ell_n}(x_n)(t_n)))-\Xi(g^{-j}(\varpi_{\ell_n}(x_n)(a)))|\leq 2L \sqrt{k}.  
\end{equation}

By  the bounded distortion property (\ref{bdist})  we have 
$$\Big|\log\frac{ |\wedge^k d_{\varpi_{\ell_n}(x_n)(t_n)}g^{-m} T_{\varpi_{\ell_n}(x_n)(t_n)}\varpi_{\ell_n}(x_n)|}{ |\wedge^k d_{\varpi_{\ell_n}(x_n)(a)}g^{-m} T_{\varpi_{\ell_n}(x_n)(a)}\varpi_{\ell_n}(x_n)|}\Big|\leq 2.$$

It follows that for all $y\in \varpi_{\ell_n}(x_n)$,

$$\frac{1}{m}\log |C_{k,p}^{(
-m)}(y, T_y\varpi_{\ell_n}(x_n))|\leq -\widehat{\chi}+\frac{2pkM}{m}\log M_f+\frac{2+2L\sqrt{k}}{m}.$$

Since $n$ was arbitrary, then for all $x,y\in\varpi$, 

\begin{equation}\label{new22b}
   \limsup_{m\to +\infty} \frac{1}{m}\log d(g^{-m}(x),g^{-m}(y)\leq -\log 2.
\end{equation}
and once $T_y\varpi$ is well defined 
\begin{equation*}
   \limsup_{m\to +\infty} \frac{1}{m}\log |C_{k,p}^{(
    -m)}(y, T_y\varpi)|\leq -\widehat{\chi}.
\end{equation*}

Recall the discussion in the beginning of \textsection \ref{ACofConds}, where we explain why the tangent space is well-defined $\lambda_\varpi$-a.e. on $\varpi$. Therefore, since $\delta$ is arbitrary, for $\widehat{\mathbf{p}}$-a.e. $\varpi$, for $\lambda_\varpi$-a.e. $x\in\varpi$,
\begin{equation*}
    \limsup_{m\to +\infty} \frac{1}{m}\log|C_{k,p}^{(-m)}(x,\iota(T_x\varpi))|\leq -
    \widehat{\chi}.
\end{equation*}

Since $
C_{k,p}:\wedge^k T\mathbf{M}\circlearrowleft$ is a linear cocycle over $(\mathbf{M},
g)$, by the Oseledec theorem, we get that for 
$\widehat{\mathbf{p}}$-a.e. $\varpi$, for $\mu_\varpi$-a.e. $x\in\varpi$,
\begin{equation}\label{LyaExps1}
    \chi_{C_{k,p}}(x,\iota(T_x\varpi))=\lim_{m\to +\infty} \frac{1}{m}\log|C_{k,p}^{(m)}(x,\iota(T_x\varpi))|\geq 
    \widehat{\chi}.
\end{equation}

By arguing as in the proof of Lemma \ref{limOnP}, one easily checks that 
 for every $x\in \mathbf{M}$, and for every $\omega \in \wedge^k T_x\mathbf{M}$, 
\begin{equation}\label{AcocCcoc}
    \chi_{C_{k,p}}(x,\omega)\leq p\lim_q\chi_{A_{k,q}}(x,\omega).
\end{equation}

Therefore, by Theorem \ref{entForm}, for $\widehat{\mathbf{p}}$-a.e. $\varpi$, for a $\lambda_{\varpi}$-positive measure set of $x$'s, $\lim_{q\to\infty}\lim_{m\to+\infty} \frac{1}{m}\log |A_{k,q}^{(m)}(x,\iota(T_x\varpi))|\geq \chi_{C_{k,p}}(x,\iota(T_x\varpi))\geq  \frac{\widehat{\chi}}{p}$.
We can now choose $D'$ to be such a $\widehat{\mathbf p}$-typical $\varpi$. Let $B$ be a subset of $D'$ with $\Vol_k(B)>0$ and $N\in \mathbb N$ such that for any $x\in B$ and any $m>N$ we have  $\frac{1}{m}\log |C_{k,p}^{(m)}(x,\iota(T_x\varpi))|>p \chi>p \frac{3k^2}{r-1}R(f)$. Finally, in Definition \ref{En} of $E_n$ we may replace $B_n^p$ by $B$, so that the proof of Corollary \ref{seten} ensures now that we can assume that the new subsets $E_n$ in fact do not depend on $n$. 

\end{proof}

\begin{definition}
	$$\widehat{E}^{\ell,\delta}:=\Big\{x\in E:\forall \ell'\geq\ell, \frac{\lambda(P^{\ell'}(x)\cap E)}{\lambda(P^{\ell'}(x))}\geq 1-\delta \Big\}.$$
\end{definition}

\begin{claim}\label{Leb}
For all $\delta>0$ $\exists\ell_\delta$ s.t.  $\frac{\lambda(\widehat{E}^{\ell_\delta,\delta})}{\lambda(E)}\geq 1-\delta$. In other words, $\frac{\lambda(P^{\ell}(x)\cap E)}{\lambda(P^{\ell}(x))}\xrightarrow{\ell\to\infty}1$ for $\lambda$-a.e. $x\in E$.
\end{claim}
\begin{proof}
	This is the Lebesgue density lemma, where instead of using a differentiation basis of balls, we use the nested Yomdin partitions $P^{\ell}:=\left\{\mathrm{Im}(\sigma\circ \theta_{\mathbf i^\ell}): \theta_{\mathbf i^\ell}\in \Theta_{\mathbf i^\ell},\  \mathbf i^\ell\in \mathcal T_\ell\right\}$.
  
    For the sake of completeness, we present a short  proof. It is enough to show that the set  $\widetilde{E}^\gamma:=\{x\in E: \ \liminf_\ell  \frac{\lambda(P^{\ell}(x)\cap E)}{\lambda(P^{\ell}(x))}< 1-\gamma \}$ has zero $\lambda$-measure for some $\gamma\in (0,\delta)$. Let $\widetilde{F}^\gamma:=\{x\in E: \ \liminf_\ell  \frac{\lambda(P^{\ell}(x)\cap E)}{\lambda(P^{\ell}(x))}\leq 1-\gamma \}$. We fix $\gamma\in(0,\delta)$  s.t.   $\lambda(\widetilde{E}^\gamma)=\lambda(\widetilde{F}^\gamma)$. For any $x\in \widetilde{E}^\delta$ and for any $N$ we let $\ell_N(x):=\min\{\ell \geq N : \  \frac{\lambda(P^{\ell}(x)\cap E)}{\lambda(P^{\ell}(x))}< 1-\gamma  \}$. Then the sets $\widetilde{E}^\gamma(N):=\bigcup_{x\in \widetilde{E}^\gamma}P^{\ell_N(x)}(x)$ satisfy   $\widetilde{E}^\gamma\subseteq \bigcap_N\downarrow \widetilde{E}^\gamma(N) \subseteq\widetilde{F}^\gamma $. To see the last inclusion, note that if a point $x'$ lies in all of $\widetilde{E}^\gamma(N)$, then for infinitely many $\ell$'s, it belongs to $P^\ell(x)$ where $\frac{\lambda(P^{\ell}(x)\cap E)}{\lambda(P^{\ell}(x))}< 1-\gamma$. Moreover as $E$ is compact and the diameter of $P^\ell$ is going to $0$ with $\ell$, we have also $x'\in E$.   Hence $x'$ belongs to $\widetilde{F}^\gamma$. However, it follows from definition that  $\lambda(\widetilde{E}^\gamma)<(1-\gamma)\lambda(\widetilde{E}^\gamma(N))$. Consequently  $\lambda(\widetilde{E}^\gamma)=\lim_N \lambda(\widetilde{E}^\gamma(N))=0$.
\end{proof}

\begin{definition}\label{ellDeltaRef}\text{ }
\begin{enumerate}
    \item Given $\ell\in\mathbb{N}$, set  $$P^\ell_E:=\{P\in P^\ell: P\cap E \neq\varnothing \}.$$
\item    Let $\ell_\delta$ to be the smallest natural number to satisfy Claim \ref{Leb}, and to satisfy $\frac{\lambda(\bigcup P^\ell_E)}{\lambda(E)}\leq 1+\delta$ for all $\ell\geq\ell_\delta$, then set
	$$\widehat{E}^{\delta}:= \widehat{E}^{\ell_\delta,\delta}.$$

\end{enumerate}
\end{definition}

\medskip
\noindent\textbf{Remark:} Definition \ref{ellDeltaRef} is proper, as since $E$ is compact, $E=\bigcap_{\ell\geq0}\bigcup P_E^\ell$.

\begin{cor}\label{denseEnInPnell} For all $\delta>0$, for all $\ell_\delta\leq \ell$, for all $N\in\mathbb{N}$, for all $x\in \widehat{E}^\delta$
	$$\frac{\lambda(P
    ^\ell(x)\cap E)}{\lambda(P
    ^\ell(x))}= \frac{\lambda(P
    ^{\ell+N}(x)\cap E)}{\lambda(P
    ^{\ell+N}(x))}=e^{\pm\delta}.$$
\end{cor}

\begin{definition}\label{pHat}\text{ }
	\begin{enumerate}
	\item 	$\widehat{E}^{\ell,N,\delta}:=\bigcup\{P\in P
    ^{\ell+N}: \frac{\lambda(P\cap E)}{\lambda(P)}\geq 1-\delta \}$,
	\item $\widehat{P}
    ^{\ell,N,\delta}:=\Big\{P
    ^\ell(x):\frac{\lambda(P\cap \widehat{E}^{\ell,N,\delta})}{\lambda(P)}\geq 1-\sqrt{2\delta}\Big\}$,
	\item Given $x\in E
    $, and 
    $\ell,N\in \mathbb{N}$, $$\widehat{\varrho}_{\ell}^\delta(x):= \frac{\lambda_{P
    ^\ell(x)\cap \widehat{E}^{\ell,N
    ,\delta}}\circ g^{-\ell}}{\lambda(P
    ^\ell(x)\cap E)} \circ \varpi_\ell(x),$$
	\item 	$$\overline{p}_n^{M,\delta}:=\int_{
E}\frac{1}{n}\sum_{\overset{\ell_\delta\leq \ell\in \mathrm{BG}_ {n-N_n} ^M(x):}{P^\ell(x)\in 
\widehat{P}^{\ell,N_n,\delta} }}\delta_{(\varpi_\ell(x),\widehat{\varrho}_\ell^\delta(x))}d\lambda.$$
	\end{enumerate}
\end{definition}

\medskip
\noindent\textbf{Remark:} Note, $\widehat E^{\delta}\subseteq \widehat{E}^{\ell,N,\delta}$, when $\ell\geq \ell_\delta$, by Corollary \ref{denseEnInPnell}.

\begin{lemma}\label{abovethis}
	For all $\delta>0$, for all $\ell \geq \ell_\delta$, 
	$$\lambda(\widehat{E}^{\ell, N,\delta})\geq \lambda(\widehat E^\delta)\geq \lambda(\bigcup P_
    E^\ell)-2\delta.$$
\end{lemma}
This is a consequence of Claim \ref{Leb}.

\begin{lemma}	 For all $\delta>0$, for all $\ell \geq \ell_\delta$ and $N\in\mathbb{N}$, 
	$$\lambda\Big(\bigcup P_
    E^\ell\setminus \bigcup\widehat{P}_
    E^{\ell,N
    ,\delta}\Big)\leq \sqrt{2\delta}.$$
\end{lemma}
\begin{proof} By Lemma \ref{abovethis},
\begin{align*}
	\sum_{P\in P_
    E^\ell}\lambda(P)\cdot(1- \frac{\lambda(P\cap \widehat{E}^{\ell,N,\delta})}{\lambda(P)})=\lambda(\bigcup P_
    E^\ell)-\lambda( \widehat{E}^{\ell,N,\delta})\leq 2\delta.
\end{align*}
Then by the Markov inequality,
\begin{align*}
	\lambda\Big(\bigcup \Big\{P\in P_
    E^\ell: \frac{\lambda(P\cap \widehat{E}^{\ell,N,\delta})}{\lambda(P)}\leq 1-\sqrt{2\delta}\Big\}\Big)\leq \sqrt{2\delta}.
\end{align*}
\end{proof}

\begin{lemma}
	For all $M\in\mathbb{N}$, for all $\ell\geq 
    \ell_\delta$, for every $P\in \widehat{P}_
    E^{\ell,N,\delta}\cap P_
    E^{\ell,M}$, where $P_
    E^{\ell,M}$ is defined as $P_
   n^{\ell,M}$ where $E_n\equiv E$, for all $h\in \mathrm{Lip}(\mathbf{M})$ with $0\leq h\leq 1$,
	$$
    \nu_\varpi^{(\ell)}|_{\widehat{E}^{\ell,N,\delta}}(h)= e^{\pm 2\delta}\varrho_\ell^\delta(P)(h)\pm \mathrm{Lip}(h)e^{-\frac{\chi}{2}N},$$
	and $\nu_\varpi(\widehat{E}^{\ell,N,\delta})\geq 1-\sqrt{2\delta}$.
\end{lemma}
\begin{proof}
	This lemma follows from the definition of $\widehat{P}_
    E^{\ell,N,\delta}$, when we follow the scheme of proof of Theorem \ref{entForm}.
\end{proof}

\begin{lemma}
Let $n_j$ s.t.  $\overline{p}_{n_j}^{M,\delta}\xrightarrow[j\to\infty]{}\overline{p}^{M,\delta}$ for all $M\in\mathbb{N}$ and $\delta\in \mathbb{Q}^+$. Then, 
$$\overline{p}^{M,\delta}\circ \pi^{-1}\xrightarrow[\delta\to0]{}p^M\circ \pi^{-1}.$$
\end{lemma}
\begin{proof}
$$0\leq p_n^M\circ \pi^{-1}-\overline{p}_n^{M,\delta}\circ \pi^{-1} \leq \frac{\ell_\delta}{n}+\sqrt{\delta}.$$
\end{proof}

\begin{cor}\label{smoothNuConds}
		For 
	$\widehat{\mathbf{p}}$-a.e. $\varpi$,
	$$\varrho_\varpi\circ \varpi^{-1}=\nu_\varpi.$$
\end{cor}
\begin{proof}
	We may follow the proof of Theorem \ref{entForm}, where the only two inequalities of \eqref{toRevisitthree} and \eqref{toRevisittwo} can now be reversed; While the constant $\frac{1}{\delta^3}$ can be replaced by $e^{\pm \delta}$. Indeed, when sending $\delta\to 0$, we conclude that for $\mathbf p\circ\pi^{-1}$-a.e. $\varpi$, $\int\varrho \circ \varpi^{-1}d\mathbf p_\varpi=\nu_\varpi$. 
\end{proof}

\begin{cor}\label{everyGuGibbsIsSmooth}
		For every G-$u$-Gibbs measure $\widehat{\mu}'
        $, $\widehat{\mu}'$ can be written as $\widehat{\mu}'=\int \mu_{\varpi}d\widehat{\mathbf{p}}(\varpi)$ where for $\widehat{\mathbf{p}}$-a.e. $\mu_\varpi$, $\mu_\varpi=\nu_\varpi$.
\end{cor}
\begin{proof}
We may assume w.l.o.g. that $\widehat{\mu}'$ is ergodic (recall the remark after Definition \ref{GuGibbs}). Therefore, we may assume that the set $E$ (recall Proposition \ref{EisConstProp}) is a subset of the $\widehat{\mu}'$-generic points, and carry out the construction of $\widehat{\mu}$ accordingly. In particular, it follows that $\widehat{\mu}\propto\lim_{M}\lim_{j}\int_E \frac{1}{n_j}\sum_{\ell\in \mathrm{BG}_{n_j}^M(x)}\delta_{g^\ell(x)}d\lambda\leq \widehat{\mu}' $, where $\propto$ denotes proportionality. Since $\widehat{\mu}'$ is ergodic and both $\widehat{\mu}$ and $\widehat{\mu}'$ are probability measures, $\widehat{\mu}=\widehat{\mu}'$. By Corollary \ref{smoothNuConds}, we are done.
\end{proof}

\section{Ergodic and geometric properties of the invariant measure}\label{propOfMuSect}

\subsection{Positive exponents tangent to disks, and disks subordinated to unstable leaves}

In this section we prove some of the geometric and ergodic properties of the measure we construct in \textsection \ref{constMu}. Namely, we show that it admits at least $k$-many exponents larger than $\chi$ almost everywhere. In addition, we prove a lower bound on its entropy by the Lyapunov exponents, which is not given by the standard entropy formula (as we do not assume that the measure is an SRB, nor  a $u$-Gibbs measure, and moreover its disintegration is not given by a measurable partition).

\begin{cor}[$k$ positive exponents]\label{SRB1}
    For $\widehat{\mathbf{p}}$-a.e. $\varpi$, $\mathrm{Im}(\varpi)$ is contained in an  unstable leaf, and $\widehat{\mu}$ admits at least $k$ exponents greater than $\chi$. 
\end{cor}
\begin{proof}
By \eqref{new22b}, 
for $\widehat{\mathbf{p}}$-a.e. $\varpi$, for every $x\in \varpi$, 
\begin{equation*}\label{234567}
   \mathrm{Im}(\varpi)\subseteq \Big\{y\in\mathbf{M}: \limsup\frac{1}{m}\log d(f^{-m}(x),f^{-m}(y))\leq -\frac{\log 2}{p}\Big\}.
\end{equation*}
Hence it is a classical fact from Pesin theory that for $\mu_\varpi$-a.e. $x\in \varpi$, 
$$ \mathrm{Im}(\varpi)\subseteq V^u(x):=
\Big\{y\in\mathbf{M}: \limsup\frac{1}{m}\log d(f^{-m}(x),f^{-m}(y))<0\Big\}.$$

In particular, $$\mathrm{dim}(V^u(x))\geq \mathrm{dim}(\mathrm{Im}(\varpi))=k.$$
In addition, by \eqref{LyaExps1}, for $\mu_\varpi$-a.e. $x$, $\chi_{C_{k,p}}(x,\iota(T_x\varpi))\geq \widehat{\chi}$. One can easily check as in  Lemma \ref{BGtimesAreHyp} that for all $m\in\mathbb{N}$,
$$|C_{k,p}^{(m)}(x,\iota(T_x\varpi))|\leq \|d_xg^m|_\varpi\|_\mathrm{co},$$
and so $\widehat{\mu}$ admits at least $k$ Lyapunov exponents (w.r.t. $f$) greater than $\frac{\widehat{\chi}}{p}>\chi$ a.e., tangent to the disks. 
\end{proof}

\medskip
\noindent\textbf{Remark:} In particular, it follows from Corollary \ref{SRB1} that $\widehat{\mu}$ is a G-$u$-Gibbs measure over $k$-disks. By \eqref{AcocCcoc}, we show further that for $\widehat{\mathbf{p}}$-a.e. $\varpi$, for $\mu_\varpi$-a.e. $x$, $\lim_{q\to\infty}\chi_{A_{k,q}}(x,\iota(T_x\varpi))\geq \frac{\widehat{\chi}}{p}>\frac{3k^2}{r-1}R(f)$.

\subsection{The invariant measure gives no weight to sources}

\begin{definition}
    A point $x$ is called a {\em source} if it is a periodic point with only positive Lyapunov exponents.
\end{definition}

\begin{claim}[No sources]\label{SRB2}
$\widehat{\mu}$ gives $0$ measure to sources.	
\end{claim}
\begin{proof}
It is a classical theorem that every ergodic measure with only positive Lyapunov exponents is carried by a source. However, $\widehat{\mu}$ has absolutely continuous conditionals, hence gives zero measure to periodic orbits. Therefore almost all ergodic components of $\mu$ are not sources.
\end{proof}

\subsection{Entropy and exponents}\label{entropySect}

In this sub-section we prove a lower bound to the entropy by the positive Lyapunov exponents. In particular this implies the presence of negative Lyapunov exponents. Since our measure is not given by a disintegration w.r.t a partition subordinated to an intermediate unstable foliation, we cannot use conditional entropies. The proof relies on a technique of Besicovitch-Bowen covers.

\begin{definition}[Pesin blocks]\label{PesinBlocks}
	Let $\nu$ be an $f$-invariant ergodic probability measure which admits a positive Lyapunov exponent. Denote by \\ $\underline\chi=\underline\chi(\nu):=(\chi_1,m_1,\ldots ,\chi_{\ell_{\underline{\chi}}},m_{\ell_{\underline{\chi}}})$ the positive Lyapunov exponents, and dimensions of the corresponding Oseledec subspaces of $\nu$ in a decreasing order. 
\begin{enumerate}
	\item Let $0<\tau\leq \tau_{\underline{\chi}}:= \frac{1}{100d}\min\{\chi_{\ell_{\underline{\chi}}},\chi_{i}-\chi_{i+1}: i\leq \ell_{\underline{\chi}}-1\}$, and let $C_{\underline{\chi},\tau}(\cdot)$ be the Lyapunov change of coordinates for points in \\ $\mathrm{LR}_{\underline{\chi}}=\{\text{Lyapunov regular points with an index }\underline\chi\}$ (see \cite{KM}).
	\item Let $\mathrm{PR}_{\underline{\chi}}=\{x\in\mathrm{LR}_{\underline{\chi}}:\limsup_{n\to\pm\infty}\frac{1}{n}\log\|C_{\underline{\chi},\tau}^{-1}(f^n(x))\|=0,\forall 0<\tau\leq \tau_ {\underline{\chi}}\}$, the set of {\em $\underline{\chi}$-Pesin regular} points which carries $\mu$. $\mathrm{PR}:=\bigcup_{\underline\chi}\mathrm{PR}_{\underline\chi}$ is called the set of {\em Pesin regular} points.
\item Given $x\in \mathrm{PR}_{\underline{\chi}}$, let $E_j(x)$ be the Oseledec subspace of $x$ corresponding to $\chi_j$.
\item A {\em Pesin block} $\Lambda^{(\underline{\chi},\tau)}_\ell$ is a subset of $\bigcup_{|\underline{\chi}'-\underline\chi|_\infty< \tau}\mathrm{PR}_{\underline\chi'}$ which is a level set $[q_\tau\geq\frac{1}{l}]$ of a measurable function $q_\tau: \bigcup_{|\underline{\chi}'-\underline\chi|_\infty< \tau}\mathrm{PR}_{\underline\chi'}\to (0,1)$ s.t.  (a) $\frac{q_\tau\circ f}{q_\tau}=e^{\pm \tau}$, (b)  $q_\tau(\cdot)\leq\frac{1}{\|C^{-1}_{\underline\chi,\tau}(\cdot)\|^\frac{d}{\beta}}$. Often we omit the subscript $\ell$ when the dependence on $\ell$ is clear from the context.
\end{enumerate}	
\end{definition}

\begin{definition}
        Let $\widehat{\mu}=\int \widehat{\mu}_xd\widehat{\mu}(x)$ be the ergodic decomposition of $\widehat{\mu}$. Set $\chi_1(x)>\ldots>\chi_{u(x)}(x)>0$ to be the list of all positive exponents of $\widehat{\mu}_x$, for $\widehat{\mu}$-a.e. $x$.
\end{definition}

\noindent\textbf{Remark:} Note, by Corollary \ref{SRB1}, $u(x)\geq k$ for $\widehat{\mu}$-a.e. $x$.

\medskip
We are now in a position to prove Theorem \ref{entropy}.

\begin{theorem}[Entropy and exponents]\label{EntAndExp}
Let $\widehat{\mu}$ be a G-$u$-Gibbs measure over $k$-disks. Then for $\widehat{\mu}$-a.e. $x$,
$$h_{\widehat{\mu}_x}(f)\geq \sum_{i=u(x)}^{u(x)-k+1}\chi_i(x).$$
\end{theorem}
\begin{proof}
Let $u\geq k$ s.t.  $\widehat{\mu}(\{x: u(x)=u\})>0$. Let $h\geq 0$
, and let $\chi_1>\ldots>\chi_u $ s.t.  for $\underline\chi=(\chi_1,m_1,\ldots ,\chi_u, m_u)$, and $0<\tau<\tau_{\underline{\chi}}$, we have $\widehat{\mu}(\Omega_\tau)>0$ where w.l.o.g. $\ell\geq \frac{1}{\tau}$ and (recall Definition \ref{PesinBlocks}),
$$\Omega_\tau:=\{x\in \Lambda_\ell^{(\underline{\chi},\tau)}: u(x)=u, h_{\widehat{\mu}_x}(f)=h\pm \tau\}
.$$

By \cite[Lemma~3.1,~Theorem~3.4]{NLE}, for $\widehat{\mu}$-a.e. $x$,
\begin{align*}\label{NLEisE}
  h_{\widehat{\mu}_x}(f)=& \lim\limits_{r\to 0}\limsup\limits_{n\to\infty}\frac{-1}{n}\log\widehat{\mu}(B(x,n,e^{-nr}))\\
  =&\lim\limits_{r\to 0}\liminf\limits_{n\to\infty}\frac{-1}{n}\log\widehat{\mu}(B(x,n,e^{-nr})),\nonumber
\end{align*}
where $B(x,n,e^{-nr}):=\{y:\forall 0\leq j\leq n, d(x,y)\leq e^{-nr}\}$ with $d(\cdot,\cdot)$ being the Riemannian distance on $\mathbf M$. Therefore, there exists $n_\tau\in\mathbb{N}$ and $0<r<\tau$ s.t.  $\widehat{\mu}(K_\tau)\geq (1-\tau)\widehat{\mu}(\Omega_\tau)$ where,
$$K_\tau:=\{x\in \Omega_\tau: \forall n\geq n_\tau, \widehat{\mu}(B(x,n,e^{-nr}))=e^{-nh_{\widehat{\mu}_x}(f)\pm n\tau} \}.$$

Let $\varpi\in \mathcal{W}$ s.t.  $\varrho_\varpi\circ \varpi^{-1}(K_\tau)>0$. Let $x_0$ be a $\varrho_\varpi\circ \varpi^{-1}$-density point of $K_\tau$. Recall that by Corollary \ref{everyGuGibbsIsSmooth}, $\varrho_\varpi\circ\varpi^{-1}=\nu_\varpi$. Recall also that by 
the definition of a G-$u$-Gibbs measure, for every $x\in \mathrm{Im}(\varpi)\cap K_\tau$,
\begin{equation}\label{subseteqVuLoc}
    \mathrm{Im}(\varpi)\subseteq V^u_\mathrm{loc}(x).
\end{equation}

For $\widehat{\mathbf{p}}$-a.e. $\varpi$, let $L_\varpi:=\max \frac{d \nu_\varpi}{d\lambda_{\varpi}}$, so for all $x\in K_\tau\cap \mathrm{Im}(\varpi)$,
\begin{equation}\label{LeavesDense}
    \nu_\varpi(B(x,n,e^{-nr}))\leq L_\varpi\cdot  e^{- n\sum_{i=u}^{u-k+1}m_i \chi_i}e^{3d\tau n},
\end{equation} 
by \eqref{subseteqVuLoc}, and by bounding the $\lambda_{\varpi}$-volume of a cross-section of the Bowen ball (in a Pesin chart).

Let $n\geq n_\tau'$, and cover $B_{\mathrm{Im}(\varpi)}(x_0,e^{-n\tau})\cap K_\tau$ by a cover $\mathcal{C}_n$ of balls of the form $B(\cdot,n,e^{-nr})$, centered at elements of $K_\tau\cap B_{\mathrm{Im}(\varpi)}(x_0,e^{-n\tau})$ and with a multiplicity bounded by $e^{3d\tau n}$. This is possible by \cite[Lemma~2.2]{NLE}. Then,
\begin{align}
    \frac{ \nu_\varpi(B(x_0,e^{-n\tau})\cap K_\tau)}{ L_\varpi\cdot  e^{- n\sum_{i=u}^{u-k+1}m_i \chi_i}e^{3d\tau n}}\leq& \#\mathcal{C}_n
    \leq \frac{e^{3d\tau n}}{\min\limits_{B\in\mathcal{C}_n}\widehat{\mu}(B)}\\
    \leq &e^{nh_{\widehat{\mu}_x}(f)+n\tau}e^{3d\tau n}. \nonumber
\end{align}
Since $x_0$ is a density point, we may assume w.l.o.g. that $\nu_\varpi(B(x_0,e^{-n\tau})\cap K_\tau)\geq \frac{1}{2}\nu_\varpi(B(x_0,e^{-n\tau}))\geq \frac{1}{2}\frac{1}{L_\varpi}e^{-n\tau k}$. Thus in total,
\begin{equation*}
    \frac{1}{2L_\varpi^2}e^{ n\sum_{i=u}^{u-k+1}m_i \chi_i}e^{-3d\tau n}e^{-n\tau k}\leq e^{nh_{\widehat{\mu}_x}(f)+4d\tau n}.
\end{equation*}
We get $h_{\widehat{\mu}_x}(f)\geq \sum_{i=u}^{u-k+1}m_i \chi_i(x)-10d\tau$, for $\widehat{\mu}$-a.e. $x\in \Omega_\tau$. Since $\tau>0$ was arbitrary and $\underline{\chi}$ was arbitrary, we are done.

\end{proof}

\begin{cor}
    $h_{\widehat{\mu}}(f)\geq \int \sum_{i=u(x)}^{u(x)-k+1}\chi_i(x) d\widehat{\mu}(x)$.
\end{cor}
\begin{proof}
    It follows from the affinity of both sides of the equation in Theorem \ref{EntAndExp} w.r.t. the ergodic decomposition. 
\end{proof}

\subsection{Disintegration by embedded disks}

Recall the definition of $\nu_\varpi$ in Definition \ref{DefNuVarpi} for any $\varpi$, and the following remark.

\begin{lemma}
    The measure $\widehat{\mu}$ can be written as $\widehat{\mu}=\int  \nu_\varpi d\widehat{\mathbf{p}}'$, where $\widehat{\mathbf{p}}'$  is a probability on $\mathcal W$ s.t.  $\widehat{\mathbf{p}}'$-a.e. $\varpi:[0,1]^k\to \mathbf{M}$ is a $C^{r-1,1}$ embedding.
\end{lemma}
\begin{proof}
    First, by the definition of $\nu_\varpi$ (for $\widehat{\mathbf{p}}$-a.e. $\varpi$), if $\mathrm{Im}(\varpi')\subseteq \mathrm{Im}(\varpi)$, then
    $$\nu_\varpi|_{\mathrm{Im}(\varpi')}=\nu_{\varpi}(\mathrm{Im}(\varpi'))\cdot\nu_{\varpi'}.$$
This can be seen as both measures are given by the same density function, normalized each over its respective domain.

By Sard's lemma, $\nu_\varpi(\varpi[\Jac(\varpi)=0])=0$. We write the open set $[\Jac(\varpi)\neq 0]\subset [0,1]^k$ as  a countable disjoint (up to zero volume set) union of subcubes $C$ such that $\varpi|_C$ is a diffeomorphism onto its image. Then we let $\mathcal P_\varpi:=\{ \varpi|_C, \ C\}$. 

We consider the probability $\widehat{\mathbf{p}}'$ on $\mathcal W$ defined  as follows   
$$\forall h\in C(\mathcal{W}),\ \widehat{\mathbf{p}}'(h)=\int \sum_{\varpi'\in \mathcal{P}_\varpi
}\nu_\varpi(\mathrm{Im}(\varpi'))\cdot h(\varpi')d\widehat{\mathbf{p}}(\varpi).$$

To take care of measurability of the integrand, the cubes can be chosen in the following measurable way: We start by subdividing $[0,1]^k$ into $k$-cubes $C$ of size $\frac{1}{2}$. We leave the cubes which do not intersect $[\Jac(\varpi)<\frac{1}{2}]$ and whose image under $\varpi$ is a ``good graph", where a ``good graph" is the graph of a $1$-Lipschitz function over $\mathrm{Im}(d_{x_C}\varpi)$ where $x_C$ is the center of $C$. We continue subdividing the rest of the cubes by cubes of halved size, where we requires the cubes in the $N$-th step to not intersect $[\Jac(\varpi)<\frac{1}{2^N}]$. We continue this way by attaining a (possibly infinite) collection of cubes whose size is a power of $\frac{1}{2}$, and whose image  under $\varpi$ is a local graph with a non-vanishing Jacobian. The set $[\Jac(\varpi)=0]$ is closed, and so any point outside of it is covered by our collection, as it admits a small neighborhood where the image under $\varpi$ is a good graph by the inverse function theorem. For every possible cube in the collection which we describe, the properties we require are a closed property in $\varpi$, and so the integrand is measurable.

 We continue to show that $\widehat{\mathbf{p}}'$ satisfies what we need. Indeed, since $\nu_{\varpi}=\sum\limits_{\varpi'\in \mathcal{P}_\varpi}\nu_\varpi(\mathrm{Im}(\varpi'))\cdot \nu_{\varpi'}$, we 
get $\widehat{\mu}=\int \sum\limits_{\varpi'\in \mathcal{P}_\varpi
}\nu_\varpi(\mathrm{Im}(\varpi'))\cdot \nu_{\varpi'}d\widehat{\mathbf{p}}(\varpi)=\int \nu_{\varpi'}d\widehat{\mathbf{p}}'(\varpi')$.
\end{proof}

\section{The Strong Viana Conjecture}\label{SVCSect}

\subsection{Hyperbolic SRB measures for co-dimension $1$}
As a first step towards the proof of Theorem \ref{codim}, we prove the existence of a hyperbolic SRB measure in the co-dimension $1$ case.

\begin{claim}\label{excodim}
If $k=d-1$, then $\widehat{\mu}$ from \textsection \ref{constMu} is a hyperbolic SRB measure.
\end{claim}
\begin{proof}

By Theorem \ref{EntAndExp} and Corollary \ref{SRB1}, almost every ergodic component of $\widehat \mu$ has positive entropy, and hence by the Ruelle inequality it admits a negative Lyapunov exponent. Therefore, there are at most $(d-1)$-many positive Lyapunov exponents at almost every point. By Corollary \ref{SRB1}, almost every point admits at least $(d-1)$-many positive Lyapunov exponents. Thus in total, almost every point admits exactly $(d-1)$-many positive Lyapunov exponents, and one  negative exponent. By Theorem \ref{EntAndExp}, a.e. ergodic component of $\widehat{\mu}$ satisfies the entropy formula, thus it is a hyperbolic SRB measure with exactly $(d-1)$-many positive Lyapunov exponents.

\end{proof}

\subsection{The strong Viana conjecture}

In this last section we prove Theorem \ref{codim}
and Theorem \ref{strongv}.

\subsubsection{New negative exponents}

\begin{definition}\label{newNegExp}For $k\in\{1,\ldots,d-1\}$ and $x\in \mathbf M$, we 
set $$\varkappa_{k+1}^-(x):=\lim_{\Delta\to 0}\limsup_{q\to\infty}\liminf_{n\to\infty} \varkappa_{k+1,q,n,\Delta}^-(x),$$
where
	$$\varkappa_{k+1,q,n,\Delta}^-(x):=\min_{\overset{\mathcal{L}\subseteq \{1,\ldots,n\}}{\Delta n\leq \#\mathcal{L}\leq n}}\frac{1}{\#\mathcal{L}}\sum_{\ell\in\mathcal{L}}\Phi_{k+1}^{(q)}\circ f^{\ell}(x),$$
	and
	$$\Phi_{k+1}^{(q)}(x):=\frac{1}{q}\log^+\frac{\|\wedge^{k} d_xf^q\|}{\|\wedge^{k+1} d_xf^q\|}.$$
\end{definition}
\medskip
\noindent\textbf{Remark:}\text{ }
 The exponent $\varkappa^-_{k+1}(x)$ does not depend on the choice of the norm on $\wedge^lT\mathbf M$, $l=1,\cdots, d$. For an appropriate norm, the ratio $\frac{\|\wedge^k d_xf^q\|}{\|\wedge^{k+1} d_xf^q\|}$ is just the inverse of the $(k+1)$-th singular value of $d_xf^q$.

\begin{definition}\label{betaKPlusOne}
For $x\in \mathbf M$ we denote by $p\omega(x)$ the limit set of the {\em empirical measures} $\{\nu_x^{(n)}\}_{n\geq0}$ where $\nu_x^{(n)}:=\frac{1}{n}\sum_{0\leq \ell<n}\delta_{f^\ell(x)}$. 
 Given $k\in\{1,\ldots,d-1\}$, set 
$$\beta_{k+1}(x):=\sup_{\nu\in p\omega(x)}\mathrm{ess \, sup}_\nu\  \chi_{k+1}.$$

\end{definition}

\begin{prop}\label{limsupToLiminf}
    $$\varkappa_{k+1}^-\leq \max\{-\beta_{k+1}, 0\}
    .$$
\end{prop}
\begin{proof}
Let $\delta>0$. Let $\Delta>0$ and $q_0$ s.t.  for infinitely many $q_i\geq q_0$, $\varkappa^-_{k+1}(x)\leq \liminf\limits_n \varkappa^-_{k+1,n,q_i,\Delta}(x)+\frac{\delta}{2}$. For all $q_i\geq q_0$, for all $n\geq n_i$,
let $\mathcal{L}_{n,q_i}$ which attains the minimum s.t.  $\#\mathcal{L}_{n,q_i}\geq \Delta\cdot n$ and $\varkappa^-_{k+1}(x)\leq \frac{1}{\#\mathcal{L}_{n,q_i}}\sum_{\ell\in\mathcal{L}_{n,q}}\Phi^{(q_i)}_{k+1}\circ f^\ell(x)+\delta$.
    
Then for any subsequence $n_t\uparrow \infty$ s.t.  $\frac{1}{n_t}\sum_{\ell\leq n_t}\delta_{f^\ell(x)}\to \nu$ (that is for any $\nu\in p\omega(x)$), by a diagonal argument there exists a sub-subsequece $\{n_j\}_{j\geq0}\subseteq \{n_t\}_{t\geq 0}$ s.t.  $\frac{1}{\#\mathcal{L}_{n_j,q_i}}\sum_{\ell\in \mathcal{L}_{n_j,q_i}}\delta_{f^\ell(x)}\to \nu_{q_i}$ for all $q_i\geq q_0$. 
    
 We can fix $q_{\delta,\nu}\geq q_0$ s.t.  for all $q\geq q_{\delta,\nu}$, \begin{equation}\label{fd}\nu([ \Phi^{(q)}_{k+1}=\max\{-\chi_{k+1},0\}\pm \sqrt\delta])> 1-\Delta\cdot \delta.\end{equation}
 Assume w.l.o.g. that $q_{\delta,\nu}\in \{q_i, \ i\in \mathbb N\}$. Then, indeed, $\nu_{q_{\delta,\nu}}\cdot \Delta\leq \nu$, and $\int \Phi^{(q_{\delta,\mu})}_{k+1}d\nu_{q_{\delta,\nu}}\geq \varkappa^-_{k+1}(x)-\delta$. 

Then $\nu_{q_{\delta,\nu}}([\Phi^{(q_{\delta,\nu})}_{k+1}\geq \varkappa^-_{k+1}(x)-\sqrt\delta])\geq \delta$; otherwise for all $\delta>0$ small enough so $\delta\cdot\log M_f-\sqrt\delta<-\delta$,
$$\int\Phi^{(q_{\delta,\nu})}_{k+1}d\nu_{q_{\delta,\nu}} \leq \delta \log M_f+\varkappa^-_{k+1}(x)-\sqrt\delta<\varkappa^-_{k+1}(x)-\delta\text{, a contradiction}.$$

Therefore,
$$\nu([\Phi^{(q_{\delta,\nu})}_{k+1}\geq \varkappa^-_{k+1}(x)-
\sqrt\delta ])\geq \Delta\delta,$$
and by \eqref{fd}, we get
$$\nu([\max\{-\chi_{k+1},0\}\geq \varkappa^-_{k+1}(x)-
2\sqrt\delta])>0.$$

Therefore, $\varkappa^-_{k+1}(x)\leq \inf_{\nu \in p\omega(x)} \mathrm{ess \, inf}_\nu \max\{-\chi_{k+1},0\}+2\sqrt\delta$ for all $\delta>0$ sufficiently small, so that 
\begin{align*}
    \varkappa^-_{k+1}(x)\leq& \inf_{\nu \in p\omega(x)} \mathrm{ess \, inf}_\nu \max\{-\chi_{k+1},0\}\\
    =&\max\{-\sup_{\nu\in p\omega (x)} \mathrm{ess \, sup}_\nu \chi_{k+1},0\}=\max\{-\beta_{k+1}(x),0\}.
\end{align*}
\end{proof}

\begin{lemma}
When $x$ is not a source we have 
$$\beta_d(x)\leq 0.$$
\end{lemma}
\begin{proof}
    If $\beta_d(x)>0$ there is $\nu\in p\omega(x)$ with an ergodic component $\xi$ satisfying $\chi_d(\xi)>0$. Therefore $\xi$ is the atomic measure at a source. Then we have necessarily $\mathrm{p\omega}(x)=\{\nu\}=\{\xi\}$ and $x$ is a  source.
\end{proof}

Recall that $\chi_{k+1}(x)$ denotes the $(k+1)$-th Lyapunov exponent at $x$ (with multiplicity).

\begin{lemma}\label{negExpCoincide}
	For an $f$-invariant  probability $\nu$, for $\nu$-a.e. $x$, $\varkappa^-_{k+1}(x)=\max\{-\chi_{k+1}(x),0\}$ and $\chi_{k+1}(x)=\beta_{k+1}(x)$.
\end{lemma}
\begin{proof}
	$\nu$-a.e. $x$ is typical w.r.t. an ergodic component of $\nu$, hence we may assume w.l.o.g. that $\nu$ is ergodic. Let $\Delta\in \left(0,\frac{1}{4(\max\{\log M_f,1\})^2}\right)$, and let $0<\tau\ll\Delta^2$. 
Given $x$, 
let $q_{x,\tau}\in\mathbb{N}$ s.t.   for all $q\geq q_{x,\tau}$, $\Phi_{k+1}^{(q)}(x)=\max\{0,-\chi_{k+1}(x)\}\pm\tau$.

Let $q_\tau$ sufficiently large so $\nu([q_{x,\tau}\leq q_\tau])\geq 1-\frac{\tau}{2}$. 
Let $n_\tau
$ s.t.  $\nu
(\Omega_\tau
)\geq 1-\tau$, where 
$$\Omega_\tau
:=\left\{x:\forall n\geq n_\tau
, \frac{1}{n}\sum_{\ell\leq n}\mathbb{1}_{[q_{x,\tau}\leq q_\tau]}\circ f^{ \ell}(x)\geq 1-\tau\right\}.$$

Then for any $x\in \Omega_\tau
$, for all $n\geq n_\tau$, for all $q\geq q_\tau$
$$\varkappa_{k+1,q,n,\Delta}^-(x)=(\max\{0,-\chi_{k+1}(x)\}\pm \tau) \frac{\Delta-\tau}{\Delta}\pm \frac{\tau}{\Delta} \cdot \log M_f.$$

Therefore $\lim_n \varkappa^-_{k+1,q,n,\Delta}(x)=\max\{0,-\chi_{k+1}(x)\}\pm \sqrt{\Delta}$, for all $p\geq p_\tau$ for all $x\in \Omega_\tau$. Then we are done as the measure of $\Omega_\tau$ can be made arbitrarily close to $1$, and we may send $q$ to $\infty$ and $\Delta$ to $0$.

The assertion concerning $\beta_k$ follows easily from the definition as $\mathrm{pw}(x)=\{\nu\}$ for a typical point $x$ for an ergodic measure $\nu$.
\end{proof}

\begin{definition}
	For $\chi\geq 0$ and $k\in \{1,\ldots,d-1\}$, let 
	\begin{align*}
		\mathrm{Hyp}_\chi^k:=\Big\{x\in\mathbf{M}:
        \lambda_{k
        }(x),\varkappa_{k+1
        }^-(x)
        >\chi
		\Big\}.
	\end{align*}
\end{definition}

\medskip
\noindent\textbf{Remark:} Notice, in particular there is no condition on the angles between expanding or contracting directions for points in $\mathrm{Hyp}_\chi^k$.

\subsubsection{Invariant measure with non-zero Lyapunov exponents}\label{forTheProofSection}

Recall the construction of the sets $E_n\subseteq E$ from Definition \ref{En}. We will now consider them as a sequence of subsets of $\mathrm{Hyp}_\chi^k$, where we assume $\Vol_k(\mathrm{Hyp}_\chi^k)>0$.

\begin{theorem}[Strong Viana Conjecture]\label{SVC} 
If $\Vol(\mathrm{Hyp}_\chi^k)>0$, then there exists a $\chi$-hyperbolic SRB measure with exactly $k$ positive Lyapunov exponent almost everywhere.
\end{theorem}
\begin{proof} We provide a proof for a stronger statement, where we only use the fact that for $x\in \mathrm{Hyp}_\chi^k$, we have  $\lambda_k(x)>\chi$ and $\beta_{k+1}(x)<-\chi$ (recall Proposition \ref{limsupToLiminf}).

\medskip
For all $N\in \mathbb N^*$ we let $\mathcal{C}_N\subseteq B_{\|\cdot\|_\mathrm{Lip}}(0,N)$ s.t.  $\bigcup_{h\in \mathcal{C}_N}B_{\|\cdot\|_{\mathrm{Lip}}}(h,\frac{1}{N})$ is a finite cover of $B_{\|\cdot\|_\mathrm{Lip}}(0,N)\subseteq \mathrm{Lip}(\mathbf{M})$. Then $$d(\nu,\eta):=\sum_{N\geq 1 }\sum_{h\in \mathcal{C}_N}\frac{1}{2^N\#\mathcal{C}_N}|\nu(h)-\eta(h)|$$ defines a convex metric  on $\mathbb{P}(\mathbf M)$ compatible with the weak-$*$ topology, i.e.  $d(a\eta+(1-a)\nu, a\eta'+(1-a)\nu')\leq a d(\eta,\eta')+(1-a)d(\nu, \nu')$ for any $a\in [0,1]$ and any $\nu,\nu'\in \mathbb{P}(\mathbf M)$. 

\medskip
Let $\mathcal K(\mathbb P(\mathbf M))$ be the set of compact subsets of $\mathbb P(\mathbf M)$ endowed with the Hausdorff-Gromov distance. 

\medskip
Let $E:=\{x\in \mathbf M : \beta_{k+1}(x)<-\chi \}$.  Since the map $x\mapsto p\omega(x)$ is Borel (see for example \cite[Appendix~B]{Burguetphysical}), there is a compact subset $E'\subseteq E$ with a positive measure such that $x\mapsto p\omega(x)$ is continuous on $E'$. Then $p\omega(E'):=\bigcup_{x\in E' }p\omega(x)$ is compact. By Egorov's theorem, there exists a compact subset of positive measure $E''\subseteq E'$, s.t.  $d(\nu_x^{(n)}, p\omega(E'))\xrightarrow[n\to\infty]{\text{uniformly}}0$ on $E''$. We carry out the construction of $\widehat\mu$ and $\mu$ as in \textsection \ref{constMu} by the set $E''$, choosing $E_{n_j}\subseteq E''$.  We prove now that $\chi_{k+1}<-\chi$ $\widehat{\mu}$-a.e.

\medskip
It is enough to show that any weak limit of $\widetilde{\mu}_n:=\int \nu_x^{(n\cdot p)}\, d\lambda_n$ has its  $k+1$-exponent negative a.e., since $\forall M$, $\mu_n^M\leq p\cdot \widetilde{\mu}_n$. By taking a subsequence of $\{n_j\}_{j\geq 0}$ we may assume that $\widetilde\mu_{n_j}$ is converging to a probability $\widetilde{\mu}$. Hence, $\mu\leq p\cdot \widetilde{\mu}$. 
\medskip

Recall the measure $\lambda_n$ is supported on $E_n^n$. 
By Lemma \ref{sameEmpLims}, for all $x\in E_n$ and $y\in P_n^n(x)$, we have  $d(\nu_y^{(n)},p\omega(x))\leq d(\nu_x^{(n)},p\omega(x))+\frac{1}{\sqrt n}\sum_{N\geq 1}\frac{N}{2^N}$. Therefore
$\sup_{y\in E_n^n}d(\nu_y^{(n)}, p\omega(E'))\xrightarrow{n\to\infty}0$.
We ``discretize" the measure $\widetilde{\mu}_n$ : let $\alpha_n$ be a partition of $E_n^n$ and fix $x_n^A\in A\subset E_n^n$, for all $A\in \alpha_n$, such that $\widetilde{\mu}_n':=\sum_{A\in \alpha_n}\lambda_n(A)\cdot \nu_{x_n^A}^{(n\cdot p)}$ satisfies  $d(\widetilde{\mu}_n,\widetilde\mu_n')\xrightarrow{n\to \infty}0$.

\medskip
Let $\xi_n^A\in p\omega(E')$, such that $d(\nu_{x_n^A}^{(n\cdot  p)},\xi_n^A)=d(\nu_{x_n^A}^{(n\cdot p)},p\omega(E'))$. Write $\xi_n:=\sum_{A\in \alpha_n}\lambda_n(A)\cdot \xi_n^A$. Then we have,
\begin{align*}
d(\widetilde{\mu}_n', \xi_n)&\leq \sum_{A\in \alpha_n}\lambda_n(A)\cdot d(\nu_{x_n^A}^{(n\cdot p)},\xi_n^A)\leq \sup_{x\in E''}d( \nu_x^{(n\cdot p)},p\omega(E'))\xrightarrow{n\to \infty}0.
\end{align*}

\medskip
Let $\vartheta\in \mathbb P(p\omega(E')$  be a limit point of $\{\sum_{A\in \alpha_n}\lambda_n(A)\cdot \delta_{\xi_n^A}\}_{n>0}$ (recall that $\mathbb{P}(p\omega(E'))$ is compact). Then $\widetilde\mu=\int \xi\,  d\vartheta(\xi)$, and so $\chi_{k+1}<-\chi$ $\widetilde{\mu}$-a.e., as $\vartheta$-a.e. $\xi$ belongs to $p\omega(E')$, where by the definition of $\beta_{k+1}$, we have   $\xi$-a.e. $\chi_{k+1}<-\chi$.

\medskip
Therefore, it follows that $\chi_{k+1}<-\chi$ $\mu$-a.e., and so $\widehat{\mu}$-a.e., as the Lyapunov exponents are $f$-invariant functions. Since almost every ergodic component of $\widehat{\mu}$ is also a G-$u$-Gibbs measure with a disintegration on disks of dimension $k$, and it admits $k$ exponents larger than $\chi$ almost everywhere, it follows that almost every such an ergodic component is therefore a $\chi$-hyperbolic SRB measure with exactly $k$-positive Lyapunov exponents. \end{proof}

\medskip

\noindent\textbf{Remark:} The same proof applies to show that (note that $\beta_{k+1}$ may be non-positive),
\begin{enumerate}
\item If the set $[\lambda_k>\chi\geq\alpha\geq  \beta_{k+1}]$ has positive volume for some $\chi>0$ and for some $\alpha\in\mathbb{R}$ ($\alpha$ may be non-negative), there is a $u$-Gibbs measure $\mu$ with exactly $k$ Lyapunov exponents greater 
 than $\chi$, i.e. $\mu$ has absolutely continuous conditionals along strong unstable manifolds of dimension $k$, and whose $(k+1)$-th Lyapunov exponent (with multiplicity) is less or equal to $\alpha$. 
\item If the set $[\lambda_k>\chi>0\geq \beta_{k+1} ]$ has a positive volume,  there is an SRB measure (a-priori perhaps not hyperbolic) with exactly $k$ positive Lyapunov exponents, and they are greater than $\chi$. In particular this provides an alternative proof of Claim \ref{excodim} which does not involve the entropic characterization of SRB measures.
\end{enumerate}

We conclude now the proofs of Theorem \ref{strongv} and Theorem \ref{codim} by establishing the basin property. 

\begin{cor}\label{fin}
$\mathrm{Vol}$-a.e. $x\in\mathrm{Hyp}_\chi^k$ (resp. $\mathrm{Vol}$-a.e. $x$ with $\lambda_{d-1}(x)>0$) lies in the basin of attraction of an ergodic $\chi$-hyperbolic SRB measure with exactly $k$ positive Lyapunov exponent almost everywhere.
\end{cor}
\begin{proof}
Assume for contradiction that there is a positive volume subset $G$ of $\mathrm{Hyp}_\chi^k$ s.t. no point in it belongs to the basin of attraction of a $\chi$-hyperbolic SRB measure with exactly $k$ positive Lyapunov exponents. In particular, there is a $k$-disk $D$, and a subset $F$ of $G\cap D$ of positive disk volume s.t.   $\chi_{A_k,p_0}(x,\iota(T_xD))>\chi$ for all $x\in F$, for some $p_0\in \mathbb{N}$ (recall \textsection \ref{theDiskAndSubset}). 
By Claim \ref{Leb}, there is a subset $E$ of $F$ with positive disk volume  such  that 
\begin{align}\label{fiif}
\frac{\lambda(P^{\ell}(x)\cap F)}{\lambda(P^{\ell}(x))}\xrightarrow{\ell \to \infty}1,
\end{align} 
 uniformly  in  $x\in E$. Carry out the construction of a hyperbolic SRB measure $\mu$ with with exactly $k$ positive Lyapunov exponents  by this subset $E$ (recall Theorem \ref{SVC}).  Recall that $\mu\geq \mu^0=\int \varrho_\varpi\circ \varpi^{-1}d p^0(\varpi)$ with $\varrho_\varpi\circ \varpi^{-1}=\nu_\varpi$ being equivalent to the induced Riemannian volume on $\varpi$. 
 
 Let $\Lambda$ be a Pesin block with $\mu^0(\Lambda)>\frac{\beta}{2}$. For $p^0$ a.e. $\varpi$, for all $n$ large enough, there exist $x_n\in E$ and $\mathrm{BG}_n(x_n)\ni \ell_n\xrightarrow{n\to \infty}\infty$ such that $\varpi_{\ell_n}(x_n)$ goes to $\varpi$ in the $C^{r-1,1}$ topology. Fix such a $\varpi$ with $\nu_\varpi(\Lambda)>\frac{\beta}{2}$. 
 
 Observe that $\nu_\varpi$-a.e. point lies in the basin of an ergodic component $\nu$ of $\mu$
which is also a $\chi$-hyperbolic SRB measure with exactly $k$ positive exponents. This is since ergodic components are saturated by unstable leaves, by the pointwise ergodic theorem.  

 Set $\Lambda_\nu:=\Lambda\cap \mathcal B(\nu)$, and so, in particular, $\nu_\varpi(\Lambda_\nu)= \nu_\varpi(\Lambda)>\frac{\beta}{2}$. 
For a subset $K$ of $\mathbf M$ we denote its stable set  by $V^s(K):=\Big\{x\in \mathbf M: \ \exists y\in K \text{ s.t. }
 \limsup_{m}\frac{1}{m}\log d(f^{m}(x),f^{m}(y))<0\Big\}$. For all $\ell \in \mathbb N$ we have $f^{-\ell}[V^s(\Lambda_\nu)]\subseteq f^{-\ell}[V^s(\mathcal B(\nu))]\subseteq \mathcal B(\nu)$. 

The stable foliation is transversal to $\varpi$ as this disk is contained in an unstable leaf by Corollary \ref{SRB1}; and $\varpi_{\ell_n}(x_n)\xrightarrow[]{C^{r-1,1}}\varpi$. Then, by the absolute continuity of the stable foliation, there exists $c>0$ such that,  $$\liminf_n\lambda_{\varpi_{\ell_n}(x_n)}\left(\varpi_{\ell_n}(x_n)\cap V^s(\Lambda_\nu)\right)\geq c.$$

Thus, as $\ell_n$ lies in $\mathrm{BG}_n(x_n)$ (i.e. it admits a uniform lower bound on its induced volume), we have for some $c'>0$,
\begin{align*}
\liminf_n\frac{\lambda_{\varpi_{\ell_n}(x_n)}\left(\varpi_{\ell_n}(x_n)\cap V^s(\Lambda_\nu)\right)}{\lambda_{\varpi_{\ell_n}(x_n)}(\varpi_{\ell_n}(x_n))}\geq c'.
\end{align*}
By the bounded distortion property (\ref{bdist}), we have 
\begin{align}\label{eiie}\liminf_n\frac{\lambda\left(P^{\ell_n}(x_n)\cap \mathcal B(\nu)\right)}{\lambda(P^{\ell_n}(x_n))}&\geq\liminf_n\frac{\lambda\left(P^{\ell_n}(x_n)\cap f^{-\ell_n}[V^s(\Lambda_\nu)]\right)}{\lambda(P^{\ell_n}(x_n))}\nonumber\\
&\geq \frac{1}{4}\liminf_n\frac{\lambda_{\varpi_{\ell_n}(x_n)}\left(\varpi_{\ell_n}(x_n)\cap V^s(\Lambda_\nu)\right)}{\lambda_{\varpi_{\ell_n}(x_n)}(\varpi_{\ell_n}(x_n))} \nonumber\\
&\geq \frac{c'}{4}.
\end{align}

 By combining \eqref{eiie} and \eqref{fiif} we get   $\varnothing \neq \mathcal B(\nu)\cap F\subset \mathcal B(\nu)\cap G$. This contradicts our assumption on $G$.

\end{proof}

\medskip
Finally, we sketch the proof of Corollary \ref{viana}. 
\begin{proof}[Proof of Corollary \ref{viana}]
Recall, we consider a partially hyperbolic attractor $\Lambda$ of the form 
$T_\Lambda M=E^u\oplus E^c\oplus E^s$ with $\limsup_n\frac{1}{n}\log \mathrm{Jac}(d_xf|_{E^c(x)})\leq 0$ for all $x\in \Lambda$. Let $k=\mathrm{dim}(E^u)$. By Remark 
(2)  following Lemma \ref{limOnP}, we have $\lambda_{k+1}(x)=\limsup_{n}\frac{1}
{n}\log \|d_xf^n|_{(E^c\oplus E^s)(x)}\|=\limsup_{n}\frac{1}{n}\log 
\|d_xf^n|_{E^c(x)}\|$. Assume that $\limsup_{n}\frac{1}{n}\log \|d_xf^n|_{E^c(x)}\|$ is 
larger than $\chi>0$ on  a set of positive volume. Following the construction in 
\textsection \ref{constMu} we get a generalized $u$-Gibbs measure  over $(k+1)$-disks $\mu$ supported on the 
attractor $\Lambda$. In particular $\mu$ has at least one 
central   exponent larger than $\chi$ almost everywhere. By the assumption on the Jacobian of the restriction of the dynamics to the central bundle (i.e. dissipative),  $\mu$ also admits
a central exponent which is less than $-\chi$ almost everywhere. Therefore,  $\mu$  is a $\chi$-hyperbolic 
SRB measure. Finally, the property of the basin is obtained in the same way as in the proof of Corollary \ref{fin}.

\end{proof}

\begin{appendix}

\section{Ergodic components of generalized $u$-Gibbs measures}

\begin{claim}\label{ergGuGibbs}
    Almost every ergodic component of a generalized $u$-Gibbs measure over $k$-disks, is a generalized $u$-Gibbs measure over $k$-disks.
\end{claim}
\begin{proof}
Let $P_x$ be the ergodic component of $x$ w.r.t. $f^{-1}$ :
\begin{align*}
    P_x:=\Big\{&y\in \mathbf M: \ \forall \phi\in C(\mathbf M), \\&\limsup_{n\to\infty}\Big|\frac{1}{n}\sum_{0\leq k<n}\phi(f^{-k}(x))-\phi(f^{-k}(y))\Big|=0\Big\}.
\end{align*}
It is well-known that the associated partition $P=\{P_x\}_{x\in\mathbf{M}}$ is a measurable partition of the measured space $(\mathbf M, \widehat \mu)$ and that the associated conditional measures $\{\widehat{\mu}_{P_x}\}_{x\in\mathbf{M}}$ are the ergodic components of $\widehat \mu$. For $\widehat{\mathbf p}$-a.e. $\varpi\in \mathcal W$ we have $\mathrm{Im}(\varpi)\subseteq P_{\varpi(0)}$, $0\in [0,1]^k$ (where $\mu_\varpi$-a.e. $x$ admits $V^u(x)\supseteq \mathrm{Im}(\varpi)$). Therefore $Q=\Phi^{-1}P$ is a measurable partition of $(\mathcal W, \widehat{\mathbf p})$ with  $\Phi:\mathcal W\to \mathbf M$, $\varpi\mapsto \varpi(0)$. Then, by denoting  the associated conditional measures by  $\widehat{\mathbf p}_{Q_{\varpi}}$ we have for $\widehat{\mathbf{p}}$-a.e. $\varpi'$ 
$$\widehat \mu_{\varpi'(0)}=\int \mu_\varpi d\widehat{\mathbf p}_{Q_{\varpi'}}(\varpi). $$
Thus any ergodic component of the generalized $u$-Gibbs measure $\widehat{\mathbf \mu}$  is itself a generalized $u$-Gibbs measure.
\end{proof}

\section{Angle estimates}\label{anglee}

\subsection{Exterior Powers}

We endow $\mathbb{R}^{d}$ with the Euclidean norm $\|\cdot \|$ and denote by $\langle\cdot,\cdot \rangle$ the associated inner product. For each $k=1, \ldots, d-1$, the induced Euclidean norm on the $k$-th exterior power $\wedge^{k}\left(\mathbb{R}^{d}\right)$ is defined for indecomposable vectors $Z$ of $\wedge^{k}\left(\mathbb{R}^{d}\right)$, i.e., $Z=z_{1} \wedge \ldots \wedge z_{k}$ with $z_{i} \in \mathbb{R}^{d}$, by the determinant of the Gram matrix of $z_{1}, \ldots, z_{k}$:

\[
\|Z\|^{2}=\operatorname{det}\left(\langle z_{i}, z_{j} \rangle\right)_{i, j},
\]

and then extended by bilinearity. The associated inner product will also be denoted by $\langle, \rangle$ and the corresponding operator norm on $\wedge^{k} \mathcal{L}\left(\mathbb{R}^{d}, \mathbb{R}^{d}\right)$ and $\mathcal{L}\left(\mathbb{R}^{d}, \wedge^{k} \mathcal{L}\left(\mathbb{R}^{d}, \mathbb{R}^{d}\right)\right)$ by $\|\cdot \|$.

\subsection{Grassmanian}\label{grass}

The Grassmanian $\operatorname{Grass}(k, d)$ is the set of $k$-planes in $\mathbb{R}^{d}$. When $Z=z_{1} \wedge \ldots \wedge z_{k}$ is an indecomposable vector of $\wedge^{k}\left(\mathbb{R}^{d}\right)$, we denote by $g(Z)$ the $k$-plane spanned by $z_{1}, \ldots, z_{k}$. Notice that a $k$-plane is defined, up to a scalar, by a unique indecomposable vector of $\wedge^{k}\left(\mathbb{R}^{d}\right)$. This is the so-called Plücker embedding of the Grassmanian $\operatorname{Grass}(k, d)$ in the projective space $\mathbb P \wedge^{k}\left(\mathbb{R}^{d}\right)$ which we denote by $\iota: \operatorname{Grass}(k, d) \rightarrow \mathbb P \wedge^{k}\left(\mathbb{R}^{d}\right)$. This embedding with the Euclidean norm on $\wedge^{k}\left(\mathbb{R}^{d}\right)$ induced a distance on $\operatorname{Grass}(k, d)$, called the Cayley distance, as follows:

\[
\cos d_{c}\left(\iota(Z), \iota\left(Z^{\prime}\right)\right)=\frac{\langle Z, Z^{\prime} \rangle}{\|Z\| \|Z^{\prime}\|}.
\]

\subsubsection{Stationary Angles Between Planes}

Given two $k$-planes $F$ and $G$, we define the angle $\angle F, G$ between $F$ and $G$ as follows:

\[
\angle F, G=\max _{u \in F} \min _{v \in G} \angle u, v.
\]

This angle is related to the orthogonal projectors onto $F$ and $G$:

\begin{lemma}
\[
\sin \angle E, F=\left\|p_{E}-p_{F}\right\|.
\]
with $p_{E}$ and $p_{F}$ the orthogonal projector on $E$ and $F$, and $\| \|$ the operator norm induced by the Euclidean norm.
\end{lemma}

In fact, one can define a sequence of $k$-angles $\left(\theta_{k}\right)_{k=1, \ldots, d}$ between $F$ and $G$, known as the stationary angles. They are built by induction as follows:

\[
\theta_{1}=\min _{u \in F} \min _{v \in G} \angle u, v=\angle u_{1}, v_{1}.
\]

and for all $i=1, \ldots, d-1$:

\[
\theta_{i+1}=\min _{u \in F, u \perp \operatorname{Vect}\left(u_{1}, \ldots, u_{i}\right)} \min _{v \in G, v \perp \operatorname{Vect}\left(v_{1}, \ldots, v_{i}\right)} \angle u, v=\angle u_{i}, v_{i}.
\]

This sequence does not depend on the choice of the vectors $u_{i} \in F$ and $v_{i} \in G$ which realize the minima in the previous constructions. Observe that $\theta_{1} \leq \ldots \leq \theta_{d}$. The last stationary angle $\theta_{d}$ coincides with the angle $\angle F, G$. The Cayley distance is related to the stationary angles and the angle $\angle$ as follows:

\[
\cos d_{c}(F, G)=\prod_{i=1}^{d} \cos \left(\theta_{i}\right) \leq \cos \angle F, G.
\]

\subsection{Estimates}
For a linear map $A:\mathbb R^k\to \mathbb R^d$ we let $V_A=A(\mathbb R^k)$ and by abuse of notations we write  $\wedge^k(A)$ for $\wedge^k(A)(e_1\wedge\cdots \wedge e_k)\in \wedge^k(\mathbb R^d)$. 
\begin{lemma}\label{angl}
Let $A, B$ be linear map from $\mathbb R^k$ to  $\mathbb R^d$. We assume that \\ $\left\|\wedge^{k}(A)-\wedge^{k}(B)\right\| \leq K\left\|\wedge^{k}(A)\right\|$ with $K \leq \frac{1}{2}$. Then,

\[
d_{c}\left(V_A,V_B\right) \leq \arccos\left(1-\frac{K^{2}}{2(1-K)}\right).
\]
In particular for $K=1/10$ we get $\angle V_A,V_B \leq d_{c}\left(V_A,V_B\right) < \pi/6$.
\end{lemma}

\begin{proof}
We first remark that

\[
\begin{aligned}
\left\|\wedge^{k}(B)\right\| & \geq\left\|\wedge^{k}(A)\right\|-\left\|\wedge^{k}(A)-\wedge^{k}(B)\right\| \\
& \geq(1-K)\left\|\wedge^{k}(A)\right\|.
\end{aligned}
\]

It follows that

\[
\begin{aligned}
\left\|\wedge^{k}(A)-\wedge^{k}(B)\right\|
& \leq \frac{K}{\sqrt{1-K}} \sqrt{\left\|\wedge^{k}(A)\right\|\left\|\wedge^{k}(B)\right\|}.
\end{aligned}
\]

But the square of the left member can be rewritten as

\[
\left\|\wedge^{k}(A)-\wedge^{k}(B)\right\|^{2}=\left\|\wedge^{k}(A)\right\|^{2}+\left\|\wedge^{k}(B)\right\|^{2}-2\langle\wedge^{k}(B), \wedge^{k}(A)\rangle.
\]

Therefore we conclude that

\[
\begin{aligned}
\frac{\langle\wedge^{k}(A), \wedge^{k}(B)\rangle}{\left\|\wedge^{k}(A)\right\|\left\|\wedge^{k}(B)\right\|} & =\frac{\left\|\wedge^{k}(A)\right\|^{2}+\left\|\wedge^{k}(B)\right\|^{2}}{2\left\|\wedge^{k}(A)\right\|\left\|\wedge^{k}(B)\right\|}-\frac{\left\|\wedge^{k}(B)-\wedge^{k}(A)\right\|^{2}}{2\left\|\wedge^{k}(A)\right\|\left\|\wedge^{k}(B)\right\|} \\
& \geq 1-\frac{K^{2}}{2(1-K)}.
\end{aligned}
\]
\end{proof}

\begin{lemma}\label{graph}
Let $D$ be a $C^1$ embedded $k$-disk of $\mathbb R^d$ an $E$ be a vector subspace of $\mathbb R^d$ of dimension $k$ with $\angle T_xD,E<\frac{\pi}{6}$ for all $x\in D$. 

Then, for any $x$ there is $\epsilon>0$ such that  $D\cap B(x,\epsilon)=\Gamma_\psi$ where   $\psi: E_x\subset E\to E^\bot$ is a $C^1$ function with $\|d_\cdot \psi \|\leq 1$.
\end{lemma}
\begin{proof}It is enough to see that any vector $v\in T_xD$ writes as $v=v_E+v_{E^\bot}$ with $v_E\in E$, $v_{E^\bot}\in E^\bot$ where  $\|v_E\|=\|x\|(1\pm \|p_{T_xD}-p_E\| )$  and $\|v_{E^\bot}\|\leq \|x\|\cdot\|p_{T_xD}-p_E\| $. 
\end{proof}

\section{Covering number of $C^r$ smooth disks}
For a subset $E$ of $\mathbb R^d$ we let $\mathrm{Cov}(E, R)$ be the minimal  number of  euclidean balls of radius $R$  covering $E$. 
\begin{lemma}\label{cover}
Let $\sigma:[0,1]^k\to \mathbb R^d$ be a $C^r$ map witk $k\leq d$. Then  there is a constant $C_{r,d}$ depending on $r$ and $d$ such that 
$$\mathrm{Cov}\left(\mathrm{Im}(\sigma),1\right)\leq C_{r,d}\max\left(1,\|d^r\sigma\|\right)^{k/r} \max_{0\leq l\leq k}\|\wedge^{l} d\sigma \|.  $$
\end{lemma}

\begin{proof}
By Lemma 4 in \cite{BurguetSard} (or Chapters 5 and 7 of  \cite{Yom1}) it holds for any polynomial map $\sigma$ of degree less than $r$. We deduce the general case by polynomial interpolation. We first subdivide $[0,1]^k$ into subcubes $\mathbf C$ of size $\max\{1,\|d^r\sigma\|\}^{-1/r}$. Let $\sigma_\mathbf C=\sigma\circ \theta_\mathbf C $ with $\theta_\mathbf C:[0,1]^l\to \mathbf C$ being an affine  reparametrization of $\mathbf C$. Then $\|d^r\sigma_\mathbf C\|\leq 1$. Consider the Lagrange interpolation polynomial $P_\mathbf C$ of $\sigma_\mathbf C$ at the center of $\mathbf C$. 
We have $\|P_\mathbf C-\sigma_\mathbf C\|_{r}\leq 1$. Moreover
by Lemma 6.2 of \cite{Yom1}  we have  for all $x\in [0,1]^n$, for all $l=1,...,k$  and for some constant $C_d$ depending only on $d$ with the convention $\|\wedge^0\sigma_\mathbf C\|=1$  
\begin{align*}\label{pedro}
\|\wedge^lP_{\mathbf C}(x)\|& \leq C_{d}\sum_{j=0}^l\|\wedge^jd\sigma_\mathbf C(x)\|\|d_x\sigma_\mathbf C-d_xP_{\mathbf C}\|^{l-j},\\
& \leq  dC_{d}\max_{0\leq l\leq k}\|\wedge^{l} d\sigma \|.
\end{align*} 

Therefore we have  for some constants $C_{r,d}$ which may change at each step 
\begin{align*}\mathrm{Cov}(\mathrm{Im}(\sigma_\mathbf C),1)&\leq 2^d \mathrm{Cov}(\mathrm{Im}(P_\mathbf C),2)\\
&\leq C_{r,d} \max_{0\leq l\leq k}\max\|\wedge^{l} dP_\mathbf C\|\\ 
&\leq  C_{r,d} \max_{0\leq l\leq k}\max\|\wedge^{l} d\sigma_\mathbf C\|.
\end{align*} 
Thus,
$$\mathrm{Cov}(\mathrm{Im}(\sigma), 1)\leq  C_{r,d}\|d^r\sigma\|^{k/r} \max_{l\leq k}\|\wedge^{l} d\sigma_\mathbf C. $$

\end{proof}

\end{appendix}

\bibliographystyle{abbrv}
\bibliography{Refs}

\Addresses

\end{document}